\documentclass[a4paper,11pt,leqno]{article}

\usepackage[latin1]{inputenc}
\usepackage[T1]{fontenc} 
\usepackage[english]{babel}
\usepackage{verbatim}
\usepackage{calligra}
\usepackage{enumerate}
\usepackage{dsfont}
\usepackage{amsfonts}
\usepackage{amsmath}
\usepackage{amsthm}
\usepackage{amssymb}
\usepackage{mathtools}
\usepackage{mathrsfs}
\usepackage{color}
\usepackage{graphicx}
\usepackage{bm}
\usepackage{srcltx}

\usepackage{amsmath,amsfonts,amssymb,amsthm}
\usepackage{hyperref}
\usepackage{multirow}

\usepackage{tikz}
\usepackage[retainorgcmds]{IEEEtrantools}

\usepackage{graphicx}		
\usepackage[hypcap=false]{caption}

\providecommand{\bbR}{\mathbb{R}}

\providecommand{\eps}{\varepsilon}

\def\longrightharpoonup{\relbar\joinrel\rightharpoonup}
\def\cv{\longrightharpoonup}
\def\cvwstar{\stackrel{*}{\longrightharpoonup}}

\renewcommand{\leq}{\leqslant}
\renewcommand{\geq}{\geqslant}

\renewcommand{\div}{\operatorname{div}}

\theoremstyle{definition}
\newtheorem{defin}{Definition}[section]

\newtheorem{rmk}[defin]{Remark}

\theoremstyle{plane}
\newtheorem{thm}[defin]{Theorem}
\newtheorem{prop}[defin]{Proposition}
\newtheorem{cor}[defin]{Corollary}
\newtheorem{lemma}[defin]{Lemma}

\newcommand{\tit}{\textit}
\newcommand{\tsc}{\textsc}
\newcommand{\tsl}{\textsl}

\newcommand{\mbb}{\mathbb}

\newcommand{\mc}{\mathcal}
\newcommand{\mf}{\mathfrak}
\newcommand{\mds}{\mathds}

\newcommand{\veps}{\varepsilon}
\newcommand{\what}{\widehat}
\newcommand{\wtilde}{\widetilde}
\newcommand{\vphi}{\varphi}
\newcommand{\oline}{\overline}

\newcommand{\ra}{\rightarrow}

\newcommand{\g}{\gamma}
\renewcommand{\k}{\kappa}
\newcommand{\s}{\sigma}
\renewcommand{\t}{\tau}
\newcommand{\z}{\zeta}
\newcommand{\lam}{\lambda}
\newcommand{\de}{\delta}
\renewcommand{\o}{\omega}

\newcommand{\lan}{\langle}
\newcommand{\ran}{\rangle}

\newcommand{\R}{\mathbb{R}}

\newcommand{\N}{\mathbb{N}}
\newcommand{\Z}{\mathbb{Z}}
\newcommand{\T}{\mathbb{T}}
\renewcommand{\P}{\mathbb{P}}
\newcommand{\D}{\mathbb{D}}

\renewcommand{\div}{{\rm div}\,}
\newcommand{\divh}{{\rm div}_h}
\newcommand{\curlh}{{\rm curl}_h}
\newcommand{\curl}{{\rm curl}\,}

\newcommand{\Supp}{{\rm Supp}\,}

\newcommand{\dx}{ \, {\rm d} x}
\newcommand{\dt}{ \, {\rm d} t}

\newcommand{\al}{\alpha}
\newcommand{\bt}{\beta}

\allowdisplaybreaks

\def\d{\partial}
\def\div{{\rm div}\,}
\newcommand{\dd}{{\rm d}}

\begin{document}

\newcommand{\fra}[1]{\textcolor{blue}{#1}}

\title{\textsc{\Large{\textbf{
Fast rotating non-homogeneous fluids in thin domains and the Ekman pumping effect}}}}

\author{\normalsize\textsl{Marco Bravin}$\,^{1,}\footnote{Present address: \tit{Delft Institute of Applied Mathematics},
\tsc{Delft University of Technology} -- Mekelweg 4, 2628 CD Delft, THE NETHERLANDS.}\qquad$ and $\qquad$
\textsl{Francesco Fanelli}$\,^{2}$ \vspace{.5cm} \\
\footnotesize{$\,^{1,2}\;$ \textsc{Universit\'e de Lyon, Universit\'e Claude Bernard Lyon 1}}  \vspace{.1cm} \\
{\footnotesize \it Institut Camille Jordan -- UMR 5208}  \vspace{.1cm}\\
{\footnotesize 43 blvd. du 11 novembre 1918, F-69622 Villeurbanne cedex, FRANCE} \vspace{.3cm} \\
\footnotesize{Email addresses: $\,^{1}\;$\ttfamily{bravin@math.univ-lyon1.fr}}, $\;$
\footnotesize{$\,^{2}\;$\ttfamily{fanelli@math.univ-lyon1.fr}}
\vspace{.2cm}
}

\date\today

\maketitle

\subsubsection*{Abstract}
{\footnotesize In this paper, we perform the fast rotation limit $\veps\ra0^+$ of the density-dependent incompressible Navier-Stokes-Coriolis system in a thin strip
$\Omega_\veps\,:=\,\R^2\times\,]-\ell_\veps,\ell_\veps[\,$, where $\veps\in\,]0,1]$ is the size of the Rossby number and $\ell_\veps>0$ for any $\veps>0$.
By letting $\ell_\veps\longrightarrow0^+$ for $\veps\ra0^+$ and considering Navier-slip boundary conditions at the boundary of $\Omega_\veps$,
we give a rigorous justification of the phenomenon of the Ekman pumping in the context of non-homogeneous fluids.
With respect to previous studies (performed for flows of contant density and for compressible fluids),
our approach has the advantage of circumventing the complicated analysis of boundary layers.

To the best of our knowledge, this is the first study dealing with the asymptotic analysis of fast rotating incompressible fluids with variable density
in a $3$-D setting. In this respect, we remark that
the case $\ell_\veps\geq\ell>0$ for all $\veps>0$ remains largely open at present.

}

\paragraph*{\small 2020 Mathematics Subject Classification:}{\footnotesize 35B25 
(primary);
35B40 
76U05, 
35Q86 
(secondary).}

\paragraph*{\small Keywords: }{\footnotesize Navier-Stokes-Coriolis system; variable density; thin domain; low Rossby number; Ekman pumping.}


\section{Introduction} \label{s:intro}

This paper is devoted to the study of the dynamics of non-homogeneous viscous incompressible fluids whose motion is primarily influenced by the action
of a strong Coriolis force.
The main application we have in mind is the description of the motion of currents in the ocean (see \cite{CR-B}, \cite{Vallis}),
but other physical phenomena actually display similar features.

To begin with, let us present the equations we are going to consider and describe precisely the problem we want to tackle.

\subsection{Formulation of the problem} \label{ss:formulation}

We assume that the viscous incompressible non-homogeneous fluid occupies the space domain
\begin{equation} \label{eq:Omega_e}
\Omega_\veps\,:=\,\R^2\times\,]-\ell_\veps,\ell_\veps[\,, 
\end{equation}
where $\veps\in\,]0,1]$ and the sequence $\big(\ell_\veps\big)_{\veps\in\,]0,1]}$ is a decreasing sequence of strictly positive real numbers,
$\ell_\veps>0$ for all $\veps>0$, such that
\begin{equation} \label{eq:ell}
 \ell_\veps\,
 \searrow\,0^+\qquad\qquad\qquad \mbox{ for }\qquad \veps\ra0^+\,.
\end{equation}
Here we speak about sequences, but this is not unfair if one keeps in mind the choice $\veps_n\,=\,1/n$ for $n\in\N\setminus\{0\}$, which we will always
tacitly assume in this work.

Then, for any $\veps\in\,]0,1]$, the motion of the fluid is described by the following system of PDEs,
%
\begin{equation} \label{eps:sys}
\left\{ \begin{array}{l}
\partial_t \rho_{\eps}\, +\, \div\big(\rho_{\eps}\, u_{\eps}\big)\,= \, 0 \\ 
\partial_t\big(\rho_{\eps}\, u_{\eps}\big)\, +\, \div\big(\rho_{\eps}\, u_{\eps} \otimes u_{\eps}\big)\, -\, \Delta u_{\eps}\, +\,
\dfrac{1}{\eps}\,\nabla \pi_{\eps}\, +\, \dfrac{1}{\eps}\,e_{3} \times \rho_{\eps}\,u_{\eps}\, = \, 0 \\ 
\div\big(u_{\eps}\big)\,= \, 0\,, 
        \end{array}
\right.
\end{equation}
which we set in the time-space cylinder
\[
\R_+\times\Omega_\veps\,.
\]

In the previous system, the unknowns are the scalar functions $\rho_\veps\geq0$ and $\pi_\veps$ and the vector field
$u_\veps\in\R^3$, representing respectively the density of the fluid, its pressure and its velocity field.
They are functions of the time and space variables $(t,x)\,\in\,\R_+\times\Omega_\veps$.

The first and second equations appearing in \eqref{eps:sys} represent, respectively, the physical principles of conservation of mass and of linear momentum.
The fluid is assumed to be incompressible, a condition which is translated by the last equation in
\eqref{eps:sys}; correspondingly, the gradient term $\nabla\pi_\veps$ appearing in the momentum equation can be interpreted as a
Lagrangian multiplier enforcing at any time the divergence-free constraint over $u_\veps$.
Finally, the term $e_3\times \rho_\veps\,u_\veps$ encodes the effects of the Coriolis force
on the dynamics of the fluid; here, $e_3$ denotes the unit vector directed along the vertical direction and the symbol $\times$ denotes
the usual external product of vector fields in $\R^3$.

In system \eqref{eps:sys}, for any $\veps\in\,]0,1]$ we consider initial conditions
\[
\big(\rho_\veps\,,\,\rho_\veps\,u_\veps\big)_{|t=0}\,=\,\big(\rho^{in}_\veps\,,\,m^{in}_\veps\big)\,,
\]
for suitable functions $\rho^{in}_\veps\geq0$ and $m^{in}_\veps\in\R^3$ which will be better specified below.

In order to close the system, we need suitable boundary conditions. In the present work, we decide to work
with \emph{Navier-slip boundary conditions} (sometimes also called \emph{Robin boundary conditions}).
Before introducing them, let us fix a sequence of real numbers
\[
\big(\alpha_\veps\big)_{\veps\in\,]0,1]}\,\subset\,\R_+\,. 
\]
More precise assumptions on that sequence will be given in the statement of our results.
Next, we observe that
\[
\d\Omega_\veps\,=\,\R^2\times\big\{\pm\,\ell_\veps\big\}\,,
\]
in particular its exterior normal $n_\veps$ is always equal to $n\,=\,\pm e_3$ (depending whether we are on the upper boundary or lower boundary
of the slab $\Omega_\veps$) and does not depend on $\veps$.
Then, on $\d\Omega_\veps$ we impose
\begin{equation} \label{eq:bc}
\big(u_\veps\cdot n\big)_{|\d\Omega_\veps}\,=\,0\qquad\qquad \mbox{ and }\qquad\qquad
\big((\D u_\veps)n\times n\big)_{|\d\Omega_\veps}\,=\,-\,\alpha_\veps\,\big(u_\veps\times n\big)_{|\d\Omega_\veps}\,,
\end{equation}
where $\D u\,:=\,\big(Du\,+\,\nabla u\big)/2$ denotes the symmetric part of the Jacobian matrix $Du$ of the velocity field,
$\nabla u=(Du)^t$ being its transpose matrix.

\medbreak

Our main goal here is to study the fast-rotation limit for the above system \eqref{eps:sys}.
This means that, taken a sequence $\big(\rho_\veps,u_\veps\big)_{\veps\in\,]0,1]}$ of solutions\footnote{As a matter of fact, typically the pressure gradient
$\nabla\pi_\veps$ can be recovered from $\rho_\veps$ and $u_\veps$ by solving an elliptic equation.} to \eqref{eps:sys},
we aim at describing their asymptotic dynamics in the limit for $\veps\ra0^+$: more precisely,
we are interested in proving their convergence to some limit point $\big(\rho,u\big)$, and in identifying the evolution
equations this couple satisfies.

Before presenting our main results (for which we refer to Subsection \ref{sss:res_over} below), let us give an overview of previous works related to our problem,
which also serves to explain the main motivations for our investigation.
The study of singular perturbations in fluid mechanics being an old topic, with a broad literature devoted to it in various contexts, we
will limit ourselves to mention only the works which are directly related to ours, as well as others which are less directly related, but are still
functional to understand the contents of this paper.

\subsection{Related studies: incompressible \tsl{vs} compressible flows} \label{sss:related}

The first fact one can remark 
is that, if $\rho_{\veps}^{in}\equiv1$, then $\rho_\veps\equiv1$ for all later times, and then system \eqref{eps:sys} reduces to the classical
incompressible Navier-Stokes system with Coriolis force. In that context, the effects of a fast rotation on the asymptotic behaviour of solutions and on
their lifespan have been largely investigated. We refer \tsl{e.g.} to the book \cite{C-D-G-G} for a compendium of the results available in that situation.

The case of non-homogeneous fluids, \tsl{i.e.} fluids for which the density is non-constant, have been considered afterwards, but essentially in the framework
of compressible flows, namely when the divergence-free condition on $u_\veps$ is dropped in \eqref{eps:sys}.
Many works have been performed about the asymptotic behaviour of solutions in the physically relevant case of combined fast rotation (low Rossby number) limit
and incompressible (low Mach number) limit, even in presence of external forces, like gravity and centrifugal force. More recently,
the case of the full Navier-Stokes-Fourier system was considered, where heat transfer processes are taken into account.
Interestingly, the compressible flows framework is well-suited for a multi-scale analysis, where the Rossby and Mach numbers (and possibly the Froude number,
measuring stratification effects) are assumed to be small, but with different orders of magnitude (\tsl{i.e.} different powers of the small parameter $\veps>0$).
We refer to \tsl{e.g.} \cite{F_2019} and \cite{DS-F-S-WK} for an overview of the available results, further references and recent developments on the subject.

\medbreak
Observe that, for density-dependent fluids, the Coriolis term $e_3\times \rho_\veps\,u_\veps$ is skew-symmetric with respect to the $L^2$ scalar product,
but it is no more skew-symmetric with respect the $H^s$ scalar product when $s>0$, unless some smallness condition is imposed on the densities.
Because of this, in order to study the fast rotation limit for density-dependent fluids, it is natural to work in the setting of finite energy weak solutions:
this is the case of \cite{F_2019}, \cite{DS-F-S-WK} and essentially all the works mentioned therein. The present paper will do no exception to this.

Correspondingly, in order to treat the singular perturbation problem, one needs to consider initial densities which are close to
stationary solutions of the equations, namely $\rho_{\veps}^{in}\,=\,\wtilde\rho\,+\,\veps\,r_\veps^{in}$, where the reference density
$\wtilde\rho$ 
solves the corresponding system with $u_\veps\equiv0$ and the sequence of density variations $r_\veps^{in}$ is bounded in suitable norms.
Thus, in the compressible setting, the precise form of $\wtilde\rho$ is completely identified (still, not in a unique way, in general) by the pressure
term (which is now given; we will come back to this in a while) and the presence of external forces. For instance, in absence of gravitational and centrifugal forces,
one has $\wtilde \rho\equiv1$.

Finally, there is another point which deserves special attention in all this: as already mentioned, in the compressible case
the pressure term $\nabla\pi_\veps$ is no more an unknown of the problem.
For barotropic flows, for instance, the pressure $\pi_\veps=\pi(\rho_\veps)$ is a known function of the density; in the heat-conducting case, instead
$\pi$ also depends on the temperature function. Then, even though the analysis of the pressure term is technically involved
(especially when working with weak solutions), two great advantages appear. First of all, the structure of the problem allows one to prove the same density
decomposition as above also at any later time: one has $\rho_\veps\,=\,\wtilde\rho\,+\,\veps\,r_\veps$, where each $r_\veps$ represents, roughly speaking, the
evolution of the initial density perturbations $r_\veps^{in}$ by the flow\footnote{Here we allow ourselves to improperly speak about the flow associated to the solution
although the problem is not known to be well-posed in the context of weak solutions, just to give the reader a simple picture to retain.}
of the solution.
In addition, the pressure gradient immediately gives an information on the target density variation $r\,:=\,\lim_{\veps\ra0^+} r_\veps$
and, in some special cases (namely, when the Rossby and the Mach numbers are assumed to be of the same order of magnitude), it allows to derive a fundamental
relation linking $r$ and the target velocity field.

\subsection{Previous results for incompressible density-dependent fluids} \label{sss:dd-incompr}

In the incompressible case with variable density, as considered in \eqref{eps:sys}, we notice that the coupling between the mass equation
and the momentum equation becomes weaker, inasmuch as the pressure term $\nabla\pi_\veps$ is only a Lagrangian multiplier associated to
the divergence-free condition over $u_\veps$.

This fact entails several consequences, which make the study for incompressible density-dependent fluids much more involved than in the compressible instance.
First of all, even in absence of external forces, one may consider several reference density states $\wtilde\rho$: certainly the case $\wtilde\rho\equiv1$ is
one choice (this corresponds to what we will call \emph{quasi-homogeneous} case), but one may consider also the case when $\wtilde\rho$ is truly variable
(what we will call \emph{fully non-homogeneous} case).
Moreover, the pressure gradient does not give any direct information on the density functions, not even at the limit.
The consequence of this is that, for three-dimensional flows, in the fully non-homogeneous case one misses fundamental properties when performing the fast rotation
limit, which prevent one from deriving what is called the \emph{Taylor-Proudman theorem} in geophysics\footnote{This theorem says that the flow which undergoes
the action of a strong Coriolis force looks like two-dimensional (the motion essentially taking place on a plane orthogonal to the rotation axis)
and is spolied of any vertical structure. We refer to \tsl{e.g.} \cite{CR-B}, \cite{C-D-G-G} and \cite{Vallis} for more insights both from the physical
and mathematical viewpoints.}.

This is the main reason why the study of the fast rotation limit of non-homogeneous incompressible fluid equations had not received attentions for a long time.
To the best of our knowledge, the first work treating this problem is \cite{F:G}, where the authors were able to perform the asymptotic study
both in the quasi-homogeneous and in the fully non-homogeneous cases, but only in two space dimensions (the corresponding system is simply
the projection of the full $3$-D system \eqref{eps:sys} onto the horizontal plane). Of course, in that instance the derivation of the Taylor-Proudman theorem
was no more an issue (and this is the fundamental simplification that the $2$-D geometry entails), although several complications
still appeared, at least for non-constant reference densities
$\wtilde\rho$. For instance, here it is no more clear that having $\rho^{in}_\veps\,=\,\wtilde\rho\,+\,\veps\,r^{in}_\veps$
implies that the same decomposition holds true for any later time, because $\wtilde\rho$ is no more transported by the flow and, differently from the compressible case,
one cannot rely anymore on the pressure term to guarantee that property.
The key point of the analysis was to use a fundamental boundedness-compactness property hidden in the wave system, \tsl{i.e.} the system which describes
propagation of fast oscillating waves (the so-called Rossby waves) associated to the singular part of the equations. That property
allowed one to prove that the claimed decomposition is in fact true, although in a very weak sense. This implied a series of consequences which, in turn,
allowed to pass to the limit. Owing to the mild information at one's disposal, however, the limit dynamics which was identified was underdetermined,
as only one equation was available for a combination of two target quantities, namely the vorticity of the limit flow and the limit $r$ of the density perturbations.

To conclude this part, we mention that the study of \cite{F:G} has recently been generalised to the case of the MHD system with Coriolis force 
in \cite{C-F}. Let us also quote work \cite{Sbaiz}, devoted to the inviscid case (the study therein was performed only in the quasi-homogeneous regime,
though, because one was obliged to deal with regular solutions).

\subsection{Overview of the main results} \label{sss:res_over}

In the present paper, we consider the fast rotation limit for density-dependent incompressible Navier-Stokes system \eqref{eps:sys}
in the infinite slab $\Omega_\veps$ defined in \eqref{eq:Omega_e}. 
The main goal of our study is twofold.

On the one hand, we aim at giving a better understanding of the general three-dimensional setting.
Because of the difficulties mentioned in the previous subsection, it is natural to consider the problem in a framework which allows to
recover the properties set forth by the Taylor-Proudman theorem in the limit. This is why we impose the narrowness condition \eqref{eq:ell} on the size
$\ell_\veps$ of each slab $\Omega_\veps$.
Notice that the general case $\ell_\veps\geq\ell>0$ remains largely open at present.

On the other hand, by imposing the Navier-slip boundary conditions \eqref{eq:bc} at the boundary $\d\Omega_\veps$, we aim at giving a new derivation
of the \emph{Ekman pumping} effect in geophysics. As a direct consequence of the Taylor-Proudman theorem, it is well-known that, in a small region close
to the boundary of the domain (region called ``Ekman layer''), the flow must be slowed down for passing from an essentially horizontal configuration
to being at rest at the boundary. Because of this, a global circulation phenomenon, called \emph{Ekman suction}, is created, which involves small 
portions of the fluid in the boundary layer and other ones in the interior of the domain.
Even though it concerns a small amount of the fluid, the Ekman suction phenomenon has an important effect in the energy balance, as it is
responsible for dissipation of kinetic energy in the boundary layer by damping the inner motion. This effect is precisely what is called Ekman pumping.
At the mathematical level,  the Ekman pumping effect is encoded by the appearing of a damping term in the limit momentum equation, and is usually derived
by a complicated analysis of the Ekman boundary layer. We refer for instance to \cite{Gr-Masm}, \cite{Masm} and \cite{C-D-G-G} for the case of homogenous fluids,
to \cite{B-D-GV} and \cite{B-F-P} for the analysis for compressible flows. Notice that, in that analysis, one has to make a smallness assumption
on the size of the vertical diffusion. One is thus led to introduce an anisotropy in the Lam\'e operator (viscous stress tensor),
which is a source of troubles in the compressible case: namely, existence of finite energy weak solutions in that case still remains largely open
at present (see for instance \cite{B-J}), and one has to postulate their existence and uniform energy bounds on the small parameter $\veps>0$,
without having a theory guranteeing that.

\medbreak
As already mentioned, we give here a new derivation of the Ekman pumping phenomenon for non-homogeneous fluids,
by avoiding the complicated analysis of Ekman layers. We decide to do that for the incompressible system \eqref{eps:sys}, although a similar analysis
would apply also to its compressible counterpart. Notice that fast rotating compressible flows in thin domains have already been considered in the literature,
see \tsl{e.g.} \cite{D-C-N-P}, but, to the best of our knowledge, only in the case of complete slip boundary conditions.
We also point out that, in our investigation, the horizontal part $\R^2$ of the domain $\Omega_\veps$ could be replaced by the two-dimensional
torus $\T^2$ with essentially no changes in the analysis (which actually would become slightly simpler).

Thus, we will work on system \eqref{eps:sys}, supplemented with Navier-slip boundary conditions \eqref{eq:bc}.
We will perform our asymptotic study $\veps\ra0^+$ in the context of global in time \emph{finite energy weak solutions}
and for general \emph{ill-prepared} initial data. We will consider both settings: the quasi-homogeneous one (where the reference density
state $\wtilde\rho$ is constant, say $1$) and the fully non-homogeneous one (in which $\wtilde\rho$ is truly variable).

In passing, we observe that, the Navier-slip boundary conditions being not so common in the mathematical literature, 
in the present paper we also establish the existence of global in time finite energy weak solutions (solutions \tsl{\`a la Leray}) for our model.

Coming back to the asymptotic study, we show that the only relevant regime to consider is the one in which
\[
\lim_{\veps\ra0^+}\frac{\alpha_\veps}{\ell_\veps}\,=\,\lam\in[0,+\infty[\,.
\]
Indeed, in the case $\lam=+\infty$, the limit dynamics is trivial (the limit velocity field $u$ is just $0$ in that case).
When $\lam=0$, in the limit $\veps\ra0^+$ we recover the same limiting systems identified in the purely two-dimensional case (see \cite{F:G} and \cite{C-F}),
whereas if $\lam>0$ an additional term appears, corresponding exactly to the Ekman pumping phenomenon.

The core of the proof essentially consists of two main steps. In the fist one, we prove precise quantitative estimates, which allow us to deduce on the one hand
that the motion is progressively spoiled of its vertical component when $\veps\ra0^+$, on the other hand that the solution $\big(\rho_\veps,u_\veps\big)$ tends to
be independent of the vertical variable. This is exactly the Taylor-Proudman theorem in our context.
The consequence of all this is that only vertical averages of the solutions matter to understand the limit
dynamics. With this idea in mind, thanks to the quantitative estimates mentioned above, we can somehow linearise the non-linear terms, in the sense that, for understanding
the limit of the vertical average of a product (for instance, this is the case of the convective term), we can work on the product of the vertical averages.
Then, and this is the second main step of the proof, we can adapt the main ingredients of the analysis of the planar case \cite{F:G}
to prove the convergence of the vertical averages towards the target system.

We conclude this part by mentioning that, as it was already the case in \cite{F:G} and \cite{C-F},
the limit dynamics is well-identified only in the quasi-homogeneous regime, while
it remains underdetermined in the fully non-homogeneous case. On the other hand, we will somehow simplify the arguments of those works.
For instance, the proof of the convergence in the quasi-homogeneous case does not rely on the analysis of the system of Rossby
waves, in particular we need not to apply a regularisation procedure of the solutions; instead, we will rather prove that the vertical average of
the horizontal component of the velocity fields locally strongly converges to the target velocity profile by a direct analysis of the momentum equation.

\subsubsection*{Structure of the paper}

To conclude, we present a short outlook of the remaining part of the paper.

In the next section, we fix the value of the parameter $\veps\in\,]0,1]$ and we prove the existence of a finite energy weak solution to system
\eqref{eps:sys}-\eqref{eq:bc}, which is defined globally in time. The arguments follow the main lines of the ones used for classical
no-slip or complete slip boundary conditions, but, in absence of a precise reference, we decided to give most of the details.

In Section \ref{s:fast}, we present the main results concerning the fast rotation asymptotics $\veps\ra0^+$ of system \eqref{eps:sys}-\eqref{eq:bc}.
The proof of the results is carried out in Sections \ref{s:bounds} and \ref{s:proof-main}.
In the former section, we collect uniform estimates for the family of weak solutions $\big(\rho_\veps,u_\veps\big)_\veps$ and basic consequences
of those bounds. We will also be able to prove the statements concerning the degenerate case $\lam=+\infty$ and the quasi-homogeneous limit $\wtilde\rho=1$.
The latter section, instead, is devoted to the proof of the convergence in the most general situation of the fully non-homogeneous regime,
which is also the hardest case to handle.

\subsection*{Acknowledgements}

{\small
The work of both authors has been partially supported by the project CRISIS (ANR-20-CE40-0020-01), operated by the French National Research Agency (ANR).

The work of the first author has been partially supported by the NWO grant OCENW.M20.194.

The work of the second author has been partially supported by the LABEX MILYON (ANR-10-LABX-0070) of Universit\'e de Lyon, within the
program ``Investissement d'Avenir'' (ANR-11-IDEX-0007), and by the projects BORDS (ANR-16-CE40-0027-01) and SingFlows (ANR-18-CE40-0027),
all operated by the French National Research Agency (ANR).

}


\section{Existence of finite energy weak solutions}

In order to perform our programme about the asymptotic behaviour of system \eqref{eps:sys}, we first need a result guaranteeing us
the existence, for any value of the parameter $\veps\in\,]0,1]$, of finite energy weak solutions, in the same spirit of
Leray's solutions to the classical (\tsl{i.e.} with $\rho\equiv1$) incompressible Navier-Stokes system.

Of course, some results about existence of Leray-type weak solutions for system \eqref{eps:sys} are available in the literature,
but, to the best of our knowledge, mainly in the case where \emph{no-slip boundary conditions} are imposed (namely, the condition $u=0$ is assumed
on the boundary of the domain). This has been done \tsl{e.g.} in \cite{Kaz} under the assumption
that the initial density satisfy $\rho^{in}\geq\rho_*>0$, a condition which was relaxed afterwards in \cite{Simon} to $\rho^{in}\geq0$.
The theory was later extended in \cite{lions:book} to allow the viscosity coefficient to depend on the density.

On the contrary, the existence of Leray-type weak solutions to \eqref{eps:sys} in the case of Navier-slip boundary conditions \eqref{eq:bc},
as considered in this work, seems not having received too much attention in the past. Proving such an existence result
is the goal of the present section.

Here, we are going to work at any $\veps\in\,]0,1]$ fixed. Since its precise value, as well as the values of the
other parameters $\ell_\veps$ and $\alpha_\veps$, does not play any role, without loss of generality
we set
\begin{equation} \label{eq:fixed-param}
 \veps\,=\,1\,,\qquad\qquad \ell_\veps\,=\,1\,,\qquad\qquad \alpha_\veps\,=\,\alpha\,\geq\,0
\end{equation}
in system \eqref{eps:sys}-\eqref{eq:bc}. Throughout all this section, we will implicitly work with that choice of the parameters; correspondingly,
we will work in the domain
\[
 \Omega\,=\,\R^2\times\,]-1,1[\,.
\]

To begin with, let us introduce some important functional spaces. We define
\begin{align*}
L^2_{\sigma}(\Omega;\R^3)\,&:=\,
\Big\{ u \in L^2(\Omega; \mathbb{R}^3)\; \Big|\quad \div(u)\,=\,0\quad \mbox{ and }\quad (u\cdot n)_{|\d\Omega}\,=\,0\ \Big\} \\
H^1_{0,\sigma}(\Omega;\R^3)\,&:=\,L^2_{\sigma}(\Omega;\R^3)\,\cap\, H^1(\Omega;\R^3)\,.
\end{align*} 
We remark that (see Chapter III of \cite{Galdi} and Chapter 1 of \cite{Tsai} for details) $L^2_\s(\Omega;\R^3)$
coincides with the closure of $C^\infty_{c,\s}(\Omega)$,
the space of smooth compactly supported functions on $\Omega$ having null divergence, with respect to the
$L^2$ norm.
Moreover for a subset $ A \subset \Omega $, we denote by $  \mds{1}_A $ the characteristic function of $ A $,
more precisely the function $ \mds{1}_A : \Omega  \to \{0,1\} $ which takes value $ 1 $ for $ x \in A $ and $ 0 $ otherwise.

Next, we introduce the definition of Leray-type weak solutions to system \eqref{eps:sys}-\eqref{eq:bc}-\eqref{eq:fixed-param}.

\begin{defin}
\label{def:weak:sol}
Let $ \rho^{in} \in L^{\infty}(\Omega)$ be a function such that $ \rho^{in} \geq 0 $ and
\begin{equation} \label{eq:non-deg}
	\exists\,\de>0\qquad \mbox{ such that }\qquad\qquad\qquad 
	\frac{1}{\rho^{in}}\,\mds{1}_{\{\rho^{in}\leq\de\}}\,\in\,L^1(\Omega)\,.
\end{equation}
Let $ m^{in}\in L^2(\Omega;\R^3) $ be a vector field such that
$ m^{in} \equiv 0 $ on the set $\big\{ x\in\Omega\,\big|\; \rho^{in}(x) = 0 \big\}$ and $ \big|m^{in}\big|^2/ \rho^{in} \in L^1(\Omega) $.

Then, given $T>0$, we say that a couple $ (\rho, u) $ is a \emph{finite energy weak solution} of system \eqref{eps:sys}-\eqref{eq:bc}-\eqref{eq:fixed-param}
on $[0,T[\,\times\Omega$, related to the initial
datum $\big(\rho,\rho u\big)_{|t=0}\,=\,\big(\rho^{in},m^{in}\big)$, if the following conditions are satisfied:
\begin{enumerate}[(i)]
 \item $\rho\in L^\infty\big([0,T[\,\times\Omega\big)$, with $\rho\geq0$ almost everywhere in $[0,T[\,\times\Omega$;
 \item $u\in L^\infty\big([0,T[\,;L^2_\s(\Omega;\R^3)\big)\,\cap\,L^2\big([0,T[\,;H^1_{0,\s}(\Omega;\R^3)\big)$;
 \item $\rho$ solves the first equation of \eqref{eps:sys} in a distributional and renormalised sense: for any $\bt\in C^1_b(\R)$ and any
$\vphi\in C^\infty_c\big(\R_+\times\oline{\Omega}\big)$, one has
\[ 
-\iint_{\R_+\times\Omega} \bt(\rho)\,\big(\d_t\vphi\,+\,u\cdot\nabla\vphi\big)\,\dx\dt\,=\,
\int_\Omega\beta(\rho^{in})\,\varphi(0,\cdot)\,\dx \,;
\] 
\item $u$ solves the second equation of \eqref{eps:sys} in the distributional sense: for any
$ \psi \in C^{\infty}_{c}\big(\R_+\times\oline{\Omega};\R^3\big)$ such that $ \div (\psi) = 0 $ and $\psi\cdot n=0$ on $\d\Omega$, one has
\begin{align}
&\iint_{\R_+\times\Omega}\Big(-\rho\,u\cdot\d_t\psi\,-\,\rho\,u\otimes u:\nabla\psi\,+\,e_3\times\rho\,u\cdot\psi\Big)\dx\dt
\label{equ:weak:eps} \\
&\qquad\qquad\quad
\,+\,\iint_{\R_+\times\Omega}\nabla u:\nabla\psi\,\dx\dt +\,2\,\alpha\iint_{\R_+\times\d\Omega}u\cdot\psi\,\dx_h\dt\,=\,
\int_\Omega m^{in}\cdot\psi\,\dx\,, \nonumber
\end{align}
where, for two matrices $A$ and $B$ of size $k\times k$, we have denoted by $A:B\,:=\,{\rm tr}\big(A\cdot B^t\big)$
their Forbenius product;
\item the following \emph{energy inequality} holds true: for amost any $t\in[0,T[\,$, one has
\[ 
\hspace{-1cm}
\frac{1}{2}\int_{\Omega} \rho(t,\cdot)\,|u(t,\cdot)|^2\,\dx\,+\,\iint_{[0,t]\times\Omega}\big|Du\big|^2\,\dx\,\dd\t\,+\,
2\,\alpha\,\iint_{[0,t]\times\d\Omega}|u|^2\,\dx_h\,\dd\t\,\leq\, \int_{\Omega} \frac{|m^{in}|^2}{\rho^{in}}\,\dx\,.
\] 
\end{enumerate}
The solution is said \emph{global in time} if the previous conditions hold for any $T>0$.
\end{defin}

Let us remark that the meaning of satisfying the initial condition of the momentum equation is non trivial, but well understood by now; see for instance
the comments preceding Lemma 2.1 of \cite{lions:book} and Theorem 2.2 of the same book.

The main result of the present section reads as follows.

\begin{thm}
\label{exi:theo}
Let $ \rho^{in} \in L^{\infty}(\Omega)$ be a function such that $ \rho^{in} \geq 0 $ and satisfying condition \eqref{eq:non-deg}.
Let $ m^{in}\in L^2(\Omega;\R^3) $ be a vector field such that
$ m^{in} \equiv 0 $ on the set $\big\{ x\in\Omega\,\big|\; \rho^{in}(x) = 0 \big\}$ and $ \big|m^{in}\big|^2/ \rho^{in} \in L^1(\Omega) $.

Then, there exists a global in time finite energy weak solution $(\rho, u)$ of system \eqref{eps:sys}-\eqref{eq:bc}-\eqref{eq:fixed-param}
related to the initial condition $ \big( \rho^{in}, u^{in} \big) $, in the sense of Definition \ref{def:weak:sol}.
\end{thm}

The proof follows the ones in \cite{Simon} and \cite{lions:book} for no-slip conditions, with some adjustments.
For sake of completeness, here we will briefly give most of the details of the main changes which have to be performed.
First of all, we are going to prove Theorem \ref{exi:theo} for a smooth bounded domain $Q\subset\R^3$, and then extend the result to the unbounded
case $\Omega=\R^2\times\,]-1,1[\,$ which is of interested for us, by the technique of invading domains.

Thus, for the time being let us work on some smooth bounded domain $Q\subset\R^3$. 
We are going to prove the following result, which is the adaptation of Theorem \ref{exi:theo} to the domain $Q$.
\begin{prop} \label{p:existence}
Let $Q$ be a bounded smooth domain of $\R^3$.
Let $ \rho^{in} \in L^{\infty}(Q)$ be a function such that $ \rho^{in} \geq 0 $ and satisfying condition \eqref{eq:non-deg}.
Let $ m^{in}\in L^2(Q;\R^3) $ be a vector field such that
$ m^{in} \equiv 0 $ on the set $\big\{ x\in Q\,\big|\; \rho^{in}(x) = 0 \big\}$ and $ \big|m^{in}\big|^2/ \rho^{in} \in L^1(Q) $.

Then, there exists a global in time finite energy weak solution $(\rho, u)$ of system \eqref{eps:sys}-\eqref{eq:fixed-param} on $\R_+\times Q$,
supplemented with the boundary conditions \eqref{eq:bc} at $\d Q$ and
related to the initial datum $ \big( \rho^{in}, u^{in} \big) $.
\end{prop}

For proving Proposition \ref{p:existence}, we need a further preliminary reduction. More precisely, we start by considering the problem of existence
of weak solutions for \emph{well-behaved} data, in the sense specified in the assumptions of the next lemma.

\begin{lemma} \label{l:Galerkin}
Let the couple $\big(\rho^{in},m^{in}\big)$ be as in Proposition \ref{p:existence}. Assume moreover that
\[ 
\rho^{in}\in C^1(Q)\,,\qquad\qquad \mbox{ with }\qquad \rho^{in} \geq \rho_* > 0\,,
\] 
for a suitable positive constant $\rho_*\in\R$, and that
\[
 u^{in}\,:=\,\frac{m^{in}}{\rho^{in}}\qquad\qquad \text{ verifies }\qquad\qquad
 \div\big(u^{in}\big) = 0 \quad \text{ and }\quad  \big(u^{in} \cdot n\big)_{|\d Q} = 0\,. 
\]

Then, there exists a global in time finite energy weak solution $(\rho, u)$ of the system \eqref{eps:sys}-\eqref{eq:bc}-\eqref{eq:fixed-param}
over $\R_+\times Q$, related to the initial condition $ ( \rho^{in}, \rho^{in} u^{in} ) $. 
\end{lemma}

\begin{proof}
For proving the previous result, we use a Galerkin method. So, let us choose a countable set
$W:=\big\{w_i\big\}_{i\in\N}\subset \big\{w\in C^\infty(\oline Q)\,|\ \div (w)=0\,,\ \big(w\cdot n\big)_{|\d Q}=0\big\}$
of smooth functions (see Theorem 10.13 of \cite{F-N}), such that $W$ is an orthonormal basis of $ L^2_{\sigma}(Q)$, orthogonal in $H^1_{0,\s}(Q)$.

For $N\in\N$, denote $ W_N\,:=\,{\rm span}\big\{w_1\ldots w_N\big\}$. Following the method of \cite{Simon} (see the proof
of Theorem 9 therein), for any $N\in\N$, we can construct approximate solutions $\big(\rho^N, u^N\big)$, with
\[
\rho^N\in C^1\big([0,T_N];C^1(\oline Q)\big)\qquad\quad \mbox{ and }\qquad\quad  u^N(t,x)\,=\,\sum_{i = 1}^N g^N_i(t)\, w_i(x)\,\in\, C^1\big([0,T_N];W_N\big)\,.
\]
The approximate solutions $\big(\rho^N, u^N\big)$ satisfy the transport equation
\[
\partial_t \rho^N + \div(\rho^N u^N ) = 0\,,\qquad\qquad\quad \rho^N_{|t=0}\,=\,\rho^{in} 
\]
in the sense of $\mc D'\big([0,T_N]\times Q\big)$,
and the momentum equation projected onto $W_N$, still in the weak sense: for any $v\in W_N$, one has
\[ 
\int_{Q} \left(\rho^N \partial_t u^{N} + \rho^N u^{N} \cdot \nabla u^{N} + e_3 \times \rho^N u^N  \right)\cdot v\,\dx \,+\,
\int_{Q} \nabla u^N : \nabla v\,\dx\, +\, 2\,\alpha \int_{\d Q} u^{N} \cdot v\,\dd\mf s\, =\, 0\,, 
\] 
together with the initial condition $u^N_{|t=0}\,=\,\P_Nu^{in}$, where $\dd\mf s$ is the surface measure on $\d Q$ and $\P_N$ is the orthogonal
projection from $L^2_\s(Q)$ onto $W_N$.

Next, as $u^N$ is smooth with respect to the space variable, we can perform energy estimates and derive an energy inequality similar to
the one stated in item (v) of Definition \ref{def:weak:sol} for the couple $\big(\rho^N,u^N\big)$. This allows us to extend the solution
globally in time, \tsl{i.e.} one has $T_N=+\infty$.

In addition, by using also \tsl{a priori} estimates for smooth solutions of transport equations, we obtain that the sequence $\big(\rho^N\big)_N$
is bounded in $L^\infty\big(\R_+\times Q\big)$ and that $\big(u^N\big)_N$ is bounded in
$L^\infty\big(\R_+;L^2_\s(Q)\big)\cap L^2_{\rm loc}\big(\R_+;H^1_{0,\s}(Q)\big)$. Those properties allow us to identify
a couple $\big(\rho,u\big)$ of limit points in the respective functional spaces, to which $\big(\rho^N,u^N\big)$ converges (up to the extraction
of a suitable subsequence) in the weak-$*$ topologies of the respective spaces: for any $T>0$, we have
\begin{itemize}
\item $ \rho^N\,\stackrel{*}{\rightharpoonup}\, \rho $ in $ L^{\infty}\big([0,T];L^\infty(Q)\big)$, 
\item $ u^N\, \rightharpoonup\, u $ in $ L^2\big([0,T];H^1_{0,\sigma}(Q)\big)$.
\end{itemize} 
At this point, we can employ Theorem 2.4  of \cite{lions:book} to pass to the limit $N\ra+\infty$ in the equations
for $\big(\rho^N,u^N\big)$ and deduce that $\big(\rho,u\big)$ is the sought solution.
\end{proof}

Next, let us explain how to approximate a general initial datum $\big(\rho^{in}, m^{in}\big) $ by well-behaved ones,
with, in addition, good enough convergence properties which allow us to pass to the limit in the weak formulation of the equations.
\begin{lemma} \label{l:approx}
Let the couple of initial data $\big(\rho^{in}, m^{in}\big) $ be as in Proposition \ref{p:existence}.

Then, there exists a sequence $\big(\rho^{in}_{\de},m^{in}_\de,u^{in}_\de\big)_{\de>0}$ of approximate data such that:
\begin{enumerate}[(a)]
 \item for any $\de>0$, one has $\rho^{in}_\de\in C^\infty(\oline Q)$ and $\de\leq \rho^{in}_\de\leq \left\|\rho^{in}\right\|_{L^\infty}+\de$;
 \item when $\de\ra0^+$, one has the weak-$*$ convergence $\rho^{in}_\de\,\stackrel{*}{\rightharpoonup}\,\rho^{in}$ in $L^\infty(Q)$ and 
 strong convergence $\rho_\de^{in}\,\longrightarrow\,\rho^{in}$ in $L^p_{\rm loc}(Q)$, for any $1\leq p<+\infty$;
 \item for any $\de>0$, one has $m_\de^{in}\in C^\infty_c(Q;\R^3)$, and one has the strong convergence properties
\[
\frac{m^{in}_{\de}}{\sqrt{\rho^{in}_{\de}}}\,\longrightarrow\, \frac{m^{in}}{\sqrt{\rho^{in}}}\quad \mbox{ in }\ L^2(Q)\qquad\quad \mbox{ and }\qquad\quad
m^{in}_\de\,\longrightarrow\,m^{in}\quad \mbox{ in }\ L^p(Q)\quad \forall\,1\leq p<2\,,
\]
in the limit $\de\ra0^+$;
\item for any $\de>0$, one has $u^{in}_\de\in C^\infty(Q;\R^3)$, with $\div\big(u^{in}_\de\big)=0$ and
\[
\forall\,\vphi\in L^2_\s(Q)\,,\qquad\qquad
\int_{Q} \rho^{in}_{\de}\,u^{in}_{\de} \cdot \varphi\,\dx \, =\, \int_{Q} m^{in}_{\de} \cdot \varphi\,;
\]
\item there exists a constant $C>0$ such that, for any $\de>0$, one has
\[
\int_{Q}\rho^{in}_{\de}\, |u^{in}_{\de}|^2\,\dx\,\leq\,C\,\int_{Q} \frac{\big|m^{in}\big|^2}{\rho^{in}}\,\dx\,.
\]
\end{enumerate}
\end{lemma} 

\begin{proof}
We start by considering the density function: for its approximation, it is enough to define
$ \rho^{in}_{\de} = \theta_{\de} \star \rho^{in} + \de $, where $ \big(\theta_{\de}\big)_{\de>0}$ is symmetric positive convolution kernel of mass $ 1 $
and the symbol $\star$ denotes the convolution with respect to the space variable.
Then, it is plain to see that the claimed properties of $\big(\rho^{in}_\de\big)_{\de>0}$ are satisfied.

Let us now construct an approximation of $ m^{in} $. Since, by assumption, we have $ m^{in} / \sqrt{\rho^{in}} \in L^2(Q)$, there exists
(see \tsl{e.g.} Corollary 4.23 of \cite{Brezis}) a sequence $\big(w_\de\big)_{\de>0}\,\subset\,C^\infty_c(Q)$ such that
$w_\de$ converges strongly to $  m^{in} / \sqrt{\rho^{in}} $ in $ L^2(Q) $. Then, define $ m^{in}_{\de}:=  \sqrt{\rho^{in}_{\de}}\,w_{\de} $.
Notice that, for any $\de>0$, $m^{in}_\de$ is smooth and compactly supported in $Q$. Moreover, using the convergence properties
of $\big(w_\de\big)_{\de>0}$ and $\big(\rho^{in}_\de\big)_{\de>0}$, the convergence properties claimed for $\big(m^{in}_\de\big)_{\de>0}$ are easily seen
to hold true.

At this point, using the fact that $ \rho^{in}_{\de} \geq \de $ and Lemma 2.1 of \cite{lions:book}, for any $\de>0$ fixed, there exists a couple
$\big(u_{\de}^{in},q_\de\big)\in C^\infty(Q;\R^3)\times C^\infty(Q)$ such that
$ m^{in}_{\de}\, =\, \rho^{in}_{\de}\, u^{in}_{\de} \,+\, \nabla q_{\de} $, together with the properties
$ \div\big(u_{\de}^{in}\big) = 0 $ and $\big( u_{\de}^{in} \cdot n\big)_{|\d Q} = 0 $.

The lemma is thus proved.
\end{proof}

We can now prove Proposition \ref{p:existence} in its full generality.

\begin{proof}[Proof of Proposition \ref{p:existence}] 
Given the initial datum $ \big(\rho^{in}, m^{in}\big) $, let us take the approximate sequence $ \big(\rho^{in}_{\de}, m^{in}_{\de}, u^{in}_\de\big)_{\de>0} $
provided by Lemma \ref{l:approx}.

By using Lemma \ref{l:Galerkin}, for any $\de>0$ we can solve system \eqref{eps:sys} in $\R_+\times Q$, with boundary conditions \eqref{eq:bc} on $\d Q$
(recall the choice \eqref{eq:fixed-param} of the parameters), thus producing a global in time weak solution
$\big(\rho_\de,u_\de\big)$. Notice that, owing to the properties claimed in items (d)-(e) of Lemma \ref{l:approx}, the couple $ (\rho_{\de}, u_{\de}) $
is also a weak solution for the initial data $\big(\rho^{in}_{\de}, m^{in}_{\de}\big) $.

In order to complete the proof, we need to show suitable  compactness properties for $ \big(\rho_{\de}, u_{\de}\big) $. This follows from the energy inequality,
combined with a uniform bound of the time derivatives of $\rho_\de$ and $u_\de$, which can be obtained (in a classical way) by using the equations.
Passing to the limit is then a consequence of Theorem 2.4  of \cite{lions:book}.
\end{proof}

With Proposition \ref{p:existence} at hand, we can now prove Theorem \ref{exi:theo}, namely show the existence of global in time finite energy
weak solutions to our system \eqref{eps:sys}-\eqref{eq:bc}-\eqref{eq:fixed-param} in the infinite slab $\Omega$.

\begin{proof}[Proof of Theorem \ref{exi:theo}]
We apply the technique of invading domains. Namely, we write
\[
\Omega\,=\,\bigcup_{n=1}^{+\infty}Q_n\,,
\]
where, for each $n\in\N\setminus\{0\}$, the domain $Q_n\subset\Omega$ is smooth and bounded, and one has $Q_n\subset Q_{n+1}$.
Of course, such a sequence of domains exists: we can take, for instance, cylinders $B_{n+1}\times\,]-1,1[\,$ of radius $n$, where we understand that
the corners are smoothed out (here, the symbol $B_r$ stands for the ball of center $0$ and radius $r>0$ in $\R^2$).

Next, for any $n\geq1$, we localise the initial datum $\big(\rho^{in},m^{in}\big)$ in $Q_n$, to produce an initial datum $\big(\rho^{in}_n,m^{in}_n\big)$.
For instance, we can fix a cut-off function $\k\in C^\infty_c(\R^2)$, $\k$ radially decreasing, such that $\k\equiv 1$ in a ball of radius
$1/4$ and $\k\equiv0$ outside the ball of radius $1/2$. Then, for any $n\in\N\setminus\{0\}$ and any $x_h\in\R^2$, we set $\k_n(x_h)\,=\,\k(x_h/n)$.
Therefore, for any $x\in\Omega$, after writing $x=(x_h,x_3)$, we can define
\[
\rho^{in}_n(x)\,:=\,\k_n(x_h)\,\rho^{in}(x)\qquad\qquad \mbox{ and }\qquad\qquad m^{in}_n(x)\,:=\,\k_n(2x_h)\,m^{in}(x)\,.
\]
It is easy to see that the new couple $\big(\rho^{in}_n,m^{in}_n\big)$ satisfies the assumptions of Proposition \ref{p:existence}, which thus
yields the existence of a finite energy weak solution $\big(\rho_n,u_n\big)$, which is defined globally in time. Passing to the limit
for $n\ra+\infty$ in the weak formulation of the equations is not difficult, and is based once again on a Ascoli-Arzel\`a or Aubin-Lions
type argument, which follows after having obtained uniform bounds on the time derivatives of the solution.
We skip the details here and refer to Step 3 in Section 2.4 of \cite{lions:book}.

Theorem \ref{exi:theo} is finally proved.
\end{proof}

To conclude this part, let us make an observation on the sign of the parameter $\alpha$, appearing in the boundary conditions \eqref{eq:bc}.
\begin{rmk} \label{r:alpha}
In this work we have decided to choose the slip coefficient $ \alpha \geq 0 $ due to physical reasons. However, we remark that
the same analysis can be performed in the case $ \alpha < 0 $, at least in the case when one imposes the additional assumption $ \rho_0^{in} \geq \rho_* > 0$.

As a matter of fact, in this case the negative term
\[
\alpha \iint_{[0,t[\,\times\d\Omega}| u |^2\,\dx_h\dd\t
\]
appearing in the energy inequality can be easily controlled. Indeed, the space integral can be estimated by
$ \|u\|_{L^2}\,\|\nabla u\|_{L^2} $, thus yielding a uniform bound of the energy in any finite
time interval $[0,T]$ (by Gr\"onwall, the energy may actually grow exponentially in time).
The previous arguments then apply, thus yielding existence of solutions also in the case $\alpha<0$ by following the same strategy.
\end{rmk}

\section{The fast rotation limit: main results} \label{s:fast}

In this section, we present the results on the fast rotation limit for system \eqref{eps:sys}-\eqref{eq:bc}.
At the mathematical level, this corresponds to study the limit of a sequence of solutions $\big(\rho_\veps,u_\veps\big)_{\veps\in\,]0,1]}$
as $ \eps$ converges to zero, under some reasonable assumptions on the initial data.

Notice that system \eqref{eps:sys}-\eqref{eq:bc} is characterised by two parameters that depend on $ \eps $.
The first one is $ \ell_{\eps} $, which describes the thickness of the domain; recall that $ \ell_{\eps} \searrow 0^+ $ for $\veps\ra0^+$.
The other parameter is the slip coefficient $ \alpha_{\eps} \geq 0$; in dependence of its order of magnitude relative to $\ell_\veps$, we will obtain different
asymptotics, as presented in Theorems \ref{Main:Theo} and \ref{Theo:2}.

So far, we have not assumed anything on the size of the initial densities; in particular, the previous statements work for initial densities which are small
perturbations around (loosely speaking) generic reference states.
In Theorem \ref{Theo:3}, instead, we present a result on the \emph{quasi-homogeneous} case, namely on the case when the initial densities are assumed
to be small perturbations of a constant state.

\paragraph{Important notations.}
Before moving on and presenting our main assumptions and results, let us introduce some useful notations, which we will adopt throughout the rest of the paper.

First of all, for a vector $ v \in \bbR^3 $, with $v=(v_1,v_2,v_3)$, we denote by $ v_h=(v_1,v_2) $ its horizontal component; we will often write
$v\,=\,(v_h,v_3)$, implying that $v_h\in\R^2$.
In the same spirit, for a function $f:\Omega_\veps\longrightarrow \R$ we define
\[
\nabla_hf\,:=\,\big(\d_1f,\d_2f\big)\qquad\qquad \mbox{ and }\qquad\qquad
\Delta_hf\,:=\,\d_1^2f+\d_2^2f\,,
\]
respectively the gradient and Laplace operator with respect to the horizontal derivatives only.
Analogously, for a vector field $ v:\Omega_\veps \longrightarrow \bbR^3 $, we define
\[
\divh(v)\,:=\,\partial_{1} v_1 + \partial_{2} v_2\qquad\qquad \mbox{ and }\qquad\qquad \curlh(v)\,:=\,\partial_{1} v_2 - \partial_{2} v_1\,.
\]
Finally, for a function $ f: \Omega_\veps \longrightarrow \bbR $ as above, we define its \emph{vertical average} as 
\[
\forall\,x_h\in\R^2\,,\qquad\qquad \overline{f}(x_h)\,:=\,\frac{1}{2\,\ell_{\eps}}\int_{-\ell_\veps}^{\ell_{\eps}} f(x_h,x_3)\,\dd x_3\,.
\]

In our arguments, the norms of averaged quantities will play a special role. However, they have not to be confused with the average of the norms,
which are in general larger (by the Jensen inequality) and naturally arise when averaging the energy inequality appearing in Definition \ref{def:weak:sol}.
Thus, for a function $f=f(t,x_h,x_3)$, in what follows we will distinguish between the notations
\[
\left\|\oline f\right\|_{L^r_T(L^p(\R^2))}\qquad\qquad \mbox{ and }\qquad\qquad \oline{\left\|f(\cdot,\cdot,x_3)\right\|_{L^r_T(L^p(\R^2))}}\,,
\]
which represent respectively the norm of the average and the average of the norm. In the latter case, we will write $f(\cdot,\cdot,x_3)$ for the function
inside the norm, in order to stress the fact that the norm is taken with respect to $(t,x_h)\in[0,T]\times\R^2$, thus giving a $x_3$-dependent quantity, which
is averaged afterwards.

To conclude, we will often have to work with bounded sequences, in order to derive weak convergence properties. Therefore, we will adopt the following
convenient notation: given a normed space $X$ and a sequence $\big(f_\veps\big)_\veps\subset X$, we will write $\big(f_\veps\big)_\veps\Subset X$ if the sequence
is also \emph{bounded} in $X$.

\subsection{Assumptions on the initial data} \label{ss:assumptions}

We now fix our assumptions on the initial data $\big(\rho_{\eps}^{in}, m_{\eps}^{in}\big)$, for any $\veps\in\,]0,1]$ fixed.

\paragraph{The initial densities.}
Regarding the density, we assume that
\[
\rho_{\eps}^{in}\, =\,\rho_{0}^{in}\, +\, \eps\, r^{in}_{\eps}\,,\qquad\qquad \mbox{ with }\qquad \rho_0^{in}\,=\,\rho^{in}_0(x_h)\,. 
\]
In other words, the target initial density profile $\rho_0^{in}$ does not depend on the vertical variables.
This is a quite reasonable assumption, as in this kind of problems one usually assumes the initial density states to be small perturbations around
an ``equilibrium''\footnote{The notion of ``equilibrium'' is much more pertinent in the context of weakly compressible fluids, see \tsl{e.g.}
\cite{F-N}, \cite{F-G-N}, \cite{F-G-GV-N}. Here, we simply mean that it is natural to assume that the limit density profile
is compatible with the target problem, \tsl{i.e.} in this case that it does not depend on the vertical variable.}
of the target problem.

We assume that $\rho_0^{in}$ satisfies the following assumptions:
\begin{equation} \label{hyp:rho_0}
\rho_0^{in}\,\in\,C^2_b(\R^2)\,:=\,C^2(\R^2)\cap W^{2,\infty}(\R^2)\,,\qquad\qquad \mbox{ with }\qquad 0\,\leq\,\rho_0^{in}\,\leq\,\rho^*\,,
\end{equation}
for a suitable positive constant $\rho^*\in\R$.
We notice that the previous conditions are fulfilled also in the case when $\rho_0^{in}$ is a constant state, say $\rho_0^{in}\equiv1$:
this is a very special case, simpler to handle, which will be the matter of Theorem \ref{Theo:3}.
The case when $\rho_0^{in}$ is not constant, instead, is a bit more involved: for handling it,
according to \cite{F:G} (see also \cite{C-F}), we will need an additional assumption on $\rho_0^{in}$; however, this being of technical nature, we prefer to state
it at the end of this subsection.

For the existence theory of weak solutions (recall Definition \ref{def:weak:sol} and Theorem \ref{exi:theo} above), of course we also need to assume that
\begin{equation} \label{hyp:rho_in-bound}
\forall\,\veps\in\,]0,1]\,,\qquad\qquad 
0\,\leq\,\rho_\veps^{in}\,\leq\,2\,\rho^*\,.
\end{equation}
In particular, this requires a sign condition on the density perturbations $r_\veps^{in}$ and implies that
\[
\sup_{\veps\in\,]0,1]}\left(\veps\,\left\|r_\veps^{in}\right\|_{L^{\infty}(\Omega_\veps)}\right)\,\leq\,C\,.
\]

We also need to assume that
\begin{equation} \label{hyp:rho_in}
\forall\,\veps\in\,]0,1]\,,\quad \exists\,\de_\veps>0\qquad \mbox{ such that }\qquad\qquad 
\frac{1}{\rho_\veps^{in}}\,\mds{1}_{\{\rho_\veps^{in}\leq\de_\veps\}}\,\in\,L^1(\Omega_\veps)\,.
\end{equation}
Again, this condition is required for the existence theory, recall \eqref{eq:non-deg} above.
We point out that other conditions would be possible, see Chapter 2 of \cite{lions:book},
but we limit ourselves to consider the previous one, for simplicity of exposition.

We still have to specify some uniform hypotheses on the perturbation functions $r_\veps^{in}$'s. Recall that we agree to write $\big(f_\veps\big)_\veps\Subset X$
for a sequence $\big(f_\veps\big)_\veps$ in a normed space $X$ which is  \emph{bouned} in $X$. Then, we assume that
\begin{equation} \label{ub:r_in}
\left(\oline{r_\veps^{in}}\right)_{\veps\in\,]0,1]}\,\Subset\,L^\infty(\R^2)\cap H^{-2}(\R^2)\,. 
\end{equation}
Thus, there exists a function $r_0^{in}\in L^\infty(\R^2)\cap H^{-2}(\R^2)$ such that, up to a suitable extraction of a subsequence (which we omit here), we have
the weak-$*$ convergence
\[
\oline{r_\veps^{in}}\,\stackrel{*}{\rightharpoonup}\,r_0^{in}\qquad\qquad \mbox{ in }\qquad L^\infty(\R^2)\cap H^{-2}(\R^2)\,. 
\]

Before moving on, we observe that the $H^{-2}$ condition in \eqref{ub:r_in} is formulated only for simplicity, but it could actually be dispensed of. We refer
to Remark \ref{r:hyp_sigma} for more comments about this.

\paragraph{The initial momenta.}
Regarding the initial momenta $\big(m_\veps^{in}\big)_\veps$, first of all we assume that all the conditions allowing to prove existence of weak solutions
(in the sense of Definition \ref{def:weak:sol}) are satisfied. So, we assume that
\begin{align*}
\forall\,\veps\in\,]0,1]\,,\qquad\qquad
&m_\veps^{in}\in L^2(\Omega_\veps)\,, \\
&m_\veps^{in}\equiv0\quad \mbox{ on the set }\quad \Big\{x\in\Omega_\veps\,\Big|\quad \rho_\veps^{in}(x)=0 \Big\}\,, \\
&\left|m_\veps^{in}\right|^2/\rho_\veps^{in}\,\in\,L^1(\Omega_\veps)\,.
\end{align*}

Next, we require that
\[
\left(\;\oline{m^{in}_\veps}\right)_{\veps\in\,]0,1]}\,\Subset\,L^2(\R^2)\,.
\]
%
%
So, there exists a vector fields $m^{in}_0 \in L^2(\R^2)$,
so in particular $m^{in}_0=m^{in}_0(x_h)$,
such that, up to a suitable extraction (omitted here), in the limit $\veps\ra0^+$ one has
\[
\oline{m^{in}_\veps}\,\rightharpoonup\,m^{in}_0
\qquad\qquad \mbox{ weakly in }\quad L^2(\R^2)\,.
\]




\paragraph{Uniform bounds for the initial energy.}
We now need to require that the initial data $\big(\rho_\veps^{in}, m_\veps^{in}\big)_{\veps}$ have initial energies which are \emph{uniformly bounded}:
such an assumption is fundamental in order to derive, from the energy inequality (see point (v) of Definition \ref{def:weak:sol}),
uniform bounds for the corresponding family of weak solutions.
Due to the geometry of our problem, it is easy to realise that the only reasonable conditions to impose must involve vertical averages.

In light of the previous discussion, we assume that there exists a constant $C>0$ such that
\begin{equation} \label{hyp:en_bound}
\sup_{\veps\in\,]0,1]}\left(\;\oline{\left\| \big(m_{\eps}^{in}/\sqrt{\rho^{in}_{\eps}}\big)(\cdot,x_3)\right\|^2_{L^2(\R^2)}}\;\right)\,\leq\,C
\end{equation}

In order to derive uniform bounds for the velocity fields $\oline{u_\veps}$, avoiding the degeneracy
of the densities close to vacuum, we also need to complement condition \eqref{hyp:rho_in} with the following hypothesis: we assume that
\begin{equation} \label{hyp:rho_in_unif}
\exists\,\de>0\qquad \mbox{ such that }\qquad\qquad 
\left(\;\oline{\frac{1}{\rho_\veps^{in}}\,\mds{1}_{\{\rho_\veps^{in}\leq\de\}}}\;\right)_{\veps\in\,]0,1]}\,\Subset\,L^1(\R^2)\,.
\end{equation}


\paragraph{A non-degeneracy condition.}
The last assumption rests on the target density profile $\rho_0^{in}$ and is of technical nature: we require that, if $\rho_0^{in}$ is \emph{not constant}, then
the following condition,
\begin{equation}
\label{cond:12}
\forall\,K\subset\subset\R^2\quad \mbox{ compact}\,,\qquad\qquad
\lim_{\delta \to 0^+ } \mc L\left(\Big\{ x \in K\;\Big|\quad \left|\nabla_h{\rho^{in}_0}\right|\, \leq\,\delta\Big\}\right)\, =\,0\,,
\end{equation}
must hold true,
where we have denoted by $\mc L(A)$ the Lebesgue measure of a set $A\subset \R^2$.

Roughly speaking, condition \eqref{cond:12} means that the reference state $\rho_0^{in}$ is really far from being constant, at least whenever
we look at it in any compact set $K$ of $\R^2$. We also remark that requiring \eqref{cond:12} is equivalent to assume that
\[
\forall\,K\subset\subset\R^2\quad \mbox{ compact}\,,\qquad\qquad
\mc L\left(\Big\{ x \in K\;\Big|\quad \nabla_h{\rho^{in}_0}\,=\,0\Big\}\right)\, =\,0\,.
\]
However, \eqref{cond:12} really reflects what we will need in our computations, so it is a handier relation for us. We point out that
an analogous requirement was already needed in the $2$-D case, see \cite{F:G} and \cite{C-F}, and was inspired by a similar condition appearing
in \cite{G-SR} (in the context of fast rotating homogenous fluids with Coriolis force depending on the latitude).

\subsection{Statement of the main results} \label{ss:results}

In this subsection, we present our main results. The basic observation is that there exist two qualitatively different regimes, 
depending on the relative values of the two parameters $\alpha_\veps$ (the strength of the slip condition at the boundary of the domain $\Omega_\veps$)
and $\ell_\veps$ (the thickness of the domain $\Omega_\veps$). On the one hand, when
$\lam_\veps\,:=\,\alpha_{\eps} / \ell_{\eps} \,\longrightarrow\,\lam\geq0 $,
some non-trivial dynamics is expected in the limit $\veps\ra0^+$, which is of course purely two dimensional and horizontal.
On the other hand, when $\lam_\veps\,\longrightarrow\,
+\infty$, the boundary conditions imply that there is no dynamics in the limit, because, in that instance, the target velocity profile is forced to be simply $0$.

We also remark that, according to the purely $2$-D investigations (see again \cite{F:G} and \cite{C-F}), in the case when $\rho_0^{in}$ is really
non-constant (in the sense of assumption \eqref{cond:12} above), the target system is, in general, underdetermined.
On the contrary, when $\rho_0^{in}$ is taken constant, say $\rho_0^{in}\equiv1$, the limit dynamics becomes immediately clear and one obtains
a fully determined system, which is globally well-posed (see \cite{F:G} again). For this reason, we will devote to the case $\rho_0^{in}\equiv1$
a separate statement, see Theorem \ref{Theo:3} below.

\medbreak
This having been pointed out, let us state our main results. We start with the case when $\rho_0^{in}$ is truly variable and by considering
the non-degenerate regime in which $\lam_\veps\,:=\,\alpha_{\eps} / \ell_{\eps} \geq0$ converges to some finite limit $\lam\geq0$ when $\veps\ra0^+$.

\begin{thm}
\label{Main:Theo}
Let $\big(\ell_\veps\big)_{\veps\in\,]0,1]}$ and $\big(\alpha_\veps\big)_{\veps\in\,]0,1]}$ be two sequences of positive real numbers
such that condition \eqref{eq:ell} holds true. Assume also that
\[
\exists\,\lam\geq0\qquad\qquad \mbox{ such that }\qquad\qquad\qquad  \lambda_\veps\,:=\,\frac{\alpha_\veps}{\ell_\veps}\,\longrightarrow\,\lam\qquad \mbox{ when }\quad
\veps\ra0^+\,.
\]
Take a sequence of initial data $\big(\rho_\veps^{in},m_{\veps}^{in}\big)_{\veps\in\,]0,1]}$ satisfying the hypotheses fixed in Subsection \ref{ss:assumptions},
for a non-constant reference density profile $\rho^{in}_0$.
For any $\veps\in\,]0,1]$, let $\big(\rho_\veps,u_\veps\big)$ be a global in time finite energy weak solution to system \eqref{eps:sys}-\eqref{eq:bc}
on $\R_+\times\Omega_\veps$, where $\Omega_\veps$ is defined in \eqref{eq:Omega_e}.

Then, there exist a function $\s\in L^\infty\big(\R_+;H^{-2}(\R^2)\big)$ and a two-dimensional vector field $u\in L^2_{\rm loc}\big(\R_+;H^1(\R^2)\big)$,
with $\divh(u)\,=\,\divh(\rho_0^{in}\,u)\,=\,0$, such that, up to the extraction of a suitable subsequence, we have, for any fixed time $T>0$, the following
convergence properties:
\begin{enumerate}[(a)]

\item $ \overline{\rho_{\eps}} \cvwstar \rho_0^{in} $ in $ L^{\infty}\big([0,T]\times \bbR^2 \big) $ and
$ \overline{\rho_{\eps}} \longrightarrow \rho_0^{in} $ in $ C^0\big([0,T];L^p_w(K)\big) $ for any compact $ K \subset \bbR^2 $
and for any $1\leq p<+\infty$;

\item $\overline{ u_{\eps,h} } \cv u $ in $ L^2\big([0,T];H^1(\bbR^2)\big) $;

\item $ \oline{u_{\eps,3}}\, \longrightarrow\, 0 $ (strong convergence) in $ L^2\big([0,T];L^2(\R^2)\big) $;

\item $ \overline{\sigma_{\eps} }\, \cvwstar\, \sigma $ in $ L^{\infty}\big([0,T];H^{-2}(\bbR^2)\big) $. 

\end{enumerate} 
In addition, if we set $ \omega\,:=\,\curlh( u ) $ and $  \eta  = \curlh\big(\rho_0^{in} u\big) $, then there exists a distribution
$ \Gamma\in\mc D'\big(\R_+\times\R^2\big) $ such that 
the following equation,
\[
\left\{ \begin{array}{l}
        \partial_t \big( \eta - \sigma\big)\,+\, 2\, \lam\, \omega\, -\, \nu\, \Delta_h \omega\, +\, \curlh\big( \rho_0^{in}\,\nabla_h \Gamma \big)\, =\, 0 \\[1ex]
        \big( \eta  - \sigma\big)_{|t=0}\, =\,\curlh\big( m_{0}^{in}\big)\,-\,r_0^{in}\,.
        \end{array}
\right.
\]
is satisfied in the weak sense, where $r_0^{in}\,=\,r_0^{in}(x_h)$ and $m_0^{in}\,=\,m_0^{in}(x_h)$ have been defined in Subsection \ref{ss:assumptions}.
\end{thm}

In the previous statement, given a Banach space $X$ and  its dual space $X^*$, the notation $C^0\big([0,T];X^*_w\big)$
denotes the space of functions $f:[0,T]\longrightarrow X^*$ which are continuous with respect to the weak topology of $X^*$: more precisely, one has that,
for any $\vphi\in X$, the map $t\mapsto \lan f(t),\vphi\ran_{X^*\times X}$ is continuous over $[0,T]$.

\medbreak
The case where $\lam_\veps\,:=\, \alpha_{\eps} / \ell_{\eps} $ diverges, instead, is not really relevant for the investigation of the asymptotic behaviour
of the system. This is explained by the next statement: in that case the velocity fields $u_\veps$'s converge to zero, so no dynamics can
be seen in the limit.

\begin{thm}
\label{Theo:2}
Let the assumptions of Theorem \ref{Main:Theo} be in force, but suppose this time that
\[
\lam_\veps\,:=\,\frac{\alpha_\veps}{\ell_\veps}\,\longrightarrow\,+\infty\qquad\qquad \mbox{ in the limit }\qquad \veps\ra0^+\,.
\]

Then, for any $T>0$ fixed, the convergence properties of the densities stated in item (a) of Theorem \ref{Main:Theo}
still holds true; in addition, one has the strong convergence 
$\overline{u_{\eps}}\, \longrightarrow\, 0$ in the space $L^2\big([0,T];L^2(\R^2)\big)$.
\end{thm}

As already remarked at the beginning of this subsection, the result of Theorem \ref{Main:Theo} is not completely satisfactory,
inasmuch as the limit dynamics is not identified in a clear way. As a matter of fact, the limit system is underdetermined, because
only one equation is derived for the two quantities $\eta$ (or, equivalently, $\omega$) and $\s$ which encode the asymptotic dynamics. We also
notice the presence of the ``Lagrangian multiplier'' $\rho_0^{in}\,\nabla_h\Gamma$ associated to the divergence-free constraint
$\divh\big(\rho_0^{in}\,u\big)\,=\,0$; however, the distribution $\Gamma$ is not better identified either.

In the case when $\rho_0^{in}$ is constant, say $\rho_0^{in}\equiv1$, however, much more precise information on the limit dynamics can be derived.
Of course, this is not surprising: indeed, in the instance $\rho_\veps\,=\,1\,+\,\veps\,r_\veps$, the (singular) Coriolis operator can be decomposed
into
\[
 \frac{1}{\veps}\,e_3\times\rho_\veps\,u_\veps\,=\,\frac{1}{\veps}\,e_3\times u_\veps\,+\,\,e_3\times r_\veps\,u_\veps\,,
\]
which is a merely $O(1)$ perturbation of the classical homogeneous case (which is treatable in a $3$-D geometry, see \tsl{e.g.} \cite{C-D-G-G}, \cite{G-SR}).

The result for $\rho_0^{in}\equiv1$ reads as follows. Notice that we limit ourselves to state the result
in the case $\lam_\veps\longrightarrow\lam<+\infty$, because, when $\lam=+\infty$, then the limit dynamics becomes trivial (namely one has $u\equiv0$).

\begin{thm}
\label{Theo:3}
Let $\big(\ell_\veps\big)_{\veps\in\,]0,1]}$ and $\big(\alpha_\veps\big)_{\veps\in\,]0,1]}$ be two sequences of positive real numbers
such that condition \eqref{eq:ell} holds true and such that
\[
\exists\,\lam\geq0\qquad\qquad \mbox{ such that }\qquad\qquad\qquad  \lambda_\veps\,:=\,\frac{\alpha_\veps}{\ell_\veps}\,\longrightarrow\,\lam\qquad \mbox{ when }\quad
\veps\ra0^+\,.
\]
Take a sequence of initial data $\big(\rho_\veps^{in},m_{\veps}^{in}\big)_{\veps\in\,]0,1]}$ such that, for any $\veps\in\,]0,1]$, one has
$\big(\rho^{in}_\veps,m_{\veps}^{in}\big)=\big(\rho_\veps^{in},\rho_{\veps}^{in}u_{\eps}^{in} \big)$, with $ \div\big(u_{\eps}^{in}\big) = 0 $,
and such that they satisfy the hypotheses fixed in Subsection \ref{ss:assumptions},
where we assume that $\rho^{in}_0\equiv 1$ and, correspondingly, we dismiss condition \eqref{cond:12}.
Assume also that there exists a positive constant $C>0$ such that
\begin{equation*}
\forall\,\veps>0\,,\qquad\qquad \overline{\left\|r^{in}_{\eps}(\cdot,x_3)\right\|^2_{L^2(\R^2)}}\,\leq\, C \quad \text{ and } \quad \|r_{\eps}^{in}\|_{L^{\infty}(\Omega_{\eps})} \leq C.
\end{equation*}
For any $\veps\in\,]0,1]$, let $\big(\rho_\veps,u_\veps\big)$ be a global in time finite energy weak solution to system \eqref{eps:sys}-\eqref{eq:bc}
on $\R_+\times\Omega_\veps$, where $\Omega_\veps$ is defined in \eqref{eq:Omega_e}, related to the initial datum
$\big(\rho_\veps^{in},m_\veps^{in}\big)$.

Then, for any $\veps\in\,]0,1]$, one can write
\[
\rho_\veps\,=\,1\,+\,\veps\,r_\veps\,,\qquad\qquad \mbox{ with }\qquad
r_\veps\,\in\,C^0_b\big(\R_+;L^2(\Omega_\veps)\big)\cap L^\infty\big(\R_+\times\Omega_\veps\big)\,,
\]
where the notation $C^0_b$ stands for the intersection $C^0\cap L^\infty$ and $r_\veps$
solves
\[
\d_tr_\veps\,+\,\div\big(r_\veps\,u_\veps\big)\,=\,0\,,\qquad\qquad \mbox{ with }\qquad \big(r_\veps)_{|t=0}\,=\,r_0^{in}\,,
\]
in the sense of $\mc D'\big(\R_+\times\Omega_\veps\big)$. In addition,
there exist $r_0\in C^0_b\big(\R_+;L^2(\R^2)\big)\cap L^\infty\big(\R_+\times\R^2\big)$ and a two-dimensional vector field
$u\in L^\infty\big(\R_+;L^2(\R^2)\big)$, with $D_hu\in L^2\big(\R_+;L^2(\R^2)\big)$ and $\divh(u)=0$, such that,
up to the extraction of a suitable subsequence (not relabelled here), we have the following convergence properties, for any time $T>0$ fixed:
\begin{itemize}

\item $ \overline{r_{\eps}} \cvwstar r_0 $ in $ L^{\infty}\big([0,T]\times \bbR^2 \big) $ and
$ \overline{r_{\eps}} \longrightarrow r_0 $ in $ C^0\big([0,T];L^p_w(K)\big) $ for any compact $ K \subset \bbR^2 $
and for any $1\leq p<+\infty$;

\item $\overline{ u_{\eps,h} } \cv u $ in the space $ L^2\big([0,T];H^1(\mathbb{R}^2)\big)$;

\item $ \overline{u_{\eps,3}}  \longrightarrow 0 $ in $ L^2\big([0,T];L^2(\mathbb{R}^2)\big)$.

\end{itemize} 
Finally, the couple $\big(r_0, u\big) $ satisfies (in the weak sense) the system
\begin{equation} \label{final:sys}
\left\{\begin{array}{l}
\partial_t r_0\, +\, \divh\big( r_0\, u\big)\, =\, 0 \\[1ex]
\partial_t u\, +\, \divh\big( u \otimes u\big)\, -\, \Delta_h u\, +\, \nabla_h \Pi \,+\, r_0\,u^{\perp}\,+\, 2\,\lam\, u\, = \, 0 \\[1ex]
\divh(u) \,= \, 0
       \end{array}
\right.
\end{equation}
over $\R_+\times\R^2$, with respect to the initial condition $\big(r_0,u\big)_{|t=0}\,=\,\big( r_0^{in}, u_{0}^{in}\big)$
and for a suitable pressure function $\Pi$.
\end{thm}

The next sections are devoted to the proof of the previous statements. In Section \ref{s:bounds} we will derive uniform
bounds for the sequence of weak solutions we consider. As
Theorems \ref{Theo:2} and \ref{Theo:3} are direct consequences of those uniform bounds, their proofs will be
carried out in the next section. The proof of Theorem \ref{Main:Theo}, instead, is much more involved, and it will be the matter of Section \ref{s:proof-main}.

\section{Uniform bounds and consequences} \label{s:bounds}

This section is devoted to some preliminaries, which are necessary for proving our asymptotic results.
First of all, from the finite energy condition on each weak solution $\big(\rho_\veps,u_\veps\big)$ and the assumptions on the initial data,
we derive \emph{uniform bounds} in suitable norms for the whole family  $\big(\rho_\veps,u_\veps\big)_\veps$. This is the matter of Subsection
\ref{ss:finite-en}. Thanks to those uniform bounds, we can extract weak limit points: in Subsection \ref{ss:constraints} we state
some static constraints those limit points have to satisfy.
Finally, in Subsection \ref{ss:easy}, we complete the proof of Theorems \ref{Theo:2} and \ref{Theo:3}, as they are fairly direct
consequences of the uniform bounds. 

\subsection{Consequences of the finite energy condition} \label{ss:finite-en}

Directly from the definition of weak solution, recall Definition \ref{def:weak:sol}, together
with the assumptions imposed on the initial data in Subsection \ref{ss:assumptions}, we can deduce the following \tsl{a priori} estimates
for the family of weak solutions $\big(\rho_\veps,u_\veps\big)_\veps$.

\begin{lemma}
\label{Lemma:merc}
Let $\big(\rho_\veps^{in},m_{\veps}^{in}\big)_{\veps\in\,]0,1]}$ be a sequence of initial data satisfying the hypotheses fixed in Subsection
\ref{ss:assumptions}, where $\rho_0^{in}$ may be either constant, say $\rho_0^{in}\equiv1$, or non-constant.
For any $\veps\in\,]0,1]$, let $\big(\rho_\veps,u_\veps\big)$ be a global in time finite energy weak solution to system \eqref{eps:sys}-\eqref{eq:bc}
on $\R_+\times\Omega_\veps$, in the sense of Definition \ref{def:weak:sol}.

Then the following uniform estimates hold true: there exists a constant $C>0$ such that, for any $\veps\in\,]0,1]$ and almost any time $t\geq0$, one has
\begin{align*}
\left\| \overline{\rho_{\eps}}(t)\right\|_{L^\infty(\R^2)}\,\leq\,
\overline{\left\|\rho_{\eps}(t,\cdot,x_3)\right\|_{L^\infty(\R^2)}}\,&\leq\,
\left\| \rho_{\eps}^{in}(\cdot,x_3)\right\|_{L^\infty(\Omega_{\eps})}\,\leq\,C \\[1ex]
\left\|\overline{\sqrt{\rho_\veps}\,u_{\eps}}(t)\right\|_{L^2(\R^2)}^2\,\leq\,
\overline{\left\|\big(\sqrt{\rho_\veps}\,u_{\eps}\big)(t,\cdot,x_3)\right\|_{L^2(\R^2)}^2}\,&\leq\,
\oline{\left\|\big(m_\veps^{in} / \sqrt{\rho_\veps^{in}}\big)(\cdot,x_3)\right\|_{L^2(\R^2)}^2}\,\leq\,C  \\[1ex]
\left\|D_h\overline{u_{\eps}}\right\|_{L^2_t(L^2(\R^2))}^2\,\leq\,
\overline{ \left\|D u_{\eps}(\cdot,\cdot,x_3)\right\|_{L^2_t(L^2(\R^2))}^2}\,&\leq\,
\oline{\left\|\big(m_\veps^{in} / \sqrt{\rho_\veps^{in}}\big)(\cdot,x_3)\right\|_{L^2(\R^2)}^2}\,\leq\,C\, \\[1ex]
\frac{\alpha_{\eps}}{\ell_{\eps}}\,  \Big( \|u_{\eps}(\cdot,\cdot,-\ell_{\eps})\|_{L^2_t(L^2(\R^2))}^2\,+\,
\|u_{\eps}(\cdot,\cdot,\ell_{\eps})\|_{L^2_t(L^2(\R^2))}^2 \Big) \, &\leq \,
\oline{\left\|\big(m_\veps^{in} / \sqrt{\rho_\veps^{in}}\big)(\cdot,x_3)\right\|_{L^2(\R^2)}^2}\,\leq\,C\, .
\end{align*}
In addition, we have that
\[
\left(\oline{u_\veps}\right)_{\veps\in\,]0,1]}\,\Subset\,L^2_{\rm loc}\big(\R_+;H^1(\R^2)\big)\,.
\]

\end{lemma}

\begin{proof}
To begin with, we consider the finite energy inequality satisfied by the family of weak solutions, see item (v) of Definition \ref{def:weak:sol}: we have
\begin{align*}
& \frac{1}{2\,\ell_{\eps}} \int_{-\ell_\veps}^{\ell_{\eps}}\int_{\R^2}\rho_{\eps}(t)\,\left|u_{\eps}(t)\right|^2\,\dd x_h\,\dd x_3\, +\,
\frac{1}{2\,\ell_{\eps}}\int_{-\ell_\veps}^{\ell_{\eps}}\int_{0}^{t} \int_{\R^2}\left|Du_{\eps}\right|^2\,\dd x_h\,\dd\t\,\dd x_3 \\
&\qquad\qquad\qquad\qquad\qquad \,+\,
\frac{\alpha_{\eps}}{\ell_{\eps}}\int_{0}^t\int_{\R^2\times\{-\ell_\veps,\ell_\veps\}}\big|u_{\eps}\big|^2\,\dd x_h\,\dd\t\;\leq\;
\frac{1}{2\ell_{\eps}}\int_{-\ell_\veps}^{\ell_{\eps}}\int_{\R^2}\frac{\big|m_{\eps}^{in}|^2}{\rho_{\eps}^{in}}\,\dd x_h\,\dd x_3\,.
\end{align*}
Owing to our assumptions on the initial data, see in particular \eqref{hyp:en_bound},
the right-hand side of the previous inequality is uniformly bounded for $\veps\in\,]0,1]$. Hence, we deduce
that there exists a constant $C>0$ such that
\begin{align}
&\sup_{t\in\R_+}\overline{\left\|\big(\sqrt{\rho_\veps}\,u_{\eps}\big)(t,\cdot,x_3)\right\|_{L^2(\R^2)}^2}\, +\,
\overline{ \left\|D u_{\eps}(\cdot,\cdot,x_3)\right\|_{L^2_t(L^2(\R^2))}^2}\,  \label{est:unif-en} \\
&\qquad\qquad\qquad\qquad
+ \, \frac{\alpha_{\eps}}{\ell_{\eps}} \| u_{\eps}(\cdot,\cdot,-\ell_{\eps})  \|^2_{L^2_t(L^2(\bbR^2))} \, + \,
\frac{\alpha_{\eps}}{\ell_{\eps}} \| u_{\eps}(\cdot,\cdot,+\ell_{\eps})  \|^2_{L^2_t(L^2(\bbR^2))} \, \leq\,C\,. \nonumber
\end{align}
Of course, the norm of the average is controlled (for instance, by the Jensen inequality) by the average of the norm, so the second and third inequalities
are proved. Observe that the Navier-slip boundary conditions imply that
$\oline{\d_3u_{\veps,3}}\equiv0$, but we cannot infer the vanishing of the mean of the vertical derivatives of the horizontal components.

Let us now switch to consider the density functions. We notice that $\rho_\veps$ is transported by the divergence-free velocity field $u_\veps$,
so, for any $p\in[1,+\infty]$, its $L^p$ norm is (formally) preserved for all times.
By taking $p=+\infty$ and averaging with respect to the vertical variable, for almost all $t\geq0$ we immediately get
\[
\forall\,\veps\in\,]0,1]\,,\qquad\qquad
\left\|\oline{\rho_\veps}(t)\right\|_{L^\infty(\R^2)}\,\leq\,\left\|\rho_\veps(t)\right\|_{L^\infty(\Omega_\veps)}\,=\,
\left\| \rho_{\eps}^{in}\right\|_{L^\infty(\Omega_\veps)}\,\leq\,2\,\rho^*\,,
\]
where we have also used \eqref{hyp:rho_in-bound}.
By the same token, employing \eqref{hyp:rho_in_unif}, we see that
\begin{equation} \label{est:rho_unif}
\sup_{t\in\R_+}\oline{\left\|\Big(\big(1 / \rho_\veps\big)\;\mds{1}_{\{\rho_\veps\leq\de\}}\Big)(t,\cdot,x_3)\right\|_{L^1(\R^2)}}\,=\,
\sup_{t\in\R_+}\left\|\Big(\oline{\big(1 / \rho_\veps\big)\;\mds{1}_{\{\rho_\veps\leq\de\}}}\Big)(t)\right\|_{L^1(\R^2)}\,\leq\,C\,.
\end{equation}

To conclude, we have to derive the claimed uniform boundedness of the sequence $\big(\oline{u_\veps}\big)_\veps$. For this, we are going to adapt
the arguments of \cite{lions:book}; estimate \eqref{est:rho_unif} will play a key role.

We start by writing
\[
u_\veps\,=\,u_\veps^1\,+\,u_\veps^2\,,\qquad\qquad \mbox{ with }\qquad u_\veps^1\,:=\,u_\veps\,\mds{1}_{\{\rho_\veps\geq\de\}}\,,\quad
u_\veps^2\,:=\,u_\veps\,\mds{1}_{\{\rho_\veps\leq\de\}}\,.
\]
Owing to \eqref{est:unif-en}, we easily see that $\big(\,\oline{u_\veps^1}\,\big)_\veps$ is uniformly bounded in $L^\infty\big(\R_+;L^2(\R^2)\big)$.
As for the other term, for almost any time $t\geq0$ (which we omit from the notation) we can bound
\begin{align*}
\left\|\oline{u_\veps^2}\right\|_{L^1(\R^2)}\,&=\,\frac{1}{2\,\ell_\veps}\int^{\ell_\veps}_{-\ell_\veps}\int_{\R^2}
\frac{1}{\sqrt{\rho_\veps}}\,\mds{1}_{\{\rho_\veps\leq\de\}}\,\sqrt{\rho_\veps}\,\big|u_\veps\big|\,\dd x_h\,\dd x_3 \\
&\leq\,\frac{1}{2\,\ell_\veps}\int^{\ell_\veps}_{-\ell_\veps}
\left\|\left(\frac{1}{\sqrt{\rho_\veps}}\,\mds{1}_{\{\rho_\veps\leq\de\}}\right)(\cdot,x_3)\right\|_{L^2(\R^2)}\,
\left\|\big(\sqrt{\rho_\veps}\,u_\veps\big)(\cdot,x_3)\right\|_{L^2(\R^2)}\,\dd x_3\,.
\end{align*}
Observing that
\[
\left\|\left(\frac{1}{\sqrt{\rho_\veps}}\,\mds{1}_{\{\rho_\veps\leq\de\}}\right)(\cdot,x_3)\right\|_{L^2(\R^2)}\,=\,
\left\|\left(\frac{1}{\rho_\veps}\,\mds{1}_{\{\rho_\veps\leq\de\}}\right)(\cdot,x_3)\right\|^{1/2}_{L^1(\R^2)}\,,
\]
an application of the Cauchy-Schwarz inequality with respect to the integral in the $x_3$-variable finally yields
\begin{align*}
\left\|\oline{u_\veps^2}\right\|_{L^1(\R^2)}\,&\leq\,\frac{1}{2\,\ell_\veps}
\left(\int^{\ell_\veps}_{-\ell_\veps}\left\|\left(\frac{1}{\rho_\veps}\,\mds{1}_{\{\rho_\veps\leq\de\}}\right)(\cdot,x_3)\right\|_{L^1(\R^2)}\right)^{1/2}\,
\left(\int^{\ell_\veps}_{-\ell_\veps}\left\|\big(\sqrt{\rho_\veps}\,u_\veps\big)(\cdot,x_3)\right\|_{L^2(\R^2)}^2\right)^{1/2} \\
&=\,\left(\;\oline{\left\|\left(\frac{1}{\rho_\veps}\,\mds{1}_{\{\rho_\veps\leq\de\}}\right)(\cdot,x_3)\right\|_{L^1(\R^2)}}\;\right)^{1/2}\,
\left(\;\oline{\left\|\big(\sqrt{\rho_\veps}\,u_\veps\big)(\cdot,x_3)\right\|_{L^2(\R^2)}^2}\;\right)^{1/2}\,.
\end{align*}
In view of \eqref{est:rho_unif} and \eqref{est:unif-en} again, from the previous inequality we infer that, for almost any $t\geq0$,
one has $\big(\,\oline{u_\veps^2}(t)\,\big)_\veps\,\Subset\,L^1(\R^2)$.
This fact, together with the uniform bound on the gradients $\big(D_h\oline u_\veps\big)_\veps\,\Subset\,L^2\big(\R_+;L^2(\R^2)\big)$,
finally implies that $\big(\oline{u_\veps}\big)_\veps$ is uniformly bounded in $L^2_{\rm loc}\big(\R_+;H^1(\R^2)\big)$ (see Appendix B
of \cite{lions:book} for the precise argument).

The lemma is thus completely proved.
\end{proof}

Let us now derive some consequences of the uniform bounds stated in Lemma \ref{Lemma:merc}. First of all,
we see that there exist a density function $\rho_0\,=\,\rho_0(t,x_h)\,\in\,L^\infty\big(\R_+\times\R^2\big)$ and
a vector field $u\,=\,u(t,x_h)\,\in\,L^2_{\rm loc}\big(\R_+;H^1(\R^2)\big)$, with $u\,=\,\big(u_1,u_2\big)$ two-dimensional,
such that, up to an extraction (as usual, omitted here), one has
\begin{equation}
\label{Conv:u}
\overline{\rho_{\eps}}\,\stackrel{*}{\rightharpoonup}\, \rho_0\quad \mbox{ in }\ L^{\infty}\big(\R_+\times\R^2\big) \qquad\quad \mbox{ and } \qquad\quad \overline{u_{\eps,h}}\,\rightharpoonup\,u \quad \mbox{ in }\ L^2_{\rm loc}\big(\R_+;H^1(\R^2)\big) 
\end{equation}  
in the lmit $\veps\ra0^+$. By taking the average of the divergence-free condition $\div(u_\veps)=0$, which holds in the sense
of distributions, one also discovers that
\begin{equation} \label{eq:u-div-free}
\divh(u)\,=\,0\,.
\end{equation}
As for the vertical components $\big(\oline{u_{\veps,3}}\big)_\veps$, for any $\veps\in\,]0,1]$, almost any $t\geq0$ and almost any
$x=(x_h,x_3)\,\in\,\Omega_\veps$ we can compute
\begin{align*}
u_{\veps,3}(t,x_h,x_3)\,&=\,u_{\veps,3}(t,x_h,x_3)\,-\,u_{\veps,3}(t,x_h,-\ell_\veps)\,=\,\int_{-\ell_\veps}^{x_3}\d_3u_{\veps,3}(t,x_h,z)\,\dd z\,,
\end{align*}
where we have used the boundary conditions \eqref{eq:bc}.
From the previous relation we deduce that
\begin{align*}
\left\|u_{\veps,3}(t,\cdot,x_3)\right\|_{L^2(\R^2)}^2\,\leq\,4\,\ell^2_\veps\,\left\|\oline{\d_3u_{\veps,3}}\right\|^2_{L^2(\R^2)}\,,
\end{align*}
which finally implies, owing to the bounds stated in Lemma \ref{Lemma:merc}, that
\begin{equation} \label{est:u_3}
\left\|\oline{u_{\veps,3}}\right\|^2_{L^2_T(L^2(\R^2))}\,\leq\,\oline{\left\|u(\cdot,\cdot,x_3)\right\|_{L^2_T(L^2(\R^2))}^2}\,\leq\,
C\,\ell_\veps\,\overline{ \left\|D u_{\eps}(\cdot,\cdot,x_3)\right\|_{L^2_T(L^2(\R^2))}^2}\,\longrightarrow\,0
\end{equation}
when $\veps\ra0^+$, for any $T\geq0$ fixed.

Moreover, after taking the vertical average of the mass equation, we get
\begin{equation*}
\partial_t \overline{\rho}_{\eps}\, +\, \divh\big(\overline{\rho_{\eps}\,u_{\eps,h}}\big)\, =\, 0\,,
\end{equation*} 
which holds true in the distributional sense on $\R_+\times\R^2$. At this point, it is worth mentioning that, as each $\rho_\veps$ is transported by
the divergence-free vector field $u_\veps$ and the initial datum $\rho_\veps^{in}$ satisfies \eqref{hyp:rho_in-bound},
we get
\begin{equation} \label{est:rho_inf}
\forall\,\veps>0\,,\qquad\qquad 0\,\leq\,\rho_\veps\,\leq\,2\,\rho^*\,.
\end{equation}
Thus, from the equation above, we deduce that the sequence $\big(\partial_{t}\overline{\rho_{\eps}}\big)_\veps$
is uniformly bounded in the space $ L^\infty\big([0,T];H^{-1}(\R^2)\big) $, where the time $T>0$ can be taken arbitrarily large. Using classical
arguments (see \tsl{e.g.} Appendix C of \cite{lions:book} for details), we then find the strong convergence property
\begin{equation}
\label{Cweak:conv}
\overline{\rho_{\eps}}\, \longrightarrow\, \rho_0\qquad \mbox{ in }\qquad  C^0\big([0,T];L^p_w(K)\big)\,, 
\end{equation} 
for any fixed time $T>0$, any compact set $K\subset\subset\R^2$ and any finite $p\in[1,+\infty[\,$.

We conclude this part by noticing that, in the case $ \alpha_{\eps}/\ell_{\eps} \to \lambda $ with $ \lambda \in \,]0,+\infty[\,$,
then the \tsl{a priori} estimates imply the weak convergence of the boundary values of the velocities $\big(u_{\eps}\big)_\veps $.
\begin{lemma}
	\label{lemma:cons:bound:data}
	Assume that $ \alpha_{\eps}/\ell_{\eps} \to \lambda $, for some $0<\lam<+\infty$. Then there exist two two-dimensional vector fields
	$u^\pm\,=\,u^\pm(t,x_h)$, belonging to $L^2_{\rm loc}\big(\R_+;L^2(\R^2)\big)$, such that, for any $T>0$ fixed, one has
\begin{equation*}
u_{\eps}(\cdot,\cdot, -\ell_{\eps})\,\cv\,u^- \quad \text{ and } \quad u_{\eps}(\cdot,\cdot,\ell_{\eps}) \cv u^+ \qquad \text{ in }\quad
L^{2}\big([0,T]\times \bbR^2\big)\,.
\end{equation*}
	
\end{lemma}

\begin{proof}
	If $0<\lambda<+\infty $, for $ \eps>0 $ small enough one has $ \alpha_{\eps}/\ell_{\eps} > \lambda / 2 $. Then,
	fixed any $T>0$, inequality \eqref{est:unif-en} implies 
\begin{equation*}
\left\|u_{\eps}(\cdot,\cdot,-\ell_{\eps})\right\|^2_{L^2([0,T]\times \bbR^2) }\,+\,
\left\|u_{\eps}(\cdot,\cdot,\ell_{\eps})\right\|^2_{L^2([0,T]\times \bbR^2) }\, \leq\,\frac{2\,C}{\lambda}\,.
\end{equation*}
	Up to extracting a suitable subsequence, we then derive the desired convergence.
\end{proof}

\subsection{Constraints on the limit points} \label{ss:constraints}

In this subsection we will use the convergences \eqref{Conv:u} and \eqref{Cweak:conv} to obtain some constrains on the limit points
$\rho_0$ and $u$ identified above. We start
by recalling the right scaling for the Poincar\'e and Sobolev inequalities in thin domains. 

\begin{lemma} \label{l:Poincare} 
There exists a positive constant $C>0$ such that,
for $ u_{\eps} \in H^1(\Omega_{\eps}) $, one has the following inequalities:
\begin{align*}
\frac{1}{2 \ell_{\eps}} \int_{-\ell_{\eps}}^{\ell_{\eps}} \int_{\bbR^2} \left| u_{\eps} -\overline{u_{\eps}}\right|^2\,\dd x_h\,\dd x_3\,&\leq\,
C\, \ell_{\eps}\, \int_{-\ell_{\eps}}^{\ell_{\eps}} \int_{\bbR^2} |D u_{\eps}|^2\,\dd x_h\,\dd x_3 \\
\left(\frac{1}{2 \ell_{\eps}} \int_{-\ell_{\eps}}^{\ell_{\eps}} \int_{\bbR^2} \left| u_{\eps} \right|^6\,\dd x_h\,\dd x_3\right)^{1/6}\,&\leq\,
C\, \left(\frac{1}{2\ell_{\eps}}\, \int_{-\ell_{\eps}}^{\ell_{\eps}} \int_{\bbR^2}\Big(|u_{\eps}|^2 + |D u_{\eps}|^2\Big)\,\dd x_h\,\dd x_3\,\right)^{1/2}.
\end{align*}
\end{lemma}

\begin{proof}
It is enough to use the change of variables $ U(x,y,z)\,:=\,u(x,y,\ell_{\eps} z ) $ and recall that, in the set $ \bbR^2 \times \,]-1,1[\,$,
the classical Poincar\'e and Sobolev inequalities read as follows:
\begin{align*}
\int_{-1}^{1}\int_{\bbR^2}\left|U-\overline{U}\right|^2\,\dd x_h\,\dd z\,&\leq\,C\,\int_{-1}^{1} \int_{\bbR^2} \left| \partial_z U \right|^2\,\dd x_h\,\dd z \\
\left( \int_{-1}^{1}\int_{\bbR^2}\left|U\right|^6\,\dd\ x_h\,\dd z\,\right)^{1/6}\,&\leq\,C\,\left(\int_{-1}^{1} \int_{\bbR^2}
\Big(\left| U \right|^2 +\left| D U \right|^2\Big)\,\dd x_h\,\dd z\, \right)^{1/2}\,.
\end{align*}
Performing the change of variables backwards yields the claimed inequalities.
\end{proof}

From the previous Lemma \ref{l:Poincare}, we derive the following direct but fundamental consequence.
\begin{cor}
\label{cor:1}
There exists an absolute constant $C>0$ such that the following property holds true. Take $T>0$ and a couple of real numbers
$ (s, p) \in [2,+\infty]^2 $. For any couple of functions $ f_{\eps} \in L^{s}\big([0,T];L^p(\Omega_{\eps})\big) $ and
$ u_{\eps} \in L^2\big([0,T];H^1(\Omega_{\eps})\big) $, one has
\[ 
\left\|\overline{f_{\eps} u_{\eps}}\, -\, \overline{f_{\eps}} \, \overline{u_{\eps}}\right\|_{L^{\frac{2s}{s+2}}_T\left(L^{\frac{2p}{p+2}}(\bbR^2)\right)}\,\leq\,
C\,\ell_{\eps}\,\left\| \left(\overline{\left\|f(t,\cdot,x_3)\right\|^p_{L^p(\R^2)}}\right)^{\frac{1}{p}} \right\|_{L^s_T}\,
\left(\overline{\left\|Du_{\eps}(\cdot,\cdot,x_3)\right\|^2_{L^2_T(L^2(\bbR^2))}}\right)^{\frac12}
\] 
if $p<+\infty$, and if $p=+\infty$ one has instead
\[
\left\|\overline{f_{\eps} u_{\eps}}\, -\, \overline{f_{\eps}} \, \overline{u_{\eps}}\right\|_{L^{\frac{2s}{s+2}}_T\left(L^{2}(\bbR^2)\right)}\,\leq\,
C\,\ell_{\eps}\,\left\|f\right\|_{L^s_T(L^\infty(\Omega_\veps))}\,
\left(\overline{\left\|Du_{\eps}(\cdot,\cdot,x_3)\right\|^2_{L^2_T(L^2(\bbR^2))}}\right)^{\frac12}\,.
\]

In particular, when $s=p\in[2,+\infty]$, one gets
\[
\left\|\overline{f_{\eps} u_{\eps}}\, -\, \overline{f_{\eps}} \, \overline{u_{\eps}}\right\|_{L^{\frac{2p}{p+2}}_T\left(L^{\frac{2p}{p+2}}(\bbR^2)\right)}\,\leq\,
C\,\ell_{\eps}\,\left(\overline{\left\|f(t,\cdot,x_3)\right\|^p_{L^p_T(L^p(\R^2))}}\right)^{\frac1p}\,
\left(\overline{\left\|Du_{\eps}(\cdot,\cdot,x_3)\right\|_{L^2_T(L^2(\bbR^2))}}\right)^{\frac12}\,,
\]
with the same modification as before in the case $p=+\infty$.
\end{cor}

\begin{proof}
We use the definition of average to write 
\[
\overline{f_{\eps} u_{\eps}}\,  -\, \overline{f_{\eps}} \, \overline{u_{\eps}}\,=\, \overline{f_{\eps} u_{\eps}\,  -\, f_{\eps} \, \overline{u_{\eps}}} \,.
\]
Therefore, an application of the H\"older inequality yields
\[
\left\|\overline{f_{\eps} u_{\eps}}\,-\,\overline{f_{\eps}} \, \overline{u_{\eps}}\right\|_{L^{\frac{2p}{p+2}}(\R^2)}\,\leq\,
\left(\oline{\left\|f_\veps(t,\cdot,x_3)\right\|^p_{L^p(\R^2)}}\right)^{\frac1p}\,
\left(\oline{\left\|u(t,\cdot,x_3)\,-\,\oline{u}(t,\cdot)\right\|^2_{L^2(\R^2)}}\right)^{\frac12}
\]
in the case $p<+\infty$, whereas for $p=+\infty$ one simply has the bound
\[
\left\|\overline{f_{\eps} u_{\eps}}\,-\,\overline{f_{\eps}} \, \overline{u_{\eps}}\right\|_{L^{2}(\R^2)}\,\leq\,
\left\|f\right\|_{L^\infty(\Omega_\veps)}\,
\left(\oline{\left\|u(t,\cdot,x_3)\,-\,\oline{u}(t,\cdot)\right\|^2_{L^2(\R^2)}}\right)^{\frac12}\,.
\]
At this point, we conclude by taking the $L^{\frac{2s}{s+2}}$ norm in time and by an application of the Poincar\'e inequality in thin domains, see
Lemma \ref{l:Poincare} above.
\end{proof}

We are now able to deduce some information on the limit points $\rho_0$ and $u$ identified in Subsection \ref{ss:finite-en}.

\begin{prop}
\label{con:all:var}
Let $\big(\rho_{\eps}, u_{\eps}\big)_{\veps\in\,]0,1]} $ a sequence of finite energy weak solutions to system \eqref{eps:sys}-\eqref{eq:bc},
each one associated with the initial datum $\big(\rho^{in}_{\eps}, m_{\eps}^{in}\big)$ that satisfies the hypotheses fixed in Theorem \ref{Main:Theo}.
Let $\rho_0$ and $u$ be the two limit points identified in \eqref{Conv:u} above.

Then, up to the extraction of a subsequence, the following convergence properties hold true: for any $T>0$, one has

\begin{itemize}

\item $ \overline{\rho_{\eps}} \cvwstar \rho_0 $ in $ L^{\infty}\big([0,T]\times \bbR^2 \big) $ and
$ \overline{\rho_{\eps}} \longrightarrow \rho_0 $ in $ C^0\big([0,T];L^p_w(K)\big) $ for any compact $ K \subset \bbR^2 $
and any finite value of $p\in[1,+\infty[\,$;

\item $\overline{ u_{\eps,h} } \cv u $ in $ L^2\big([0,T];H^1(\bbR^2)\big) $;

\item $ \oline{u_{\eps,3}}\, \longrightarrow\, 0 $ strongly in $ L^2\big([0,T];L^2(\R^2)\big) $;
 
\item $ \overline{\rho_{\eps} u_{\eps,h}} \cvwstar \rho_{0} u $ in $ L^{\infty}\big([0,T];L^2(\mathbb{R}^2)\big)$. 

\end{itemize}
Moreover $ \rho_0 $ and $  u  $ satisfy the following properties:
\begin{equation}
\label{con:u:rho}
\divh(u)\,=\,\divh(\rho_0 \, u )\,=\,0 \qquad\qquad \mbox{ and } \qquad\qquad  \forall\,t\geq0\,,\quad \rho_0(t)\,=\,\rho_0^{in}\,,
\end{equation}
where $\rho^{in}_0$ is the reference density state introduced in \eqref{hyp:rho_0}.
\end{prop}

\begin{proof}

The first three convergence properties are quite direct consequences of the energy inequality and have already been proved,
see \eqref{Conv:u}, \eqref{Cweak:conv} and \eqref{est:u_3}.

Regarding the fourth convergence property, from the energy inequality we immediately deduce that,
up to an extraction, $\overline{\rho_{\eps} u_{\eps,h}} \cvwstar k $ in $ L^{\infty}\big([0,T];L^2(\mathbb{R}^2)\big)$,
for some two-dimensional vector field $k$ belonging to that space. So, it remains to show that $ k = \rho_{0} u $. For this,
applying Corollary \ref{cor:1} with $ f_{\eps} = \rho_{\eps} $ and $s = p = +\infty $, we deduce that 
\begin{align*}
\|\overline{\rho_{\eps} u_{\eps,h}}\, -\, \overline{\rho_{\eps}} \;\overline{u_{\eps,h}}\|_{L^2_T(L^2(\bbR^2))}\, \leq\,C\,
\|\rho_{\eps} \|_{L^{\infty}_T(L^\infty(\Omega_\veps))}\,\ell_{\eps}\,\left(\overline{ \|D u_{\eps} \|^2_{L^2_T(L^2(\bbR^2))}}\right)^{\frac12}\,\leq\,C\,\ell_\veps\,,
\end{align*}
where we have also used inequality \eqref{est:rho_inf} and the bounds of Lemma \ref{Lemma:merc}. Hence,
it is enough to show that 
\[
\overline{\rho_{\eps}} \; \overline{u_{\eps,h}} \,\longrightarrow \, \rho_{0} \, u \qquad\qquad\quad \mbox{ in }\qquad \mathcal{D}'\big(\R_+\times\R^2\big)\,.
\]
Recall that, for any ball $ B_{R}\subset\R^2 $ of center $0$ and radius $R>0$, the inclusion $ H^1_0(B_{R})\,\hookrightarrow \,L^2(B_{R})$
is compact, in particular  $   L^2(B_{R})\,\hookrightarrow\, H^{-1}(B_{R}) $ is also compact.
Thus, arguing similarly as done to get \eqref{Cweak:conv}, we infer the strong convergence of $ \overline{\rho_{\eps}}$ towards $\rho_0$ in the space
$ L^q_T\big(H^{-1}(B_{R})\big)$, for any $1\leq q < + \infty $ and $ R>0 $.
On the other hand, for all test functions $\phi=\phi(t,x_h)$ belonging to $C^\infty_c\big(\R_+\times\R^2\big)$ and such that $\Supp\phi\subset [0,T[\,\times B_{R/2}$,
for some $T>0$ and $R>0$, one has that $\left(\oline{u_{\veps,h}}\,\phi\right)_\veps\Subset L^2_T\big(H^1_0(B_R)\big)$ and it weakly
converges in that space to $u\,\phi$. Combining these facts together, we finally deduce that
$ \overline{\rho_{\eps}} \; \overline{u_{\eps,h}}\, \longrightarrow \,\rho_{0} \, u $ in $\mathcal{D}'\big(\R_+ \times \bbR^2\big)$,
which in turn implies that $\overline{\rho_{\eps} u_{\eps,h}} \cvwstar \rho_{0} u  $ in $ L^{\infty}\big([0,T];L^2(\mathbb{R}^2)\big)$, as claimed.   

We now move to the proof of \eqref{con:u:rho}.
Notice that the property $\divh(u)=0$ has already been proved, keep in mind relation \eqref{eq:u-div-free} above.
Next, we recall that, for any $\veps>0$, the couple $ \big(\rho_{\eps}, u_{\eps}\big) $ satisfies equation \eqref{eps:sys} in a weak sense:
for any test function $\psi =\psi_\veps\in C^\infty_c\big(\R_+\times\oline{\Omega_\veps};\R^3\big)$ such that $\div(\psi)=0$ and $\psi\cdot n=0$
on $\d\Omega_\veps$, we have
\begin{align}
&\iint_{\R_+\times\Omega_\veps}\Big(-\rho_\veps\,u_\veps\cdot\d_t\psi\,-\,\rho_\veps\,u_\veps\otimes u_\veps:\nabla\psi\,+\,
\frac{1}{\veps}\,e_3\times\rho_\veps\,u_\veps\cdot\psi\Big)\dx\dt \label{eq:weak_eps} \\
&\qquad\quad
\,+\,\iint_{\R_+\times\Omega_\veps}\nabla u_\veps:\nabla\psi\,\dx\dt +\,2\,\alpha_\veps\iint_{\R_+\times\d\Omega_\veps}u_\veps\cdot\psi\,\dx_h\dt\,=\,
\int_{\Omega_\veps} m^{in}_\veps\cdot\psi(0,\cdot)\,\dx\,. \nonumber
\end{align}
Now, given a scalar function $ \vphi \in C^{\infty}_c\big(\R_+ \times \bbR^2\big) $, set
$ \psi \,=\,\big(\nabla^{\perp}_h \vphi ,0\big)$ and use $\veps\psi $ in \eqref{eq:weak_eps}: we get
\begin{align*}
\frac{1}{2\ell_{\eps}}\iint_{\R_+\times\Omega_\veps}\!\!\!e_3\times\rho_\veps u_\veps\cdot\psi\dx\dt&=
\frac{\eps}{2\ell_{\eps}}\iint_{\R_+\times\Omega_\veps}\!\!\!\Big(\rho_\veps u_\veps\cdot\d_t\psi+\rho_\veps u_\veps\otimes u_\veps:\nabla\psi
-\nabla u_\veps:\nabla\psi\Big)\dx\dt \\
&\qquad
-\,\alpha_\veps\frac{\eps}{\ell_{\eps}}\iint_{\R_+\times\d\Omega_\veps}u_\veps\cdot\psi\,\dx_h\dt\,+\,\frac{\eps}{2\ell_{\eps}}
\int_{\Omega_\veps} m^{in}_\veps\cdot\psi(0,\cdot)\,\dx\,.
\end{align*}  
Notice that, in view of the uniform bounds of Subsection \ref{ss:finite-en}, the right-hand side of the previous relation converges to $0$
when $\veps\ra0^+$. On the other hand, using the special structure of the test function $\psi= \big(\nabla_h^\perp\vphi,0\big)$, we gather that
\begin{align}
\label{eq:Coriolis}
\frac{1}{2\ell_{\eps}}\iint_{\R_+\times\Omega_\veps}e_3\times\rho_\veps u_\veps\cdot\psi\,\dx\dt\,&=\,
\iint_{\bbR^+\times\R^2}\overline{\rho_{\eps} u_{\eps,h}^\perp} \cdot \nabla^{\perp}_h \vphi\, \dd x_h\dt \\
&=\,\iint_{\bbR^+\times\R^2}\overline{\rho_{\eps} u_{\eps,h}} \cdot \nabla_h \vphi\,\dd x_h\dt\,. \nonumber
%
\end{align}
Taking the limit $\veps\ra0^+$ and employing the previous convergence properties, in the end we find
\begin{equation} \label{eq:div-weak}
\iint_{\bbR^+\times\R^2} \rho_{0}\,  u_h \cdot \nabla_h \vphi\,\dd x_h\dt\, =\, 0\,.
\end{equation}
By arbitrariness of the function $\vphi$, we have thus proved the first condition appearing in \eqref{con:u:rho}.

To show that also the second condition holds true, we start by recalling that $ \rho_{\eps} $ solves the first equation in \eqref{eps:sys} in the weak
sense: for any $\vphi=\vphi_\veps\in C^\infty_c\big(\R_+\times\oline{\Omega_\veps}\big)$, we have
\[ 
-\iint_{\R_+\times\Omega_\veps} \rho_\veps\,\big(\d_t\vphi\,+\,u_\veps\cdot\nabla\vphi\big)\,\dx\dt\,=\,
\int_{\Omega_\veps}\rho^{in}_\veps\,\varphi(0,\cdot)\,\dx \,.
\]
Choosing now $ \varphi= \vphi(t,x_h)  \in C^{\infty}_{c}\big(\R_+\times \bbR^2\big) $ and using again the previous convergence properties
to pass to the limit for $\veps\ra0^+$, from the previous relation we deduce
\begin{equation*}
-\iint_{\R_+\times\R^2} \Big(\rho_0\,\d_t\vphi\,+\,\rho_0\,u\cdot\nabla_h\vphi\Big)\,\dd x_h\dt\,=\,
\int_{\R^2}\rho^{in}_0\,\varphi(0,\cdot)\,\dd x_h \,.
\end{equation*} 
At this point, we use the previous property \eqref{eq:div-weak} to complete the proof.
\end{proof}

With Lemma \ref{lemma:cons:bound:data} and Proposition \ref{con:all:var} at hand, we can already complete the asymptotic study in
the situations considered in Theorems \ref{Theo:2} and \ref{Theo:3}. This is the matter of the next subsection.

\subsection{Asymptotics in some special cases} \label{ss:easy}

Here we take advantage of the analysis of the previous section to prove the convergence in some special situations, namely when the quotient
$\alpha_\veps/\ell_\veps$ diverges to $+\infty$ on the one hand, and on the other hand when the initial reference density $\rho_0^{in}$
is constant (say equal to $1$).

Those cases are the one treated in Theorems \ref{Theo:2} and \ref{Theo:3},
respectively; correspondingly, the proofs of those statements will be carried out in Paragraphs \ref{sss:teo2} and \ref{sss:teo3} below.

\subsubsection{Proof of Theorem \ref{Theo:2}} \label{sss:teo2}

We start by proving Theorem \ref{Theo:2}. This can be done with the only help of the uniform estimates of Subsection \ref{ss:finite-en}.

\begin{proof}[Proof of Theorem \ref{Theo:2}]
The claimed convergences for the density functions follow from Proposition \ref{con:all:var}. So, it remains only to show the convergence to
zero of the vertical average of the velocity fields. We start by noticing that
\begin{align}
\overline{u_{\eps}}(t,x_h)\, &=  \, \overline{u_{\eps}}(t,x_h) - u_{\eps}(t,x_h,-\ell_\veps) + u_{\eps}(t,x_h,-\ell_\veps) \label{eq:u_bar} \\
&=\, \frac{1}{2 \ell_{\eps}} \int_{-\ell_{\eps}}^{\ell_{\eps}} \int_{-\ell_{\eps}}^z \partial_3 u_{\eps}(t,x_h,\z)\,\dd\z\,\dd z +
u_{\eps}(t,x_h,-\ell_{\eps})\,.  \nonumber
\end{align}
Observe that, by the Minkowski inequality (see \tsl{e.g.} Proposition 1.3 of \cite{B-C-D}), we can estimate
\begin{align*}
\left\|\frac{1}{2\ell_{\eps}} \int_{-\ell_{\eps}}^{\ell_{\eps}} \int_{-\ell_{\eps}}^z \partial_3 u_{\eps}(t,x_h,\z)\,\dd\z\,\dd z\right\|_{L^2(\bbR^2)}
\,&\leq\,\frac{1}{2\ell_{\eps}} \int_{-\ell_{\eps}}^{\ell_{\eps}}
\left\|\int_{-\ell_{\eps}}^z \left|\partial_3 u_{\eps}(t,x_h,\z)\right|\,\dd\z\right\|_{L^2(\bbR^2)}\,\dd z \\
&\leq2\,\ell_\veps\,
\left\|\frac{1}{2\ell_{\eps}}\int_{-\ell_{\eps}}^{\ell_\veps} \left|\partial_3 u_{\eps}(t,x_h,\z)\right|\,\dd\z\right\|_{L^2(\bbR^2)}\,.
\end{align*}
Therefore, applying Jensen's inequality to this last integral, from \eqref{eq:u_bar} we deduce
\begin{align*}
\|\overline{u_{\eps}}(t) \|_{L^2(\bbR^2)}\, 
&\leq\, 2\,\ell_{\eps}\,\left(\overline{\left\|\partial_3 u_{\eps}(t,\cdot,x_3)\right\|_{L^2(\bbR^2)}^2}\right)^{1/2}\,+\,
\left\| u_{\eps}(t,\cdot,-\ell_{\eps})\right\|_{L^2(\bbR^2)}\,.
\end{align*}	
After taking the $ L^2$-norm in time over the interval $[0,T]$, the energy inequality \eqref{est:unif-en} implies 
\begin{align*}
\|\overline{u_{\eps}} \|_{L^2_T(L^2(\bbR^2))}
	&\lesssim \, \ell_{\eps}\, \left(\overline{\left\| \partial_3 u_{\eps}\right\|_{L^2_T(L^2(\bbR^2))}^2}\right)^{1/2}\, +\,
	\left\| u_{\eps}(\cdot,\cdot,-\ell_{\eps})\right\|_{L^2_T(L^2(\bbR^2))}\;\lesssim\; \ell_{\eps}  + \sqrt{\frac{\ell_{\eps}}{\alpha_{\eps}}}\,,
\end{align*}	
which converges to zero, because both $ \ell_{\eps} $ and $ \ell_{\eps}/ \alpha_{\eps}$ converge to zero by hypothesis. 
\end{proof}

\subsubsection{Proof of Theorem \ref{Theo:3}} \label{sss:teo3}

We now turn our attention to the proof of Theorem \ref{Theo:3}, devoted to the quasi-homogeneous case, \tsl{i.e.} the case when the reference
density profile $\rho_0^{in}$ equals a constant, say $1$ for simplicity.

\begin{proof}[Proof of Theorem \ref{Theo:3}]
An advantage of assuming that $ \rho_0^{in} = 1 $ is that the quantity
\[
 r_\veps\,:=\,\frac{1}{\veps}\,\big(\rho_\veps-1\big)
\]
is also transported by the velocity field $ u_{\eps} $, namely
$ r_{\eps} $ satisfies $ \partial_t r_{\eps} + \div(r_{\eps} u_{\eps}) = 0  $, related to the initial datum $ r_{\eps}^{in} = \eps^{-1}(\rho^{in}_{\eps}-1)$.
By assumption, $\big(\,\oline{r^{in}_\veps}\,\big)_\veps\,\Subset\, L^2(\R^2)\cap L^\infty(\R^2)$.
Following the proof of Proposition \ref{con:all:var}, we thus obtain
the following convergence properties, for any $T>0$ fixed:
\begin{itemize}
\item $ \overline{r_{\eps}} \cvwstar r_0 $ in $ L^{\infty}\big([0,T]\times \bbR^2 \big)\,\cap\,L^\infty\big([0,T];L^2(\R^2)\big) $ and
$ \overline{r_{\eps}} \longrightarrow r_0 $ in $ C^0\big([0,T];L^p_w(K)\big) $ for any compact $ K \subset \bbR^2 $
and for any $1\leq p<+\infty$;



\item $ \overline{r_{\eps}\, u_{\eps,h}} \cvwstar r_0\, u $ in the space $ L^{2}\big([0,T] \times \mathbb{R}^2\big) $, 
\end{itemize}
for a suitable function $r_0$ belonging to $L^\infty\big(\R_+\times\R^2\big)\cap C^0_b\big(\R_+;L^2_w(\R^2)\big)$.
It goes without saying that the convergence properties for $u_\veps$'s stated in Lemma \ref{lemma:cons:bound:data} and Proposition \ref{con:all:var}
still hold true.

It remains to show that the couple $ \big(r_0, u \big) $  solves system \eqref{final:sys} in a weak sense. To begin with, recall that
each $ r_{\eps} $ satisfies the equation 
\begin{equation*}
\int_{\Omega_{\eps}} r_{\eps}^{in}\, \varphi(0,\cdot) \, \dx + \int_{0}^T \int_{\Omega_{\eps}} r_{\eps} \partial_t \varphi \, \dx \dt  + \int_0^T \int_{\Omega_{\eps}} r_{\eps} u_{\eps} \cdot 	\nabla \varphi \, \dx \dt = 0 
\end{equation*}
for any $ \varphi = \varphi_{\eps} \in C^{\infty}_c\big([0,T[\,\times \oline{\Omega_{\eps}}\big) $, for arbitrary $T>0$.
Notice that $ \varphi = \phi/(2\ell_{\eps}) $, for some $ \phi \in C^{\infty}_c\big([0,T[\,\times \bbR^2\big) $,
is an admissible test function for the above weak fomulation; using it in the previous equality, in particular we easily deduce 
\begin{equation}
\label{equ:r:esp:bar}
\int_{\bbR^2} \overline{r_{\eps}^{in}}\, \phi(0,\cdot) \, \dx_h\, +\,\int_{0}^T \int_{\bbR^2} \overline{r_{\eps}}\,\partial_t \phi \, \dx_h\, \dt\,+\,
\int_0^T \int_{\bbR^2} \overline{r_{\eps} u_{\eps,h}} \cdot 	\nabla_h \phi \, \dx_h\, \dt\, =\, 0\,.
\end{equation} 
%
%
%
Therefore, using the convergence properties of listed at the beginning of the proof and the ones of Lemma \ref{lemma:cons:bound:data} and Proposition
\ref{con:all:var}, one deduces, in the limit for $\veps\ra0^+$, the equality
\begin{equation*}
\int_{\bbR^2} r^{in}_0\,\phi(0,\cdot) \, \dx_h \,+\,
\int_{0}^T \int_{\bbR^2} r_0 \partial_t \phi \, \dx_h\, \dt\,  +\,\int_0^T \int_{\bbR^2} r_0\, u \cdot \nabla_h \phi \, \dx_h\, \dt\, =\, 0\,. 
\end{equation*} 
This is exactly the weak formulation associated to the first equation in \eqref{final:sys}.

We will now pass to the limit in the momentum equation appearing in \eqref{eps:sys}. To do that, we rewrite its weak formulation
\eqref{eq:weak_eps} in an convenient way. 
Recall that, in \eqref{eq:weak_eps}, one has $\psi=\psi_\veps\in C^\infty_c\big(\R_+\times\oline{\Omega_\veps};\R^3\big)$,
with $\div(\psi)=0$ and $\big(\psi\cdot n\big)_{|\d\Omega_\veps}=0$.
Using now the equality $\rho_\veps\,=\,1+\veps r_\veps$,
%
we can rewrite the momentum equation as follows:
\begin{align}
&\iint_{[0,T]\times\Omega_\veps}\left(-u_\veps\cdot\d_t\psi\,-\,u_\veps\otimes u_\veps:\nabla\psi\,+\,
	\frac{1}{\veps}\,e_3\times u_\veps\cdot\psi\,+\, e_3 \times r_{\eps} u_{\eps} \cdot \psi \right)\dx\dt  \label{equ:to:lim:laal} \\
&\qquad\qquad\qquad\qquad\qquad  +\,\iint_{[0,T]\times\Omega_\veps}\nabla u_\veps:\nabla\psi\,\dx\dt\,+\,
2\,\alpha_\veps\iint_{[0,T]\times\d\Omega_\veps}u_\veps\cdot\psi\,\dx_h\dt  \nonumber \\
&\qquad\qquad \qquad\qquad\qquad\qquad \qquad\qquad \qquad\qquad  \qquad\qquad =\,
	\int_{\Omega_\veps} u^{in}_\veps\cdot\psi(0,\cdot)\,\dx\,+\,\eps\,\langle R_{\eps}, \psi \rangle\,, \nonumber
\end{align}
where the time $T>0$ is such that $\Supp\psi\subset[0,T]\times\oline{\Omega_\veps}$ and the brackets $\lan\cdot,\cdot\ran$ denote the duality product
of $\mc D'\times C^\infty_c$. The term $R_\veps$ is a small remainder, which will be treated later on and whose explicit expression reads
\begin{equation*}
\langle R_{\eps}, \psi \rangle \,=\,\iint_{[0,T]\times\Omega_\veps}\Big(r_\veps\,u_\veps\cdot\d_t\psi \, +\,r_\veps\,u_\veps\otimes u_\veps:\nabla\psi\, \Big)
\dx\dt\, +\, \int_{\Omega_{\eps}} r_{\eps}^{in}\, u_{\eps}^{in}\cdot \psi(0,\cdot)\, \dx\,. 
\end{equation*}

As done in Subsection \ref{ss:constraints}, see relation \eqref{eq:weak_eps} and below, we are going to use special test functions in \eqref{equ:to:lim:laal}.
More precisely, fix a smooth $\vphi\in C^\infty_c\big([0,T[\,\times\R^2\big)$, for some $T>0$,
and define $\psi\,=\,\big(\nabla_h^\perp\vphi,0\big)/(2\ell_\veps)$.
Remark that these special test functions make the singular part of the Coriolis term vanishing. Indeed,
using that $ e_{3} \times u_{\eps} = \big(u_{\eps,h}^{\perp},0\big)$ and $ \psi $ does not depend on the vertical variable,  we have
\begin{align*}
\frac{1}{\eps} \iint_{[0,T]\times \Omega_{\eps}}  e_{3}\times u_{\eps}\cdot \frac{\psi}{2\ell_{\eps}} \, \dx \dt
&= \, \frac{1}{\eps} \iint_{[0,T]\times \bbR^2}   \overline{u_{\eps, h}^{\perp }}\cdot \nabla^{\perp}_h \vphi \, \dx_h \dt \\
&=\,\frac{1}{\eps} \iint_{[0,T]\times \bbR^2}   \overline{u_{\eps, h}}\cdot \nabla_h \vphi \, \dx_h \dt\,,
 \end{align*}
 which of course vanishes, because
\[
\div\big(u_\veps\big)\,=\,0\qquad\qquad \Longrightarrow\qquad\qquad \divh\left(\oline{u_{\veps,h}}\right)\,=\,0\,,
\]
%
%
%
where we averaged the divergence-free condition with respect ot the vertical variable and we used that the vertical component $ u_{\eps,3} $
of the velocity field is zero on $ \partial \Omega_{\eps} $. We have thus proved that
\begin{equation}
\label{sing:term:cor:zero}
\frac{1}{\eps} \iint_{[0,T]\times \Omega_{\eps}}  e_{3}\times u_{\eps}\cdot \frac{\psi}{2\ell_{\eps}} \, \dx \dt = \frac{1}{\eps} \iint_{[0,T]\times \bbR^2}   \overline{u_{\eps, h}^{\perp }}\cdot \nabla_h^\perp\vphi \, \dx_h \dt\, =\, 0\,.
\end{equation}

Therefore, with this choice of the test function $\psi$, equation \eqref{equ:to:lim:laal} becomes
\begin{align}
&\iint_{[0,T]\times\bbR^2}\Big(-\overline{u_\veps}\cdot\d_t\psi\,-\,\overline{u_{\veps, h}\otimes u_{\veps, h}}:\nabla_h\psi_h\,+\,
\overline{r_{\eps}\, u_{\eps, h}^{\perp}} \cdot \psi_h \Big) \dx_h \dt \label{equ:to:limit:aka}  \\
&\qquad\qquad\quad
\,+\,\iint_{[0,T]\times\bbR^2}\nabla_h \overline{u_{\veps, h}}:\nabla_h \psi_h\,\dx_h\dt\,+\
\frac{\alpha_\veps}{\ell_{\eps}}\iint_{[0,T]\times\d\Omega_\veps}u_{\veps, h}\cdot\psi_h\,\dx_h\dt \nonumber \\
&\qquad\qquad\qquad\qquad\qquad\qquad\qquad\qquad\qquad\qquad=\,
\int_{\bbR^2} \overline{u^{in}_{\veps h}}\cdot\psi_h(0,\cdot)\,\dx_h \, + \eps \left\langle R_{\eps}, \frac{\psi}{2\ell_{\eps}} \right\rangle.  \nonumber 
\end{align}
With the convergences stated in the beginning of the proof, we can pass to the limit in any term appearing in \eqref{equ:to:limit:aka},
except the following three terms:
\begin{gather}
\iint_{[0,T]\times\bbR^2} \overline{u_{\veps, h}\otimes u_{\veps, h}}:\nabla_h\psi_h \, \dx_h \dt \label{term:nonlin} \\
\frac{\alpha_\veps}{\ell_{\eps}}\iint_{[0,T]\times\d\Omega_\veps}u_{\veps, h}\cdot\psi_h\,\dx_h\dt \label{term:bound}\\
\eps \left\langle R_{\eps}, \frac{\psi}{2\ell_{\eps}} \right\rangle. \label{term:rem}
\end{gather} 
We will show how to pass to the limit in the above terms separately.

Let start with \eqref{term:bound}.
Recall that, by assumption, $\lam_\veps\,:=\, \alpha_{\eps}/\ell_{\eps} \to \lambda\geq0 $. If $ \lambda > 0 $, Lemma \ref{lemma:cons:bound:data}
ensures that 
%
%
%
\begin{equation}
	\label{conv:bound:terms:aka}
	\frac{\alpha_\veps}{\ell_{\eps}}\iint_{[0,T]\times\d\Omega_\veps}u_{\veps, h}\cdot\psi_h\,\dx_h\dt\, \longrightarrow\,
	\lambda \iint_{[0,T]\times \bbR^2}\left(u^+ \,+\, u^-\right)\cdot \psi_h \, \dx_h \dt\,.
\end{equation}	
If $ \lambda = 0 $, instead, the H\"older inequality and Lemma \ref{Lemma:merc} impliy 
\begin{align*}
\left| \frac{\alpha_\veps}{\ell_{\eps}}\iint_{[0,T]\times\d\Omega_\veps}u_{\veps, h}\cdot\psi_h\,\dx_h\dt \right| &\leq
\, \sqrt{\frac{\alpha_\veps}{\ell_{\eps}}}\sqrt{\frac{\alpha_\veps}{\ell_{\eps}}}\left\|u_{\eps}\right\|_{L^2([0,T]\times \partial \Omega_{\eps})}\,
\|\psi\|_{L^2([0,T]\times \bbR^2)} \\
&\leq \, C\,\sqrt{\lam_\veps}\,\|\psi\|_{L^2([0,T]\times \bbR^2)}  \ \longrightarrow \ 0\,.
\end{align*} 

Next, we show that $ \eqref{term:rem} $ is a remainder, in the sense that it converges to zero when $\veps\ra0^+$.
More precisely, we are going to prove the following estimate:
\begin{equation}
\label{convergence:remainder}
\eps \left| \left\langle R_{\eps}, \frac{\psi}{2\ell_{\eps}} \right\rangle \right|\,\lesssim\, \eps\,
\left(\|\d_t\psi\|_{L^2([0,T]\times \bbR^2)} + \|\psi(0,\cdot)\|_{L^2( \bbR^2)} + \|\nabla \psi\|_{L^{\infty}([0,T]\times \bbR^2)} \right)\,.
\end{equation}
To prove \eqref{convergence:remainder}, we recall that $ \psi $ does not depend on the third variable to write
\begin{equation*}                                  
\left\langle R_{\eps}, \frac{\psi}{2\ell_{\eps}} \right\rangle = \iint_{[0,T]\times\bbR^2}\Big(\overline{r_\veps\,u_{\veps, h}}\cdot\d_t\psi_h \, +\,
\overline{r_\veps\,u_{\veps, h}\otimes u_{\veps, h}}:\nabla_h \psi_h\,\Big)\dx\dt+\int_{\bbR^2} \overline{r_{\eps}^{in} u_{\eps, h}^{in}}\cdot\psi_h\dx_h\,. 
\end{equation*}
First of all, it is easy to see that
\begin{align}
\left\|\overline{r_{\eps} u_{\eps}} \right\|_{L^{2}([0,T]\times \bbR^2)}^2\,
%
\leq\, \overline{\|r_\veps\, u_{\veps, h}\|_{L^2([0,T] \times \bbR^2)}^2}\,\leq\, \| r_{\eps} \|_{L^{\infty}([0,T]\times \Omega_{\eps})}^2 \overline{\|u_{\eps}\|_{L^2([0,T]\times \Omega_{\eps})}^2}\,\leq\,C\,,   \label{la:1}
\end{align}
owing to the fact that $ \| r_{\eps} \|_{L^{\infty}([0,T]\times \Omega_{\eps})}= \| r_{\eps}^{in}\|_{L^{\infty }(\Omega_{\eps})} \leq C $ and to
the bounds of Lemma \ref{Lemma:merc}. Similarly, we have
\begin{equation}
\label{la:2}
\left\|\overline{r_{\eps}^{in} u_{\eps}^{in}}\right\|_{L^2(\bbR^2)}^2 \leq \|r_{\eps}^{in}\|_{L^{\infty}(\Omega_{\eps})}\,
\overline{\left\|u_{\eps}^{in}\right\|_{L^2(\Omega_{\eps})}^2} \leq C\,.
\end{equation}
Regarding the last term appearing in the definition of $R_\veps$, we can write
\begin{equation*}
\overline{r_\veps\, u_{\veps, h}\otimes u_{\veps, h}} = \overline{r_\veps\, u_{\veps, h}}\otimes \overline{u_{\veps, h}} + 
\overline{r_\veps\, u_{\veps, h}\otimes (u_{\veps, h}-\overline{u_{\eps, h}})}\,
\end{equation*}
from which we deduce that 
\begin{align*}
\left\|\overline{r_\veps\, u_{\veps, h}\otimes u_{\veps, h}} \right\|_{L^1([0,T] \times \bbR^2)}\, &\leq \,
\left\| \overline{r_\veps\, u_{\veps, h}}\otimes \overline{u_{\veps, h}}\right\|_{L^1([0,T] \times \bbR^2)} +
\left\| \overline{r_\veps\, u_{\veps, h}\otimes \big(u_{\veps, h}-\overline{u_{\eps, h}}\big)}\right\|_{L^1([0,T] \times \bbR^2)} \\
&\lesssim\,\left\| \overline{r_\veps\, u_{\veps, h}}\right\|_{L^2([0,T] \times \bbR^2)}\,
\left\| \overline{u_{\veps,h}} \right\|_{L^2([0,T] \times \bbR^2)}  \\
&\qquad\qquad\qquad +\,\ell_{\eps} \left(\overline{\left\|r_\veps\, u_{\veps, h}\right\|_{L^2([0,T] \times \bbR^2)}^2}\right)^{1/2}\,
\left(\overline{\left\|Du_{\eps}\right\|_{L^2([0,T] \times \bbR^2)}^2}\right)^{1/2}\,,
\end{align*}
thanks to the application of the H\"older inequality and Corollary \ref{cor:1}.
Using \eqref{la:1} and Lemma \ref{Lemma:merc} finally yields the bound
\begin{equation}
\label{la:3}
\left\| \overline{r_\veps\, u_{\veps, h}\otimes u_{\veps, h}} \right\|_{L^1([0,T] \times \bbR^2)}\, \leq\, C\,. 	
\end{equation}
independent of $ \eps $. Putting \eqref{la:1}, \eqref{la:2} and \eqref{la:3} together, in turn we deduce \eqref{convergence:remainder}, as claimed.

It remains us to study the convergence of the integral in \eqref{term:nonlin}. As before, we write
\begin{equation*}
	\overline{u_{\eps, h} \otimes u_{\eps, h}}\, =\, \overline{u_{\eps, h}} \otimes \overline{u_{\eps, h}}\, +\,
	\overline{u_{\eps, h} \otimes \big(u_{\eps, h} - \overline{u_{\eps\,  h} }\big)}\,.
\end{equation*}
Corollary \ref{cor:1} (where we take $ s = p = 2 $) and the bounds of Lemma \ref{Lemma:merc} imply that 
\begin{equation*}
\left\|\overline{u_{\eps, h} \otimes \big(u_{\eps, h} - \overline{u_{\eps  h} }\big)}\right\|_{L^1([0,T]\times \bbR^2)}\,\lesssim\,\ell_{\eps}
\left(\overline{\left\| u_{\veps, h}\right\|_{L^2([0,T] \times \bbR^2)}^2}\right)^{1/2}\,
\left(\overline{\|D u_{\eps}\|_{L^2([0,T] \times \bbR^2)}^2}\right)^{1/2}\,\longrightarrow\,0
\end{equation*}
in the limit $\veps\ra0^+$. Therefore, we infer that passing to the limit in \eqref{term:rem} is equivalent to passing to the limit in 
\begin{equation*}
\iint_{[0,T] \times \bbR^2} \overline{u_{\eps, h}} \otimes \overline{u_{\eps, h}} : \nabla_h \psi_h \, \dx_h \dt\,.
\end{equation*}
We now claim that
\begin{equation}
	\label{Goal}
\iint_{[0,T] \times \bbR^2} \overline{u_{\eps, h}} \otimes \overline{u_{\eps, h}} : \nabla_h \psi_h \, \dx_h \dt\,
\longrightarrow\, \iint_{[0,T] \times \bbR^2} u \otimes u : \nabla_h \psi_h \, \dx_h \dt\,.
\end{equation}
We postpone the proof of the previous convergence to the end of the argument, see Lemma \ref{l:goal} below. For the time being,
let us assume that \eqref{Goal} holds true and let us resume the proof of Theorem \ref{Theo:3}.

What is left is to pass to the limit for  $ \eps\ra0^+ $ in system \eqref{equ:to:limit:aka}. Recall that $\psi=\big(\nabla_h^\perp\vphi,0\big)$,
for some smooth compactly supported function $\vphi$ depending only on the time and horizontal variables.
From the convergences presented at the beginning of the proof and the ones shown in \eqref{Goal}, \eqref{conv:bound:terms:aka} and
\eqref{convergence:remainder}, we deduce that, for $ \eps\ra0^+ $, equation \eqref{equ:to:limit:aka} converges to 
\begin{align*}
&\iint_{[0,T]\times\bbR^2}\Big(-u\cdot\d_t\psi_h\,-\, u \otimes u :\nabla_h\psi_h\, + \, r_0\, u^{\perp} \cdot \psi_h\,+\,\nabla_hu:\nabla_h \psi_h
\Big) \dx_h \dt \\
&\qquad\qquad\qquad\qquad\qquad\qquad\quad \,+\,
\lambda \iint_{[0,T]\times \bbR^2}\left(u^+\,+\,u^-\right)\cdot\psi_h\,\dx_h\dt\,=\,\int_{\bbR^2} u^{in} \cdot\psi_h(0,\cdot)\,\dx_h\,.
\end{align*}

Hence, to conclude the proof it remains only to show that $ u^+ = u^- = u$. For this, we write
\begin{equation}
	\label{identification:1}
u - u^{-} \, =\, u - \overline{u_{\eps, h}} \,+\, \overline{u_{\eps, h}}\, -\, u_{\eps, h}(\cdot,\cdot,-\ell_{\eps})\, +\,
u_{\eps,h}(\cdot,\cdot,-\ell_{\eps})\, -\, u^{-}\,.
\end{equation} 
The convergence \eqref{Conv:u} and Lemma \ref{lemma:cons:bound:data} ensure that
\begin{equation*}
	u - \overline{u_{\eps, h}} 	\cv 0 \qquad \text{ and }  \qquad u_{\eps, h}(\cdot,\cdot, -\ell_{\eps}) - u^{-}	\cv 0 \qquad\qquad\quad
	\text{ in } \qquad L^2\big([0,T];L^2(\bbR^2)\big)\,.
\end{equation*}
For the second term appearing in the right-hand side of \eqref{identification:1}, instead, we argue as in \eqref{eq:u_bar} and get
\begin{equation*}
\overline{u_{\eps }}-u_{\eps}(\cdot,\cdot,-\ell_{\eps})\, =\,
\frac{1}{2 \ell_{\eps}} \int_{-\ell_{\eps}}^{\ell_{\eps}} \int_{-\ell_{\eps}}^{z} \, \partial_3 u_{\eps}(\cdot,\cdot,z)\,\dd z\, \dx_3\,,
\end{equation*}
which implies (see the computation in the proof of Theorem \ref{Theo:2}) the bound
\begin{align*}
\left\|\overline{u}_{\eps} - u_{\eps}(\cdot,\cdot, -\ell_{\eps})\right\|_{L^2([0,T]\times \bbR^2)}^2\,&\lesssim\,
\ell_{\eps}^2\, \overline{\left\| \partial_3 u_{\eps} \right\|_{L^2([0,T] \times \bbR^2)}^2}\,.
\end{align*}
The uniform estimates of Lemma \ref{Lemma:merc} then guarantee the convergence to $0$ of the previous term, \tsl{i.e.} the strong
convergence of $ \overline{u}_{\eps} - u_{\eps}(\cdot,\cdot, -\ell_{\eps})\,\longrightarrow\,0 $ in $ L^{2}([0,T]\times \bbR^2) $.

In the end, we have shown that the right-hand side of \eqref{identification:1} weakly converges
to zero in $ L^2([0,T]\times \bbR^2) $. This implies that $ u = u^{-} $. The equality $ u = u^{+} $ follows similarly.
This concludes the proof of Theorem \ref{Theo:3}, provided we show the convergence property \eqref{Goal}.
\end{proof}

We now show the proof of \eqref{Goal}: this is the matter of the next lemma.

\begin{lemma} \label{l:goal}
Under the assumptions of Theorem \ref{Theo:3}, we have
\begin{equation}
	\label{L2:strong:conv}
\overline{u_{\eps, h}}\, \longrightarrow\, u \qquad\qquad \text{ strongly \ in }\quad  L^2_{\rm loc}\big(\R_+\times \bbR^2\big)\,.  
\end{equation}
In particular, the convergence \eqref{Goal} holds true: for any scalar function $\vphi\in C^\infty_c\big(\R_+\times\R^2\big)$,
after setting $\psi\,:=\,\big(\nabla_h^\perp\vphi,0\big)$, one has
\[
 \iint_{[0,T] \times \bbR^2} \overline{u_{\eps, h}} \otimes \overline{u_{\eps, h}} : \nabla_h \psi_h \, \dx_h \dt\,
\longrightarrow\, \iint_{[0,T] \times \bbR^2} u \otimes u : \nabla_h \psi_h \, \dx_h \dt
\]
in the limit $\veps\ra0^+$,
where the time $T>0$ is such that $\Supp\vphi\subset[0,T]\times\R^2$.
\end{lemma}

\begin{proof}
To prove \eqref{L2:strong:conv}, we apply the same strategy used to prove the $ L^2_{loc} $ strong convergence of the Galerkin approximate sequence in
the existence proof.
In the following, we only sketch the argument, and refer to the proof of Theorem 2.5 (part 2) in \cite{lions:book}.

Let start by introducing the Leray projector in two dimensions,
$ \mathcal{P}: L^2(\bbR^2;\R^2) \longrightarrow L^2_{\sigma}(\bbR^2;\R^2) $, which can be define as a Fourier multiplier by the formula
\begin{equation*}
\mathcal{F}\left(\mathcal{P} w \right)(\xi) = \mathcal{F}(w)(\xi) - \frac{\xi}{|\xi|^2} \xi \cdot \mathcal{F}(w)(\xi)
\end{equation*} 
for any two-dimensional vector field $ w $ in the Schwartz space $ \mathcal{S}(\bbR^2) $, where $ \mathcal{F} $ denotes the Fourier transform.

Next, we notice that \eqref{L2:strong:conv} is a consequence of the following fact:
\begin{equation}
	\label{conv:Progectjion}
\forall\,R > 0\,,\qquad \text{ it holds } \quad \iint_{[0,T] \times B_R} \left| \mathcal{P}(\overline{\rho_{\eps}u_{\eps, h}})\right|^2 \, \dx_h \dt \,
\longrightarrow\, \iint_{[0,T]\times B_R} |u|^2 \, \dx_h \dt\,,
\end{equation}
where $B_R$ is the ball in $\R^2$ of center $0$ and radius $R>0$ and the convergence is taken for $\veps\ra0^+$.
To see that \eqref{conv:Progectjion} implies \eqref{L2:strong:conv}, recall that $ \divh\left(\overline{u_{\eps, h}}\right) = 0 $, so
$\overline{u_{\eps, h}}\,=\,\mc P\big(\overline{u_{\eps, h}}\big)$. Hence, one can bound
\begin{align}
\iint_{[0,T]\times B_R} \left|\overline{u_{\eps, h}} \right|^2 \, \dx_h \dt\,&\leq\,
\iint_{[0,T]\times B_R} \left|\mathcal{P}\big(\overline{\rho_{\eps} u_{\eps, h}}\big) \right|^2 \, \dx_h \dt \label{useless} \\
%
&\qquad +\,
2\eps\iint_{[0,T]\times B_R}  \mathcal{P}\big(\overline{\rho_{\eps} u_{\eps, h}}\big)\cdot \mathcal{P}\big(\overline{r_{\eps} u_{\eps, h}}\big)\,\dx_h \dt
\nonumber \\
&\qquad\qquad\qquad\qquad + \eps^2 \iint_{[0,T]\times B_R} \left|\mathcal{P}(\overline{r_{\eps} u_{\eps, h}}) \right|^2 \, \dx_h \dt\,. \nonumber 
%
\end{align}
As $ \mathcal{P} $ is a bounded operator over $ L^2 $, from Proposition \ref{con:all:var} and \eqref{la:1} we infer that
\begin{align*}
\left\|\mathcal{P}(\overline{\rho_{\eps} u_{\eps, h}})\right\|_{L^2([0,T]\times \bbR^2)}\,&\leq\,
\left\|\overline{\rho_{\eps} u_{\eps, h}}\right\|_{L^2([0,T]\times \bbR^2)}\,  \leq\, C \\
\left\|\mathcal{P}(\overline{r_{\eps} u_{\eps, h}})\right\|_{L^2([0,T]\times \bbR^2)}\,&\leq\,
\left\|\overline{r_{\eps} u_{\eps, h}}\right\|_{L^2([0,T]\times \bbR^2)}\,  \leq\, C\,.
\end{align*}
Then, convergence \eqref{conv:Progectjion} and \eqref{useless} immediately imply that
\[
\forall\,R>0\,,\qquad\quad
\iint_{[0,T]\times B_R} \left|\overline{u_{\eps, h}} \right|^2 \, \dx_h \dt\,\longrightarrow\,
\iint_{[0,T]\times B_R} \left|u \right|^2 \, \dx_h \dt\qquad \mbox{ for }\quad \veps\ra0^+\,.
\]
The claimed property \eqref{L2:strong:conv} is then a consequence of the previous fact and the weak convergence $ \overline{u_{\eps h }} \cv u $ in
$ L^2([0,T]\times \bbR^2) $. 

So, let us prove \eqref{conv:Progectjion}. The starting point is again the weak fomulation of the momentum equation in
\eqref{eps:sys}, see relation \eqref{eq:weak_eps} above.
As in the previous proof, take $\psi\,=\,\big(\nabla_h^\perp\vphi,0\big)/(2\ell_\veps)$, for some smooth
$\vphi\in C^\infty_c\big([0,T[\,\times\R^2\big)$, with $T>0$.
Using the fact that such $ \psi $ does not depends on the vertical variable and taking advantage of the cancellation \eqref{sing:term:cor:zero},
we infer
\begin{align*}
&\iint_{[0,T]\times\bbR^2}\Big(-\overline{\rho_\veps\,u_{\veps, h}}\cdot\d_t\psi_h\,-\,
\overline{\rho_\veps\,u_{\veps, h}\otimes u_{\veps, h}}:\nabla_h\psi_h\,+\,
	\overline{r_\veps\,u^{\perp}_{\veps,h}}\cdot\psi_h \Big)\dx_h\dt \\
	&\quad
	\,+\,\iint_{[0,T]\times\bbR^2}\nabla_h \overline{u_{\veps, h}}:\nabla_h\psi_h\,\dx_h\dt\,+\,
\frac{\alpha_\veps}{\ell_{\eps}}\iint_{[0,T]\times\d\Omega_\veps}u_\veps\cdot\psi\,\dx_h\dt\,=\,
	\int_{\bbR^2}\overline{ m^{in}_{\veps, h}}\cdot\psi_h(0,\cdot)\,\dx_h\,.
\end{align*}
Notice that in the above equation only divergence-free functions are allowed. We deduce that
$\Big(\partial_t\mathcal{P}\big(\overline{\rho_{\eps} u_{\eps, h }}\big)\Big)_{\veps}$ is uniformly bounded in $ L^2\big([0,T]; H^{-s}(\bbR^2)\big) $,
for some $ s>0 $ large enough. Moreover $ \mathcal{P} $ is a bounded operator in $ L^2 $, from which we gather in particular that
$\Big(\mathcal{P}\big(\overline{\rho_{\eps} u_{\eps, h }}\big)\Big)_\veps$ is uniformly bounded in $ L^\infty\big([0,T]; L^2(\bbR^2)\big) $.
These two uniform boundedness properties together imply that, up to the extraction of a suitable subsequence, one has
\begin{equation}
	\label{strong:conv:pro:laal}
	\mathcal{P}\big(\overline{\rho_{\eps} u_{\eps, h}}\big)\, \longrightarrow\, u \qquad\qquad \text{ strongly \ in }\quad C^0\big([0,T];L^2_w(\bbR^2)\big)\,.
\end{equation}

Fix now $ R > 0 $ and denote by $ \chi_R $ the characteristic function of the ball $ B_R $, \tsl{i.e.} the function such that
$ \chi_R(x) =1 $ if $ |x| \leq R $ and $ 0 $ elsewhere. Let $ \eta_{n}(\cdot) = \eta(\cdot/n)\,n^{-2} $ a convolution kernel, with
$ \eta \in C^{\infty}_c\big(B_2\big) $, $ 0 \leq \eta \leq 1 $, $ \eta $ radially symmetric and of integral $ 1 $. 
Then \eqref{strong:conv:pro:laal} implies that, for any fixed $R>0$ and $n\in\N\setminus{0}$, one has
$ \big\{\chi_R \mathcal{P}\big(\overline{\rho_{\eps} u_{\eps,h}}\big)\big\} \star \eta_{n}\, \longrightarrow\, \{\chi_R u \} \star \eta_{n} $ in
$ C^0\big([0,T];L^2(\bbR^2)\big)$ when $\veps\ra0^+$,
where the symbol $\star$ denotes the convolution with respect to the space variable.
This allows us to deduce 
\begin{align}
\label{conv:regul:alala}
\iint_{[0,T]\times \bbR^2} \chi_R \mathcal{P}\big(\overline{\rho_{\eps} u_{\eps,h}}\big)\,
\Big(\mathcal{P}\big(\overline{\rho_{\eps} u_{\eps,h}}\big)\star\eta_n\Big)\, \dx_h \dt\, \longrightarrow\,
\iint_{[0,T]\times \bbR^2} \chi_R u \,\Big(u \star \eta_n\Big) \, \dx_h \dt
\end{align}
when $\veps\ra0^+$, for any fixed $R$ and $ n $.

We now claim that the property \eqref{conv:Progectjion} follows from the above convergence \eqref{conv:regul:alala}.
The proof of this implication will conclude the proof of Lemma \ref{l:goal}.

To see that \eqref{conv:regul:alala} implies \eqref{conv:Progectjion}, we notice that
\begin{equation}
	\label{com:leray}
\iint_{[0,T]\times \bbR^2} \chi_R \mathcal{P}\big(\overline{\rho_{\eps} u_{\eps, h }}\big) \cdot \mathcal{P}\big(\overline{\rho_{\eps} u_{\eps, h }}\big)
\, \dx_h \dt\, =\,\iint_{[0,T]\times \bbR^2} \mathcal{P}\Big(\chi_R \mathcal{P}\big(\overline{\rho_{\eps} u_{\eps, h }}\big)\Big) \cdot
\overline{\rho_{\eps} u_{\eps, h }} \, \dx_h \dt\,. 
\end{equation}
Next, we write the decompositon
\begin{align}
\label{decomp:in:A}
\overline{\rho_{\eps} u_{\eps, h}}\, &= \, 
%
A_1 \,+\, A_2\,+\,A_3\, +\, \{\overline{\rho_{\eps} u_{\eps, h}}\}\star \eta_{n}\,,
\end{align}
where we have defined
\begin{align*}
A_1:=\overline{\rho_{\eps} u_{\eps, h}}  - \overline{\rho_{\eps}} \, \overline{ u_{\eps, h}}\,,\quad
A_2:=\overline{\rho_{\eps}} \, \overline{ u_{\eps, h}} - \big\{\overline{\rho_{\eps}} \, \overline{ u_{\eps,h}}\big\}\star \eta_{n}\,,\quad
A_3:=\big\{\overline{\rho_{\eps}} \, \overline{ u_{\eps, h}}\big\}\star \eta_{n} - \big\{\overline{\rho_{\eps} u_{\eps, h}}\big\}\star\eta_n\,.
\end{align*}
In addition, let us denote
\begin{equation*}
	A_4\,:=\, u  - u \star \eta_{n}\,.
\end{equation*}
Using \eqref{com:leray}, \eqref{decomp:in:A} and the fact that the Leray projector $\mc P$ commutes with the convolution by $ \eta_{n} $, we get the equality
\begin{align*}
&\iint_{[0,T] \times B_R}\Big( \left|\mathcal{P}\big(\overline{\rho_{\eps} u_{\eps, h}}\big) \right|^2 - |u|^2\Big) \, \dx_h \dt \\
&\qquad\qquad=\,
\iint_{[0,T]\times \bbR^2} \mathcal{P}\Big(\chi_{R} \mathcal{P}\big(\overline{\rho_{\eps} u_{\eps,h}}\big)\Big)\cdot\big(A_1+A_2+A_3+A_4\big)\,\dx_h \dt \\
&\qquad\qquad\qquad\qquad\qquad
+\iint_{[0,T]\times \bbR^2} \Bigg( \chi_R \mathcal{P}\big(\overline{\rho_{\eps} u_{\eps,h}}\big)\,
\Big(\mathcal{P}\big(\overline{\rho_{\eps} u_{\eps,h}}\big)\star \eta_n \Big) - \chi_R u\, \Big(u \star \eta_n \Big) \Bigg)\, \dx_h \dt\,.
\end{align*}
Owing to \eqref{conv:regul:alala}, the last term of the above inequality converge to zero when $\veps\ra0^+$, for any fixed $ n $ and $R>0$.
Therefore, it remains us to show that 
 \begin{equation} \label{eq:conv_to-prove}
\lim_{R\ra+\infty}\; \limsup_{ n \to +\infty}\; \limsup_{\eps \to 0^+}
\left|\iint_{[0,T]\times \bbR^2} \mathcal{P}\Big(\chi_{R}\mathcal{P}\big(\overline{\rho_{\eps} u_{\eps, h}}\big)\Big)\cdot
\big(A_1+A_2+A_3+A_4\big)\,\dx_h\dt\right|\, =\, 0\,.
 \end{equation}
We start by showing that the terms involving $ A_1 $ and $ A_3 $ converge to zero as $ \eps $ approaches zero.  
For this, recall that the family $\Big(\mathcal{P}\big(\overline{\rho_{\eps} u_{\eps,h}}\big)\Big)_\veps $ is uniformly bounded in
$ L^{\infty}\big([0,T];L^2(\bbR^2)\big) $ and that $ \mathcal{P} $ is a bounded operator over $ L^2 $. Hence,
we deduce that also $\left( \mathcal{P}\Big(\chi_R \mathcal{P}\big(\overline{\rho_{\eps} u_{\eps,h}}\big)\Big)\right)_\veps$
is uniformly bounded in the same space. This implies that
\begin{align*}
&\Bigg|\iint_{[0,T]\times \bbR^2} \mathcal{P}\Big(\chi_{R} \mathcal{P}\big(\overline{\rho_{\eps} u_{\eps, h}}\big) \Big)\cdot
\big(A_1+A_3\big) \, \dx_h \dt  \Bigg| \\
&\qquad\qquad\qquad\qquad\qquad\qquad\qquad\qquad
\lesssim\,\left\|\overline{\rho_{\eps} u_{\eps,h}}\right\|_{L^2_T(L^2(\bbR^2))}\,
\left\|\overline{\rho_{\eps} u_{\eps, h}} - \overline{\rho_{\eps}} \, \overline{ u_{\eps, h}}\right\|_{L^2_T(L^2(\bbR^2))}\;\lesssim\;\ell_\veps\,,
%
\end{align*}
where we have used Corollary \ref{cor:1} and the bounds of Lemma \ref{Lemma:merc} in the last inequality. So, those terms satisfy
the convergence property stated in \eqref{eq:conv_to-prove}.

Next, we consider the term associated to $ A_4 $: we have
\begin{equation*}
\Bigg| \iint_{[0,T]\times \bbR^2} \mathcal{P}\Big(\chi_{R} \mathcal{P}\big(\overline{\rho_{\eps} u_{\eps,h}}\big)\Big)\cdot
\left(u-u \star \eta_n\right) \, \dx_h \dt\Bigg|\,\lesssim
\left\|\overline{\rho_{\eps} u_{\eps,h}}\right\|_{L^2_T(L^2(\bbR^2))}\,\left\|u - u \star \eta_n\right\|_{L^2_T(L^2(\bbR^2))}\,,
\end{equation*}
which obviously converges to zero when $n\ra+\infty$.

We are thus left with the term involving $A_2$. In \cite{lions:book} (see the proof of Part 2 of Theorem 2.5 therein, in particular Step 1 of that proof),
this is the most difficult term; however, here we have a simplification at our disposal, namely the fact that $\rho_\veps$ is very close to $1$.
Thus, following the idea of \cite{lions:book}, we are going to show that
%
\begin{equation*}
\sup_{\veps\in\,]0,1]}\left|\iint_{[0,T]\times \bbR^2}\mathcal{P}\Big(\chi_{R} \mathcal{P}\big(\overline{\rho_{\eps} u_{\eps,h}}\big)\Big)\cdot
\Big(\overline{\rho_{\eps}} \; \overline{ u_{\eps, h}} - \big\{\overline{\rho_{\eps}} \, \overline{ u_{\eps,h}}\big\}\star \eta_{n}\Big)\,
\dx_h \dt\right|\,\lesssim\,\de(n)\,,
\end{equation*}
for a positive function $\de(n)$ verifying $\de(n)\longrightarrow0^+$ when $n\ra+\infty$.
Recall that the sequence
$\left( \mathcal{P}\Big(\chi_R \mathcal{P}\big(\overline{\rho_{\eps} u_{\eps,h}}\big)\Big)\right)_\veps$ is uniformly bounded in
$ L^{\infty}\big([0,T];L^2(\bbR^2)\big) $.
Now, using the decomposition $\rho_\veps\,=\,1\,+\,\veps\,r_\veps$, the bounds of Lemma \ref{Lemma:merc} and the ones for $\big(\oline{r_\veps}\big)_\veps$
(recall the beginning of the proof of Theorem \ref{Theo:3}),
we see that it is enough to show that 
\begin{equation}
	\label{111}
\lim_{n \to + \infty} \sup_{\eps\in\,]0,1]}
\left\|\overline{u_{\eps, h}}\,-\,\big(\overline{ u_{\eps, h}}\star\eta_{n}\big)\right\|_{L^2_T(L^2(\bbR^2))}\,=\,0\,.
\end{equation}
The previous property is quite standard to get. Using Taylor formula, we can write
\begin{align*}
\overline{u_{\eps, h}}(t,x_h)\,-\,\big(\overline{ u_{\eps, h}}\star\eta_{n}\big)(t,x_h)\,=\,
\iint_{[0,1]\times\R^2}D_h\oline{u_{\veps,h}}(t,x_h+\s z_h)\cdot z\,\eta_n(z_h)\,\dd z_h\dd\s\,,
\end{align*}
from which we deduce that
\begin{align*}
\left\|\overline{u_{\eps, h}}(t)\,-\,\big(\overline{ u_{\eps, h}}\star\eta_{n}\big)(t)\right\|_{L^2(\R^2)}\,&\leq\,
\frac{1}{n}\iint_{[0,1]\times\R^2}\left\|D_h\oline{u_{\veps,h}}(t,\cdot+\s z_h)\right\|_{L^2}\,n\,\left|z_h\right|\,\eta_n(z_h)\,\dd z_h\dd\s \\
&\leq\,\frac{1}{n}\,\left\|D_h\oline{u_{\veps,h}}(t)\right\|_{L^2(\R^2)}\,\big\||\cdot|\,\eta\big\|_{L^1(\R^2)}\,,
\end{align*}
where we have also used the fact that the norm of the integral is less than or equal to the integral of the norm in the first inequality, and the
translation invariance of the Lebesgue measure in the second inequality. At this point, an integration in time and the use of the uniform bounds of Lemma
\ref{Lemma:merc} yield the sought property \eqref{111}.

In the end, we have concluded the proof of the lemma.
\end{proof}

\section{Proof of Theorem \ref{Main:Theo}} \label{s:proof-main}

In the present section, we complete the proof of Theorem \ref{Main:Theo}, devoted to the general case where the reference density profile $\rho_0^{in}$
is non-constant. We will borrow most of the arguments from \cite{F:G}, devoted to the two-dimensional framework. Here, the supplementary difficulty
is to handle the $3$-D geometry of the thin domain $\Omega_\veps$.

\subsection{Manipulation of the system}

In this subsection we rewrite system \eqref{eps:sys} in a more convenient way for passing to the limit for $ \eps\ra0^+ $.
To begin with, we introduce the new quantity
\[
\sigma_\veps\,:=\,\frac{1}{\veps}\,\big(\rho_\veps\,-\,\rho_0^{in}\big)\,.
\]
Of course, the family $\big(\s_\veps\big)_\veps$ is \emph{not} uniformly bounded in any sense, but the key (and somehow surprising) observation is that
the family $\big(\oline{\s_\veps}\big)_\veps$ of their vertical averages is uniformly bounded, even though in spaces of very negative regularity.

To see that this holds true, we seek an equation for $\oline{\s_\veps}$. The starting point is the weak formulation of the mass equation, which reads
\begin{equation*}
-\iint_{\R_+\Omega_{\eps}}\Big(\left(\rho_{\eps}-\rho_0^{in}\right)\partial_t \varphi\,+\,\rho_{\eps}u_{\eps}\cdot\nabla \varphi\Big)\,\dx\dt\,=\, \int_{\Omega_{\eps}}\left(\rho^{in}_{\eps}-\rho^{in}_0\right) \varphi(0,\cdot)\,\dx\,,
\end{equation*}
for any smooth test function $ \varphi\,=\,\vphi_\veps\,\in\,C^\infty_c\big(\R_+\times\oline{\Omega_\veps}\big)$.
Arguing similarly as we have done to get \eqref{equ:r:esp:bar}, we find, for any $\vphi\in C^\infty_c\big(\R_+\times\R^2\big)$, the relation
\begin{equation}
\label{equ:weak:sigma:eps}
-\iint_{\R_+\bbR^2}\Big(\eps\,\overline{\sigma_{\eps}} \partial_t \vphi \,+\,\overline{\rho_{\eps}\,u_{\eps, h}} \cdot \nabla_h \psi\Big)\,\dx\dt\,=\,
\int_{\bbR^2} \eps\, \overline{\sigma^{in}_{\eps}}\,\varphi(0,\cdot)\,\dx\,.
\end{equation}
Equality \eqref{equ:weak:sigma:eps} represents an equation for $\overline{\sigma_{\eps}} $: more precisely, it corresponds to the weak formulation of
the equation
\begin{equation}
\eps\, \partial_t \overline{\sigma_{\eps}}\, +\, \divh\big(\overline{\rho_{\eps}\, u_{\eps,h}}\big)\, =\, 0\,. \label{equ:sigm:eps}
\end{equation}

Next, we work on the weak formulation of the momentum equation, recall relation \eqref{eq:weak_eps} above.
In \eqref{eq:weak_eps}, we now use the special test function $\psi\,=\,\big(\nabla_h^\perp\vphi,0\big)/(2\ell_\veps)$, where
$\vphi\in C^\infty_c\big(\R_+\times\R^2\big)$. We notice that
\begin{align}
\label{2:riscrittura}
\int_{\Omega_\veps} m^{in}_\veps\cdot\psi(0,\cdot)\,\dx\, &=\, \int_{\bbR^2} \overline{m^{in}_\veps}\cdot \nabla_h^\perp \vphi(0,\cdot)\,\dx_h \\
\iint_{\R_+\times\Omega_{\eps}} \rho_\veps\,u_\veps\cdot\d_t\psi\, \dx \dt\, &=\,\iint_{\R_+\times\bbR^2} \overline{\rho_{\eps}\,u_{\eps,h}}\cdot
\d_t\nabla_h^{\perp} \vphi \, \dx_h \dt\,.
\label{0:riscrittura}
\end{align}
Similarly, we have
\begin{align}
\label{1:riscrittura}
&\iint_{\R_+\times\Omega_{\eps}}\Big( -\,\rho_\veps\,u_\veps\otimes u_\veps:\nabla\psi + \nabla u_\veps:\nabla\psi\Big) \, \dx\dt \\ 
	\nonumber
&\qquad\qquad\qquad\qquad =\,\iint_{\R_+\times\bbR^2}\Big(-\overline{ \rho_\veps\,u_{\veps, h}\otimes u_{\veps, h}} : \nabla_h\nabla^{\perp}_h\vphi\, +\,
\nabla_h \overline{u_{\eps, h}} : \nabla_h \nabla_h^{\perp} \vphi \Big)  \, \dx_h \dt \,.  
\end{align}
The most interesting terms are the ones associated with the boundary and the Coriolis force. For the term corresponding to the lower bottom, we have
\begin{equation}
	\label{3:riscrittura}
2 \alpha_{\eps} \iint_{\R_+\times\bbR^2} u_{\eps}(\cdot,\cdot, -\ell_{\eps}) \cdot \psi \, \dx_h\dt\, =\,
\frac{\alpha_{\eps}}{\ell_{\eps}}\,\iint_{\R_+\times\bbR^2} u_{\eps, h}(\cdot,\cdot, -\ell_{\eps}) \cdot \nabla_h^\perp \vphi \, \dx_h\dt\,. 
\end{equation}
The same computation holds if we replace $ u_{\eps}(\cdot,\cdot, -\ell_{\eps}) $ with $ u_{\eps}(\cdot,\cdot,  \ell_{\eps})$, which corresponds to
the upper boundary.
As for the Coriolis term, we argue as in \eqref{eq:Coriolis} to get, with the help of \eqref{equ:weak:sigma:eps}, the following equality:
\begin{align}
\label{4:riscrittura}
\frac{1}{\veps}\iint_{\R_+\times\Omega_{\eps}}e_3 \times \rho_{\eps}\, u_{\eps} \cdot\psi \, \dx \dt\, &= \,
\frac{1}{\veps}\iint_{\R_+\times\R^2}\overline{\rho_{\eps}\, u_{\eps, h}} \cdot \nabla_h \vphi \, \dx_h \dt  \\
&=\, \int_{\bbR^2} \overline{\sigma_{\eps}^{in}} \cdot \vphi(0,\cdot)\,\dd x_h
\,-\,\iint_{\R_+\times\bbR^2} \overline{\sigma_{\eps}} \d_t\vphi \, \dx_h \dt\,. \nonumber 
\end{align}
We can plug relations \eqref{2:riscrittura}, \eqref{0:riscrittura}, \eqref{1:riscrittura}, \eqref{3:riscrittura} and \eqref{4:riscrittura}
into \eqref{eq:weak_eps}, where we recall that we have taken $\psi\,=\,\big(\nabla_h^\perp\vphi,0\big)/(2\ell_\veps)$ as a test function.
After few integration by parts, we deduce 
%
%
\begin{align}
\label{weak:form:to:lim}
&\iint_{\R_+\times\R^2}\Big(-\left(\oline{\eta_\veps}-\oline{\s_\veps}\right)\,\d_t\vphi\,-\,
\oline{\rho_\veps u_{\veps,h}\otimes u_{\veps,h}}:\nabla_h\nabla_h^\perp\vphi\,+\,\oline{\o_\veps}\,\Delta_h\vphi\Big)\,\dd x_h\dt \\
\nonumber &\qquad\qquad\qquad\qquad\quad
-\,\frac{\alpha_{\eps}}{\ell_{\eps}}\,\iint_{\R_+\times\R^2}\big(\o_\veps^+\,+\,\o_\veps^-\big)\,\vphi\,\dd x_h\dt\,=\,
\int_{\R^2}\left(\overline{\eta_{\eps}^{in}} - \overline{\sigma_{\eps}^{in}}\right)\,\vphi(0,\cdot)\,\dd x_h\,,
\end{align}
which holds for any test function $\vphi\in C^\infty_c\big(\R_+\times\R^2\big)$ and
where we have denoted
\begin{align}
&\eta_\veps^{in}\,:=\,\curlh\left(m_{\veps,h}^{in}\right)\,,\qquad
\eta_\veps\,:=\,\curlh\big(\rho_\veps\,u_{\veps,h}\big)\,, \label{eq:notation} \\
&\qquad\qquad\qquad \omega_\veps\,:=\,\curlh\big(u_{\veps,h}\big)\,,\qquad
\o_\veps^\pm\,:=\,\curlh\big(u_{\veps,h}(\cdot,\cdot,\pm\ell_\veps)\big)\,. \nonumber
\end{align}
The integral equality \eqref{weak:form:to:lim} corresponds to the weak formulation of the following equation:
\begin{equation}
	\label{equ:to:lim}
\partial_t\left(\overline{\eta_{\eps}}- \overline{\sigma_{\eps}}\right) + \curlh\divh\big(\overline{ \rho_{\eps} u_{\eps,h} \otimes u_{\eps,h}}\big) + \frac{\alpha_{\eps}}{\ell_{\eps}}\Big(\o_{\eps}^+\,+\,\o_\veps^-\Big) -\Delta_h\oline{\o_\veps}\, =\, 0\,,
\end{equation}
related to the initial datum $\left(\overline{\eta_{\eps}}- \overline{\sigma_{\eps}}\right)_{|t=0}\,=\,\overline{\eta_{\eps}^{in}}-\overline{\sigma_{\eps}^{in}}$.

We see that equation \eqref{equ:to:lim} does not involve any fast time oscillation. It is then a convenient formulation to study the
asymptotics of the original system \eqref{eps:sys} for $\veps\ra0^+$.
Thus, in the sequel we will study the limit $ \eps\ra0^+ $ for \eqref{equ:to:lim}, or equivalently for \eqref{weak:form:to:lim}.

\subsection{Convergence of the linear terms in \eqref{equ:to:lim}}
In this subsection, we treat the convergence of the linear terms appearing in \eqref{equ:to:lim}.
First of all, we need some uniform bounds for the families $\big(\overline{\eta_{\eps}}\big)_{\veps} $ and $\big(\overline{\sigma_{\eps}}\big)_{\veps} $.

\begin{prop}
\label{conv:eta:sigma}
Under the assumptions of Theorem \ref{Main:Theo}, and with the notation introduced above, we have that
\[
\left(\oline{\eta_\veps}\right)_\veps\,\Subset\,L^\infty\big(\R_+;H^{-1}(\R^2)\big)\qquad\qquad \mbox{ and }\qquad\qquad
\left(\oline{\s_\veps}\right)_\veps\,\Subset\,L^\infty_{\rm loc}\big(\R_+;H^{-2}(\R^2)\big)\,.
\]
In particular, up to a suitable extraction, for any $T>0$ fixed one has 
\[
\overline{\eta_{\eps}} \cvwstar \curlh\big( \rho_0^{in}\, u \big)\; \mbox{ in }\  L^\infty\big([0,T];H^{-1}(\R^2)\big) \qquad \mbox{ and }\qquad
\overline{\sigma_{\eps}} \cvwstar \sigma\; \mbox{ in }\  L^\infty\big([0,T];H^{-2}(\R^2)\big)\,,
\]
for a suitable $\s\in L^\infty_{\rm loc}\big(\R_+;H^{-2}(\R^2)\big)$.
\end{prop}

\begin{proof}
By definition, $\oline{ \eta_{\eps}} = \curlh\big(\overline{\rho_{\eps}u_{\eps,h}}\big) $, where
$\left( \overline{\rho_{\eps}u_{\eps,h}}\right)_\veps\,\Subset\, L^{\infty}\big(\R_+;L^2(\bbR^2)\big)$ owing to Lemma \ref{Lemma:merc}.
This implies that
$\left(\oline{\eta_{\eps}}\right)_\veps\,\Subset\, L^{\infty}\big(\R_+;H^{-1}(\bbR^2)\big)$ and that there exists $\eta$ in that space such that,
up to subsequence, we have $\oline{\eta_{\eps}}  \cvwstar \eta $ in  $ L^{\infty}(\R_+; H^{-1}(\bbR^2)) $.
It remains to identify $ \eta $ with $ \curlh\big(\rho_0^{in} u\big) $; but this is an easy consequence of the properties
stated in Proposition \ref{con:all:var}.
%

We now switch to consider the functions $\oline{\s_\veps}$. In view of the previous properties and the assumptions on the initial data,
recall in particular \eqref{ub:r_in}, to conclude it is enough to show that
$\Big( \d_t\left(\overline{\sigma_{\eps}}-\overline{\eta_{\eps}}\right)\Big)_\veps\,\Subset\, L^{1}_{\rm loc}\big(\R_+; H^{-2}(\R^2)\big) $.

In order to prove that property, we bound all the terms appearing in equation \eqref{equ:to:lim}.
First of all, the energy estimates of Lemma \ref{Lemma:merc} imply that $\big(\oline{\o_\veps}\big)\,\Subset\,L^2\big(\R_+;L^2(\R^2)\big)$
and that $\left(\o_\veps^\pm\right)_\veps\,\Subset\,L^2\big(\R_+;H^{-1}(\R^2)\big)$.
Next, consider the term $ \overline{\rho_{\eps} u_{\eps,h} \otimes u_{\eps,h}} $: by Jensen inequality and \eqref{est:rho_inf}, we can bound
\begin{equation*}
\left\| \overline{\rho_{\eps} u_{\eps,h} \otimes u_{\eps,h}}\right\|_{L^2(\bbR^2)}^2\,\leq\,
\overline{ \left\|\big(\rho_{\eps} u_{\eps,h} \otimes u_{\eps,h}\big)(\cdot,\cdot,x_3)\right\|_{L^2(\bbR^2)}^2 }\,\leq\,
4\,\big(\rho^*\big)^2\,\overline{ \left\| \big(u_{\eps} \otimes u_{\eps}\big)(\cdot,\cdot,x_3) \right\|_{L^2(\bbR^2)}^2 }\,.
\end{equation*}
We now focus on the term $ u_{\eps} \otimes u_{\eps} $. Recall that the interpolation inequality between $L^p$ spaces extend to thin domain in the following way:
\begin{equation*}
\overline{\| f_{\eps}(\cdot,x_3) \|_{L^4(\bbR^2)}^4}\leq\overline{\|f_{\eps}(\cdot,x_3)\|_{L^2(\bbR^2)}^2}^{1/2}\;
\overline{\|f_{\eps}(\cdot,x_3)\|_{L^6(\bbR^6)}^6}^{1/2}\,.
\end{equation*}
With the help of the previous bound, from the Sobolev inequality of Lemma \ref{l:Poincare} and the Young inequality, we deduce that
\begin{align*}
\overline{\left\|\big(u_{\eps} \otimes u_{\eps}\big)(\cdot,\cdot,x_3)\right\|_{L^2(\bbR^2)}^2 }\,\lesssim\,
\overline{\left\| u_{\eps}(\cdot,\cdot,x_3)\right\|_{L^4(\bbR^2)}^4 }\,\lesssim\,\left( \overline{\left\| u_{\eps}(\cdot,\cdot,x_3)\right\|_{L^2(\bbR^2)}^2}\,+\, \overline{\left\| Du_{\eps}(\cdot,\cdot,x_3)\right\|_{L^2(\bbR^2)}^2 } \right)^2\,.
\end{align*}	
This, together with the bounds of Lemma \ref{Lemma:merc}, finally implies that, for any $T>0$, one has
\begin{align*}
\left\| \overline{\rho_{\eps} u_{\eps,h} \otimes u_{\eps,h} }\right\|_{L^1_T(L^2(\bbR^2))}\,\leq\,C\,,
\end{align*}	
which yields the claimed uniform bound of $\Big(\d_t\left(\overline{\sigma}_{\eps}-\overline{\eta}_{\eps}\right)\Big)_\veps$
in $L^{1}_{\rm loc}\big(\R_+; H^{-2}(\R^2)\big)$.
\end{proof}

\begin{rmk} \label{r:hyp_sigma}
Thanks to the chain of embeddings $H^{-1}(\R^2)\,\hookrightarrow\,H^{-2}(\R^2)\,\hookrightarrow\,W^{-k,\infty}(\R^2)$ for $k>3$, we see that we could dismiss the $H^{-2}$ boundedness in assumption \eqref{ub:r_in} on $\big(\,\oline{r^{in}_\veps}\,\big)_\veps$ and prove that
$\big(\oline{\s_\veps}\big)_\veps\,\Subset\,\bigcap_{k>3}L^\infty_{\rm loc}\big(\R_+;W^{-k,\infty}(\R^2)\big)$. This would be still enough in order
to extract a weakly convergent subsequence and a limit point $\s$ belonging to that space. The rest of the proof also would work similarly, with minor modifications
(see in particular the statement of Proposition \ref{p:wave} and the interpolation argument in the proof of Proposition \ref{prop:27}).
\end{rmk}

With the convergence proved in the previous proposition, in Proposition \ref{con:all:var} and in Lemma \ref{lemma:cons:bound:data},
we can pass to the limit in all the linear terms of equation \eqref{weak:form:to:lim}.
It remains only to show how to pass to the limit in the nonlinear term, namely in the integral
\begin{equation}
\label{nlt:to:lim}
\iint_{\R_+\times\bbR^2} \overline{\rho_{\eps} u_{\eps,h} \otimes u_{\eps,h}}: \nabla_h \nabla^{\perp}_h \vphi \, \dx_h \dt\,.
\end{equation}

\subsection{Convergence of the convective term} \label{ss:convective}

Here we explain how passing to the limit for $\veps\ra0^+$ in the convective term \eqref{nlt:to:lim}.
To begin with, we show that it is actually enough to pass to the limit in the integral
\[
\iint_{\R_+\times\bbR^2} \rho_0^{in}  \, \overline{ u_{\eps,h} }\otimes \overline{  u_{\eps,h} }: \nabla_h \nabla^{\perp}_h \vphi \, \dx_h \dt\,.
\]

\begin{prop}
\label{last:prob}
Under the assumptions of Theorem \ref{Main:Theo}, we have the following convergence properties, in the limit $\veps\ra0^+$:

\begin{itemize}

\item for any $T>0$ fixed, we have $\left\|\overline{ \rho_{\eps} u_{\eps,h}\otimes u_{\eps,h} } - \overline{ \rho_{\eps}} \; \overline{ u_{\eps,h} }
\otimes \overline{  u_{\eps,h} }\right\|_{L^1([0,T]\times \mathbb{R}^2)} \longrightarrow 0 $;

\item $ \overline{ \rho_{\eps}} \, \overline{ u_{\eps,h} }\otimes \overline{  u_{\eps,h} }-\rho_0^{in} \,\overline{ u_{\eps,h} }\otimes \overline{u_{\eps,h}}\,
\longrightarrow\, 0 $ in the sense of $ \mathcal{D}'\big(\R_+\times \mathbb{R}^2\big)$.

\end{itemize}

\end{prop}

\begin{proof}
We start by writing 
\begin{align*}
\overline{ \rho_{\eps} u_{\eps} \otimes u_{\eps} } - \overline{ \rho_{\eps}} \, \overline{ u_{\eps} }\otimes \overline{  u_{\eps} } = & \,  \overline{ \rho_{\eps} u_{\eps} \otimes u_{\eps} } - \overline{ \rho_{\eps}  u_{\eps} }\otimes \overline{  u_{\eps} } + \overline{ \rho_{\eps} u_{\eps} } \otimes \overline{ u_{\eps} } - \overline{ \rho_{\eps}} \, \overline{ u_{\eps} }\otimes \overline{  u_{\eps} } \\ 
= & \, \overline{ \rho_{\eps} u_{\eps} \otimes \left(u_{\eps}-\oline{u_\veps}\right) } +
\overline{  \left(u_{\eps}-\overline{u_{\eps}}\right)  \otimes \rho_{\eps} \overline{u_{\eps}} }\,.
\end{align*}
Then, an application of Corollary \ref{cor:1} with for $s = p = 2 $ yields the estimate
\begin{align*}
&\left\| \overline{ \rho_{\eps} u_{\eps} \otimes u_{\eps} }-
\overline{ \rho_{\eps}} \, \overline{ u_{\eps} }\otimes \overline{  u_{\eps} }\right\|_{L^1([0,T]\times \mathbb{R}^2)} \\
&\qquad\qquad\lesssim \,\ell_{\eps}\,
\overline{\left\|D u_{\eps}(\cdot,\cdot,x_3)\right\|_{L^2_T(L^2(\mathbb{R}^2))}^2}^{1/2} \\
&\qquad\qquad\qquad\qquad
\times\,\left(\overline{\left\| \big(\rho_{\eps} u_{\eps}\big)(\cdot,\cdot,x_3) \right\|_{L^2_T(L^2(\mathbb{R}^2))}^2}^{1/2}+
\overline{\left\|  \big(\rho_{\eps} \, \overline{ u_{\eps} }\big)(\cdot,\cdot,x_3)\right\|_{L^2_T(L^2(\mathbb{R}))}^2}^{1/2} \right)\,. 
\end{align*}
Using the uniform estimates of Lemma \ref{Lemma:merc} and \eqref{est:rho_inf}, which imply in particular
\begin{equation*}
\overline{\left\|\big(\rho_{\eps} \overline{u_{\eps}}\big)(\cdot,\cdot,x_3)\right\|_{L^2_T(L^2(\bbR^2))}^2}\,\leq\,4\,\left(\rho^*\right)^2\,
\left\|\overline{u_{\eps}}\right\|_{L^2_T(L^2(\bbR^2))}^2\,\leq\,C\,,
\end{equation*}
we deduce that the right-hand side of the previous inequality converges to $0$ when $\veps\ra0^+$.

For the second point, let us recall that for any compact $ K \subset \bbR^2 $ with boundary smooth enough and any time $T>0$,
the density  $\big(\overline{\rho_{\eps}}\big)_\veps $ is uniformly bounded in $ L^{\infty}\big([0,T];L^p(K)\big) $ and
$\big( \partial_t \overline{\rho_{\eps}}\big)_\veps $ is bounded in $ L^2\big([0,T];H_0^{-s}(K)\big) $ for $ s>0 $ large enough.
Hence, the Aubin-Lions theorem ensures that, up to a suitable extraction (which we omit here), $\big( \overline{\rho_{\eps}}\big)_\veps $ converges
strongly to $ \rho_0^{in} $ in the space $ C^0\big([0,T];H^{-\gamma}(K)\big) $, for any $ \gamma > 0 $. This fact, together with the uniform boundedness
of $\big( \overline{u_{\eps}} \otimes \overline{u_{\eps}}\big)_\veps $ in the space $\bigcap_{\de>0} L^1\big([0,T];H^{1-\de}\big)$ ensures that the difference
$ \overline{ \rho_{\eps}} \, \overline{ u_{\eps} }\otimes \overline{  u_{\eps} }\,-\,\rho^{in}_0 \,\overline{ u_{\eps} }\otimes \overline{  u_{\eps} }$
converges to $0$ when $\veps\ra0^+$ in the sense of
$\mathcal{D}'\big(\R_+\mathbb{R}^2\big)$.
\end{proof}

What it remains is to explain how to pass to the limit in the term
\begin{equation} \label{eq:pass-to-lim}
\iint_{\R_+\times\bbR^2} \rho_0^{in}  \, \overline{ u_{\eps,h} }\otimes \overline{  u_{\eps,h} } : \nabla_h\nabla_h^\perp\vphi\,\dd x_h\dt\,=\,
-\iint_{\R_+\times\bbR^2}\divh\big(\rho_0^{in}  \, \overline{ u_{\eps,h} }\otimes \overline{  u_{\eps,h} }) \cdot \nabla_h^\perp\vphi\,\dd x_h\dt\,.
\end{equation}
To do this, we follow the analogous argument used for the $ 2$-D setting, see Paragraph 5.2.2 of \cite{F:G}.
The proof is based on a regularization argument which is technical so we will use it only on the term on which it is essential.

\subsubsection{Preliminary reductions} \label{sss:reductions}

In order to treat the convergence of the integral in \eqref{eq:pass-to-lim}, we start by computing
\begin{align}
\label{rew:equ:term:7}
\divh\left(\rho_0^{in}  \, \overline{ u_{\eps,h} }\otimes \overline{  u_{\eps,h} }\right)\,&= \,
\rho_0^{in}\, \overline{u_{\eps,h}} \cdot \nabla_h\overline{u_{\eps,h}}\,+\,\left(\nabla_h \rho_0^{in} \cdot \overline{u_{\eps,h}}\right)\,\overline{u_{\eps,h}} \\
&=\, \frac{1}{2}\,\rho_0^{in}\, \nabla_h\left|\overline{u_{\eps,h}}\right|^2\, +\, \rho_0^{in}\,\oline{\omega_{\eps,h}}\, \overline{u_{\eps,h}^\perp}\, +\,
\left(\nabla_h \rho_0^{in} \cdot \overline{u_{\eps,h}}\right)\, \overline{u_{\eps,h}} \nonumber \\
&=\,  \frac{1}{2}\, \rho_0^{in}\,\nabla_h\left|\overline{u_{\eps,h}}\right|^2\, +\,\curlh\left(\rho_0^{in}\,\overline{u_{\eps,h}}\right)\, \overline{u_{\eps,h}^\perp}
\,+\,\Theta_{\veps}\,, \nonumber
\end{align} 
where we have used the notation introduced in \eqref{eq:notation} and we have defined
\[
\Theta_\veps\,:=\,\,-\,\left(\nabla_h^{\perp} \rho_0^{in} \cdot \overline{u_{\eps,h}}\right)\, \overline{u_{\eps,h}^\perp}\,+\,
\left(\nabla_h \rho_0^{in} \cdot \overline{u_{\eps,h}}\right)\, \overline{u_{\eps,h}}\,.
\]

We have the next result.
\begin{lemma} \label{l:first-approx}
There exist two distributions $ \Gamma_1 $ and $ \Gamma_2$ in $\mathcal{D}'\big(\R_+\times \bbR^2\big ) $ such that the following convergence properties hold true:
for any $\vphi\in C^\infty_c\big(\R_+\times\R^2\big)$, in the limit $\veps\ra0^+$ one has
\begin{align}
\label{conv:Gamma1}
\iint_{\R_+\times\bbR^2}\frac{1}{2}\, \rho_0^{in}\,\nabla_h\left|\overline{u_{\eps,h}}\right|^2 \cdot \nabla^{\perp}_{h}\vphi\, \dx_h \dt\,&\longrightarrow\,
\langle \rho_0^{in} \nabla_h \Gamma_1, \nabla^{\perp}_{h} \vphi \rangle  \\
\label{conv:Gamma2}
\iint_{\R_+\bbR^2} \Theta_\veps \cdot \nabla_h^{\perp}\vphi \, \dx_h \dt\, &\longrightarrow\,\langle \rho_0 \nabla_h \Gamma_2, \nabla_h^{\perp} \psi \rangle\,.
\end{align}
\end{lemma}

\begin{proof}
The convergence \eqref{conv:Gamma1} is a simple consequence of the uniform boundedness of $\left(\overline{u_{\eps}}\right)_\veps $ in
$L^2_{\rm loc}\big(\R_+;L^2(\bbR^2)\big)$. So, let us focus on \eqref{conv:Gamma2}.

In order to treat the term $\Theta_\veps$, we need to introduce a cut-off which separates $ \bbR^2 $ into one part where $ \left| \nabla \rho_0^{in} \right| $
is small and another where $ \left| \nabla \rho_0^{in} \right| $ is large. To do that, take a smooth non increasing function
$ b: [0,+\infty[\,\longrightarrow [0,1] $ such that $ b\equiv 1 $ in $ [0,1] $ and $ b\equiv0 $ in $ [2,+\infty[\,$.
For any $M>0$ and any $s\in\R_+$, we denote $ b_M(s)\,:=\, b(Ms) $. 

Then, as $ b_{M}\left(\left|\nabla_h \rho_0^{in}\right|\right) = 0 $ for $ \left|\nabla_h \rho_0^{in}\right| \geq 2/M $, for any $T>0$ fixed,
directly from the definition of $\Theta_\veps$ we have that
\begin{align}
\Big\| b_{M}\left(\left|\nabla_h \rho_0^{in}\right|\right)\,\Theta_\veps\Big\|_{L^1_T(L^1(\bbR^2))} &\lesssim\,
\Big\| b_{M}\left(\left|\nabla_h \rho_0^{in}\right|\right)  \nabla_h \rho_0^{in}\Big\|_{L^{\infty}([0,T]\times\bbR^2)}\,
\left\| \overline{u_{\eps,h}}\right\|^2_{L^2_T(L^2(\bbR^2))}\,\lesssim\,\frac{1}{M}\,.  \label{est:M;nonso}
\end{align}  

When $\left| \nabla_h \rho_0^{in}\right| > 0 $, instead, the set $ \left\{ \nabla_h \rho_0^{in}, \nabla_h^{\perp} \rho_0^{in} \right\} $ forms an orthogonal basis
of $ \bbR^2 $. In particular, any vector $ v \in \bbR^2 $ can be rewritten as
\begin{equation} \label{eq:vect-dec}
v\,=\, \left(v\cdot \nabla_h \rho_0^{in}\right)\,\frac{\nabla_h \rho_0^{in}}{\left|\nabla_h \rho_0^{in}\right|^2}\, +\,
\left(v \cdot \nabla_h^{\perp} \rho_0^{in}\right)\, \frac{ \nabla_h^{\perp} \rho_0^{in}}{\left|\nabla_h \rho_0^{in}\right|^2}\,.
\end{equation}
Using this idea for $\oline{u_{\veps,h}}$ and $\oline{u_{\veps,h}^\perp}$,
simple computations and special cancellations (we refer to Step 3 in \cite{F:G} and Subsection 4.4 in \cite{C-F} for details) show that
\begin{align*}
\Big(1-b_{M}\left(\left|\nabla_h \rho_0^{in}\right|\right)\Big)\,\Theta_\veps\,=\,
\frac{1-b_{M}\left(\left|\nabla_h \rho_0^{in}\right|\right)}{|\nabla_h \rho_0^{in}|^2}\,\left(\left(\nabla_h \rho_0^{in}\cdot\overline{u_{\veps,h}}\right)^2
\,+\,\left(\nabla_h^{\perp} \rho_0^{in} \cdot \overline{u_{\veps,h}}\right)^2   \right)\, \nabla_{h} \rho_0^{in}\,.
\end{align*}  
Using this equality, we can bound
\begin{align*}
&\Bigg\|\Big(1-b_{M}\left(\left|\nabla_h \rho_0^{in}\right|\right)\Big)\,\Theta_\veps \Bigg\|_{L^1_T(L^1(\bbR^2)) } \\
&\qquad\qquad\qquad\qquad\quad
\leq\,C\,
\left\|\frac{1-b_{M}\left(\left|\nabla_h \rho_0^{in}\right|\right)}{\left|\nabla_h \rho_0^{in}\right|^2}\,
\left|\nabla_h \rho_0^{in}\right|^3\right\|_{L^{\infty}([0,T]	\times \bbR^2) }\,\left\|\overline{u_{\eps,h}}\right\|_{L^2_T(L^2(\bbR^2))}^2\;
\leq\;C\,,
\end{align*}
where the multiplicative constant $ C>0 $ is independent of $\veps\in\,]0,1]$ and $ M>0 $.
By a diagonal argument, there exists a subsequence in $ \eps $ (not relabelled here)  and a sequence $ M = M_{\eps} $, diverging to $ + \infty $, such that  
\begin{equation*}
\frac{1-b_{M_{\eps}}\left(\left|\nabla_h \rho_0^{in}\right|\right)}{\left|\nabla_h \rho_0^{in}\right|^2}\,
\left(\left(\nabla_h \rho_0^{in} \cdot \overline{u_{\eps,h}}\right)^2\, +\,\left(\nabla_h^{\perp} \rho_0^{in} \cdot \overline{u_{\eps,h}}\right)^2 \right)
\,\cvwstar\, \Gamma_2
\end{equation*}
in the weak-$*$ topology of the space $ \mathcal{M}\big([0,T]\times \bbR^2\big) $ of finite Randon mesures on $ [0,T]\times \bbR^2  $. 

Therefore, using also the equality $\nabla_h\rho_0^{in}\cdot\nabla_h^\perp\vphi\,=\,\divh\left(\rho_0^{in}\,\nabla_h^\perp\vphi\right)$,
we deduce that
\begin{align*}
&\int_0^T \int_{\bbR^2}\Big(1-b_{M}\left(\left|\nabla_h \rho_0^{in}\right|\right)\Big)\,\Theta_\veps\cdot\nabla_h^\perp\vphi\,\dd x_h \dt \\
&=\,\int_0^T \int_{\bbR^2}\frac{1-b_{M_{\eps}}\left(\left|\nabla_h \rho_0^{in}\right|\right)}{\left|\nabla_h \rho_0^{in}\right|^2}\,
\left(\left(\nabla_h \rho_0^{in} \cdot \overline{u_{\eps,h}}\right)^2\, +\,\left(\nabla_h^{\perp} \rho_0^{in} \cdot \overline{u_{\eps,h}}\right)^2 \right)
\divh\left( \rho_0^{in}\,\nabla_h^\perp\vphi\right)\,\dd x_h \dt \\
&\qquad \longrightarrow\, \langle \rho_0 \nabla_h \Gamma_2, \psi \rangle\,,
\end{align*}
where the time $T>0$ is such that $\Supp\vphi\subset[0,T]\times \R^2$.
The claimed limit \eqref{conv:Gamma2} then follows from \eqref{est:M;nonso} and the above convergence.
\end{proof} 

We are thus left with the most difficult term, namely
\begin{equation}
\label{diff:term}
\iint_{\R_+\bbR^2} \curlh\left( \rho_0^{in}\, \overline{u_{\eps,h}}\right)\, \overline{u_{\eps,h}^{\perp}} \cdot\nabla_h^\perp\vphi\,\dd x_h\dt\,.
\end{equation}
To deal with this term, we first need to apply a regularisation procedure: this is the scope of the next paragraph.

\subsubsection{High frequencies cut-off} \label{sss:cut-off} 

In this section we introduce a high frequency cut-off, that allows us to regularise functions in the space variables. 
For this, fix $ \chi \in C^{\infty}_c(B_2) $, where we recall that $B_R$ denotes the ball of center $0$ and radius $R>0$ in $\R^2$,
with $ \chi \equiv 1 $ in an open neighborhood of $ \overline{B_{1}} $, $\chi$ radially symmetric and such that,
for any $ e \in  \mathbb{S}^1 $ (with $\mbb S^1$ being the unit sphere of $\R^2$), the map $r\mapsto \chi(r e ) $ is a non increasing function for
$ r \in \bbR_+ $.  
  
For any tempered distribution $ f \in \mathcal{S}'(\bbR^2) $ and for any $M\in\N$, we define
\begin{equation*}
S_M[f]\, :=\, \mathcal{F}^{-1}\Big(\chi\left(2^{-M}\,\xi\right)\;\mathcal{F}(f)(\xi)\Big)\,,
\end{equation*}
where $ \mathcal{F} $ and   $ \mathcal{F}^{-1} $ are the Fourier transform in $ \bbR^2$ and its inverse operator, respectively.
Many properties of the above cut-off operator $S_M$ can be found in Chapter 2 of \cite{B-C-D} (see also the Appendices of \cite{F:G} and \cite{C-F}).
 
Correspondingly to $S_M$, we introduce also the following notation: for any function $f_\veps$ over $\R_+\times\Omega_\veps$, we denote
\begin{gather*}
\oline{f_\veps}_M\,:=\,S_M\Big[\,\oline{f_\veps}\,\Big]\,.
\end{gather*}   
From now on, functions indexed with $ M $ are smooth in the space variable. Moreover, the following properties hold true.
\begin{prop} \label{p:wave}
For any fixed time $ T > 0 $, we have:
\begin{align*}
\forall\,s>2\,,\qquad 
&\sup_{\eps \in\,]0,1]}\left\|\overline{\sigma_{\eps}}_M - \overline{\sigma_{\eps}}\right\|_{L^{\infty}_T(H^{-s}(\bbR^2))}\, \longrightarrow\, 0 \qquad \text{ as }\ 
M \ra +\infty \\
\forall\,s>1\,,\qquad 
&\sup_{\eps \in\,]0,1]}\left\|\overline{\eta_{\eps}}_{M} - \overline{\eta_{\eps}}\right\|_{L^{\infty}_T(H^{-s}(\bbR^2))}\, \longrightarrow\, 0 \qquad \text{ as }\
M \ra +\infty\,.
\end{align*}
Moreover, the couple $ \left(\overline{\sigma_{\eps}}_M,  \overline{\eta_{\eps}}_{M}\right) $ satisfies the wave system
\begin{align*}
&\eps\,\partial_t \overline{\sigma_{\eps}}_M\, +\, \divh\left(\overline{V_{\eps}}_M\right)\, = \, 0 \\
&\eps\, \partial_t \overline{\eta_{\eps}}_M\, +\, \divh\left(\overline{V_{\eps}}_M\right)\, = \, \eps\, \overline{f_{\eps}}_M\,,
\end{align*} 
where we have defined $ V_\veps\,:=\,\rho_\veps\,u_{\veps,h} $ (so that $\overline{V_{\eps}}_M\,=\,S_M\left[\oline{\rho_\veps\,u_{\veps,h}}\right]$)
and where the family $\big(\oline{f_\veps}_M\big)_{\veps,M}$ satisfies the following bounds: for any $s\geq0$, any $T>0$ and any $M\in\N$,
there exists a positive constant $C=C(s,T,M)$, only depending on the quantities in the brackets, such that
\begin{equation*}
\sup_{\eps \in \,]0,1]}\left\| \oline{f_{\eps}}_M \right\|_{L^2_T(H^s(\bbR))} \leq C(s, T,M)\,.
\end{equation*}

\end{prop}

\begin{proof}
The proof of the first part of the proposition is an immediate consequence of the uniform bounds stated in Proposition \ref{conv:eta:sigma}.

On the other hand, the fact that $\left(\overline{\sigma_{\eps}}_M, \overline{\eta_{\eps}}_M\right) $ satisfies the wave system is a consequence of
relations \eqref{equ:sigm:eps} and \eqref{equ:to:lim}, after noticing that $ S_M $ commutes with derivatives and after defining
$\oline{f_\veps}_M\,:=\,S_M\Big[\oline{f_\veps}\Big]$, with
\begin{equation*}
\overline{f_{\eps}}\,:=\, -\left(\curlh\divh\left(\overline{ \rho_{\eps} u_{\eps,h} \otimes u_{\eps,h}}\right)\,+\,
\frac{\alpha_{\eps}}{\ell_{\eps}}\,\left(\omega_\veps^+\,+\,\o_\veps^-\right)\right)\,+\,\Delta_h\oline{\o_\veps}\,.
\end{equation*}

The last estimate stated in the proposition is a consequence of Lemma \ref{Lemma:merc} and the boundedness properties of operator $S_M$ in Lebesgue
and Sobolev spaces.
\end{proof}

Next, we need the following result.
\begin{prop}
	\label{prop:27}
The two-dimensional vector field $ \overline{V_{\eps}}_M\,=\,S_M\left[\oline{\rho_\veps\,u_{\veps,h}}\right] $ can be rewritten as
\begin{equation}
\label{dec:V:eps:M}
\overline{V_{\eps}}_M\, =\, \rho_{0}^{in}\,\overline{u_{\eps,h}}_M\, +\, H_{\eps,M} \, +\, \eps^{\theta}\,\z_{\eps,M}\, +\, \ell_{\eps}\, G_{\eps,M}\,,
\end{equation} 
for some suitable $\theta\in\,]0,1[\,$ and smooth functions $H_{\veps,M}$, $\z_{\veps,M}$ and $G_{\veps,M}$. In addition,
for any $T>0$, any $k\in\N$ and any $M\in\N$, one has
\begin{align*}
\sup_{\eps \in\,]0,1]}\Big( \left\|G_{\eps,M}\right\|_{L^2_T(H^{k}(\bbR^2))}\, +\,\left\|\z_{\eps,M}\right\|_{L^2_T(W^{k,\infty}(\R^2))}\Big)\,&\leq\,C(k,T,M) \\
\sup_{\eps \in\,]0,1] }\left\|H_{\eps,M}\right\|_{L^2_T(H^{1}(\bbR^2))}\,&\leq\,C\left(T,\left\|\rho_0^{in}\right\|_{W^{2,\infty}}\right)\,2^{-M}\,,
\end{align*}
for suitable positive constants depending only on the quantities appearing in the brackets.
\end{prop}

\begin{proof}
Starting from the definition of $V_\veps$, for any $\theta\in\,]0,1[\,$ we can write
\begin{align*}
\overline{V_{\eps}}_M\, &=\, S_M\left[\overline{\rho_\eps\, u_{\eps,h}}\right]\, =\, S_M\left[\overline{\rho_\eps} \, \overline{u_{\eps,h}}\right]\, +\, S_M\left[\overline{\rho_\eps\,\big(u_{\eps,h} - \overline{u_{\eps,h}}}\big)\right] \\
&=\, S_M\left[\rho_0^{in}\,\overline{u_{\eps,h}}\right]\,+\, 
\eps^{\theta}\,S_M\left[\frac{\overline{\rho_{\eps}}-\rho_{0}^{in}}{\eps^{\theta}}\,\overline{u_{\eps,h}}\right]\, +\,
S_M\left[\overline{\rho_\eps\,\big(u_{\eps,h} - \overline{u_{\eps,h}}}\big)\right]\,, 
\end{align*}
whence the decomposition \eqref{dec:V:eps:M} follows after defining
\begin{align*}
&H_{\eps,M}\,:=\, S_M\left[\rho_0^{in}\, \overline{u_{\eps,h}}\right]-\rho_0^{in}\,S_M\left[\overline{u_{\eps,h}}\right]\,, \\
&\qquad\qquad \z_{\eps,M}\,:=\,S_M\left[\frac{\overline{\rho_{\eps}}-\rho_{0}^{in}}{\eps^{\theta}}\,\overline{u_{\eps,h}}\right]
\qquad\quad \mbox{ and } \qquad\quad
G_{\eps,M}\,:=\,\ell^{-1}_{\eps}\,S_M\left[\overline{\rho_\eps\,\big(u_{\eps,h} - \overline{u_{\eps,h}}}\big)\right]\,.
\end{align*}

We now need to estimate the previous terms in suitable norms. First of all,
from standard commutator estimates (see \tsl{e.g.} Lemma A.7 of \cite{F:G}) and Lemma \ref{Lemma:merc}, we have that
\begin{equation*}
\sup_{\eps\in\,]0,1]}\left\|H_{\eps,M}\right\|_{L^2_T(L^2(\bbR^2))}\,\lesssim\, 2^{-M}\,\left\|\nabla \rho_0^{in}\right\|_{L^{\infty}(\bbR^2)}\,
\left\|\overline{u_{\eps,h}}\right\|_{L^2_T(L^2(\bbR^2))}\,.
\end{equation*}
%
%

The estimate for the gradient $\nabla_h H_{\veps,M}$ is analogous, so let us switch directly to the control of $\z_{\eps,M}$.
To begin with, using\footnote{Here, the notation $B^s_{p,r}=B^s_{p,r}(\R^2)$ 
stands for the non-homogeneous Besov space of regularity index $s\in\R$ integrability index $p\in[1,+\infty]$ and summation index $r\in[1,+\infty]$ over $\R^2$.
We refer to Chapter 2 of \cite{B-C-D} for basic definitions and properties of these spaces. Here we only recall that, for any $s\in\R$,
$H^s\equiv B^s_{2,2}$ and that, for any $k\in\N$, one has $B^k_{\infty,1}\hookrightarrow W^{k,\infty}\hookrightarrow B^k_{\infty,\infty}$.}
the embeddings $L^\infty(\R^2)\hookrightarrow B^{0}_{\infty,\infty}(\R^2)$ and (in dimension $d=2$, for any $s\in\R$)
$H^s(\R^2)\hookrightarrow B^{s-1}_{\infty,\infty}(\R^2)$,
an interpolation argument allows us to write
\begin{align*}
\left\|\frac{\overline{\rho_{\eps}}-\rho_{0}^{in}}{\eps^{\theta}} \right\|_{L^{\infty }_T(B^{-3\theta}_{\infty,\infty}(\bbR^2))}\, &\leq\,
\left\|\left(\overline{\rho_{\eps}}-\rho_{0}^{in} \right)\right\|_{L^{\infty}_T(B^0_{\infty,\infty}(\bbR^2))}^{1-\theta}\,
\left\|\frac{\overline{\rho_{\eps}}-\rho_{0}^{in}}{\eps} \right\|_{L^{\infty}_T(B^{-3}_{\infty,\infty}(\bbR^2))}^{\theta} \\ 
&\lesssim\,\left\|\overline{\rho_{\eps}}-\rho_{0}^{in}\right\|_{L^{\infty}_T(L^{\infty}(\R^2))}^{1-\theta}\,
\left\|\overline{\sigma_{\eps}}\right\|_{L^{\infty}_T(H^{-2}(\bbR^2))}^{\theta}\,. 
\end{align*}
Notice that the right-hand side of the previous inequality
is uniformly bounded, owing to the bounds of Lemma \ref{Lemma:merc}. On the other hand, for $\g\in\,]0,1[\,$ one has
\begin{equation*}
\left\|\overline{u_{\eps,h}}\right\|_{L^{2}_T(H^{\gamma}(\bbR^2))}\,\leq\,\left\|\overline{u_{\eps,h}}\right\|_{L^{2}_T(L^2(\bbR^2))}^{1-\gamma}
\,\left\|\overline{u_{\eps,h}}\right\|_{L^{2}_T(H^{1}(\bbR^2))}^{\gamma}\,,
\end{equation*} 
which is again uniformly bounded by a constant $C=C(T)>0$, in view of Lemma \ref{Lemma:merc}.
Then, using the previous interpolation inequalities and product laws in Besov spaces (see \tsl{e.g.} Theorems 2.82 and 2.85 of \cite{B-C-D}) we get,
for $ \gamma -3 \theta > 0 $, the bound
\begin{align*}
\left\|\frac{\overline{\rho_{\eps}}-\rho_{0}^{in}}{\eps^{\theta}}\, \overline{u_{\eps,h}} \right\|_{L^2_T(B^{-3\theta-1+\gamma}_{\infty,\infty}(\bbR^2))}\,
&\lesssim\, \left\|\frac{\overline{\rho_{\eps}}-\rho_{0}^{in}}{\eps^{\theta}} \right\|_{L^{\infty}_T(B^{-3\theta}_{\infty,\infty}(\bbR^2))}\,
\left\|\overline{u_{\eps,h}}\right\|_{L^{2}_T(H^{\gamma}(\bbR^2))}\,\leq\,C\,,
\end{align*} 
which in turn yields\footnote{Here, we also use that, for any $s\in\R$, any $p,r\,\in[1,+\infty]$ and any $\de>0$, one has the embedding
$B^s_{p,r}\hookrightarrow B^{s-\de}_{p,1}$.} the claimed estimate for the $\z_{\veps,M}$ term.
%
%

We conclude with the estimates on $ G_{\eps,M} $. Notice that 
\begin{align*}
\left\|\overline{\rho_\eps\,\big(u_{\eps,h} - \overline{u_{\eps,h}}}\big)\right\|_{L^2_T(L^2(\bbR^2))}^2\,&=\,
\int_{0}^T\int_{\bbR^2}\left( \frac{1}{2 \ell_{\eps}} \int_{-\ell_{\eps}}^{\ell_{\eps}} \rho_{\eps}\,\big(u_{\eps,h}-\overline{u_{\eps,h}}\big)\,\dx_3\right)^2\,
\dx_h \dt \\
&\leq\,\frac{1}{2\ell_{\eps}}\int^T_0\int_{\bbR^2}\int_{-\ell_{\eps}}^{\ell_{\eps}}
\left|\rho_{\eps}\right|^2\,\left|u_{\eps,h}-\overline{u_{\eps,h}}\right|^2  \, \dx \dt \\
&\leq\, C \ell_{\eps}\,\left\|\rho_{\eps}\right\|_{L^{\infty}([0,T]\times \Omega_{\eps})}\,
\overline{\left\|\nabla u_{\eps}(\cdot,\cdot,x_3)\right\|_{L^{2}_T(L^2(\R^2)) }^2}\,,
\end{align*}
where, in the last step, we have used Lemma \ref{l:Poincare}. From the above inequality and the \tsl{a priori} bounds of Lemma \ref{Lemma:merc},
we deduce the claimed estimates for $G_{\veps,M}$.

This completes the proof of the proposition.
\end{proof}

\subsubsection{Convergence of the term \eqref{diff:term}} \label{sss:difficult}

After this preparation, we can now tackle the convergence of the integral \eqref{diff:term}.
\begin{prop}
\label{p:Gamma3}
There exists a distribution $ \Gamma_3  \in \mathcal{D}'\big(\R_+\times \bbR^2\big) $ such that, for any $\vphi\in C^\infty_c\big(\R_+\times\R^2\big)$,
in the limit $\veps\ra0^+$ one has the convergence
\begin{equation*}
\iint_{\R_+\times\bbR^2}\curlh\big( \rho_0^{in}\,\overline{u_{\eps,h}}\big)\,\overline{u_{\eps,h}^{\perp}}\cdot\nabla_h^{\perp}\vphi\,\dx_h\dt\,\longrightarrow\,
\langle \rho_0^{in} \nabla_h \Gamma_3, \nabla_h^{\perp}\vphi \rangle\,.
\end{equation*}
\end{prop}

To prove the above statement, we start by noticing that it is actually enough to pass to the limit for the regularised solutions. 
\begin{lemma} \label{l:diff_reg}
For any $ \vphi \in C^{\infty}_c\big(\R_+ \times \bbR^2\big) $, we have
\begin{align*}
&\lim_{M\to \infty} \limsup_{\eps \to 0^+}
\bigg|\iint_{\R_+\times\bbR^2}\curlh\big( \rho_0^{in}\,\overline{u_{\eps,h}}\big)\,\overline{u_{\eps,h}^{\perp}}\cdot\nabla_h^{\perp}\vphi\,\dd x_h\dt \\
&\qquad\qquad\qquad\qquad\qquad\qquad\qquad \, -\,
\iint_{\R_+\times\bbR^2}\curlh\big( \rho_0^{in}\,\overline{u_{\eps,h}}_{M}\big)\,\overline{u_{\eps,h}}^{\perp}_M\cdot\nabla_h^{\perp}\vphi\,\dd x_h\dt \bigg|\,=\,0\,.
\end{align*}

\end{lemma}

\begin{proof}[Proof of Lemma \ref{l:diff_reg}]
Let us write
\begin{align*}
&\curlh\big( \rho_0^{in}\, \overline{u_{\eps,h}}\big)\, \overline{u_{\eps,h}^{\perp}} \cdot \nabla_h^{\perp}\vphi\,
-\,\curlh\big( \rho_0^{in} \overline{u_{\eps,h}}_M\big)\, \overline{u_{\eps}}_M^{\perp} \cdot \nabla_h^{\perp}\vphi \\
&\qquad =\, \curlh\big( \rho_0^{in} \overline{u_{\eps,h}}\big)\, \Big(\overline{u_{\eps,h}}- \overline{u_{\eps,h}}_M\Big)^{\perp}\cdot\nabla_h^{\perp}\vphi\,-
\curlh\Big(\rho_0^{in} \left(\overline{u_{\eps,h}}-\overline{u_{\eps,h}}_M\right)\Big) \overline{u_{\eps,h}}_M^{\perp} \cdot \nabla_h^{\perp}\vphi\,.
\end{align*}
The result then follows from this equality and with the estimates
\begin{equation*}
\sup_{M\in\N}\sup_{\veps\in\,]0,1]}\left\|\oline{u_\veps}_M\right\|_{L^2_T(H^1(\R^2))}\,\leq\,C \quad \mbox{ and }\quad
\left\|\overline{u_{\eps}} - \overline{u_{\eps}}_M\right\|_{L^{2}_T(L^2(\bbR^2))}\,\leq\,C\,2^{-M}\,\left\|\overline{u_{\eps}}\right\|_{L^{2}_T(H^1(\bbR^2))}\,,
\end{equation*}
which hold true owing to the bounds of Lemma \ref{Lemma:merc} and the properties of the operator $S_M$.
\end{proof}
We are now able to show Proposition \ref{p:Gamma3}.

\begin{proof}[Proof of Proposition \ref{p:Gamma3}]
We will show that there exists a measure
\[
\wtilde{\Gamma}\,\in\, \bigcap_{T>0}\mathcal{M}\big([0,T]\times \bbR^2\big)\,,
\]
which does not charge the set $\left\{\nabla_h\rho_0^{in}=0\right\}$, such that
\begin{equation}
\label{conv:Gamma:tilde}
\lim_{\eps\ra0^+ }\iint_{\R_+\times\bbR^2}\curlh\big( \rho_0^{in}\,\overline{u_{\eps,h}}\big)\,\overline{u_{\eps,h}^{\perp}}\cdot\nabla_h^{\perp}\vphi\,\dd x_h\dt\,=\,
\left\lan\wtilde{\Gamma}, \frac{\nabla_h \rho_0^{in}}{\left|\nabla_h \rho_0^{in}\right|}\cdot\nabla_h^\perp\vphi\right\ran\,.
\end{equation}
Indeed, this property implies Proposition \ref{p:Gamma3} by denoting $ \Gamma_3 := \wtilde{\Gamma}/|\nabla_h \rho_0 | $.

From Lemma \ref{l:diff_reg}, we already know that we can work with the regularised velocity fields $\oline{u_\veps}_M$.
For notational convenience, from now on we denote with $ o(M) $ a function that depends on $ M $, such that $ \lim_{M \to \infty} o(M) =  0 $.
%
%
At this point, fixed $ M\in\N $, we will prove that there exisxts a smooth function $ \wtilde{\Gamma}_M $ over $\R_+\times\R^2$  such that
\begin{equation}
\label{F1}
\lim_{\eps \to 0^+}\left|\iint_{\R_+\times\bbR^2}\Big(\curlh\big(\rho_0^{in}\overline{u_{\eps,h}}_M\big)\,\overline{u_{\eps,h}}_M^{\perp}\cdot\nabla_h^{\perp}\vphi\,-\,
\wtilde{\Gamma}_M\, \frac{\nabla_h \rho_0^{in}}{\left|\nabla_h \rho_0^{in}\right|}\cdot\nabla_h^\perp\vphi\Big)\,\dd x_h \dt  \right|\, =\, o(M).
\end{equation}
For proving the previous property, we resort to the cut-off $ b_M $, introduced in the proof of Lemma \ref{l:first-approx}. Then, 
we can decompose 
\begin{equation*}
\beta_{\veps,M}\,:=\,\curlh\big(\rho_0^{in}\overline{u_{\eps,h}}_M\big)\,\overline{u_{\eps,h}}_M^{\perp}
\,=\,b_M\left(\left|\nabla_h \rho_0^{in}\right|\right)\,\beta_{\veps,M}\,+\,\Big(1-b_M\left(\left|\nabla_h \rho_0^{in}\right|\right)\Big)\,\beta_{\veps,M}\,.
\end{equation*}
Similarly to what done before for \eqref{est:M;nonso}, the first term on the right-hand side is of order $o(M)$ in $ L^1_{\rm loc} $. Indeed,
for any $T>0$ and any $ R > 0 $ fixed, we can estimate
\begin{align}
 \label{con:o(1)}
\left\|b_M\left(\left|\nabla_h \rho_0^{in}\right|\right)\,\beta_{\veps,M}\right\|_{L^1_T(L^1(B_R))}\,&\leq\,
\left\|\rho_0^{in}\right\|_{W^{1,\infty}(\bbR^2)}\,\left\|\oline{u_{\eps}}_M\right\|_{L^2_T(H^{1}(\bbR^2))}\,\left\|\oline{u_{\eps}}_M\right\|_{L^2_T(L^4(\bbR^2))} \\
& \qquad\qquad\qquad \times\, \left(\mc L \left(\left\{ x \in B_R\;\Big| \quad \left| \nabla_h \rho_0^{in}\right|\,\leq\,2/M \right\}\right) \right)^{1/4} \nonumber \\
&=\, o(M), \nonumber
\end{align} 
where we have also used the uniform estimates of Lemma \ref{Lemma:merc} and assumption \eqref{cond:12}.
Next, using again decomposition \eqref{eq:vect-dec} on the support of $ 1-b_M\left(\left|\nabla_h \rho_0^{in} \right|\right) $, we can write
\begin{align}
\label{last:rewrite} 
&\Big(1-b_M\left(\left|\nabla_h \rho_0^{in}\right|\right)\Big)\,\beta_{\veps,M}\, \\
&\quad= \,
\frac{1-b_M\left(\left|\nabla_h \rho_0^{in}\right|\right)}{\left|\nabla_h\rho_0^{in}\right|^2}\,\big(\beta_{\veps,M}\cdot \nabla_{h}\rho_0^{in}\big)\,
\nabla_{h}\rho_0^{in}
\,+\,\frac{1-b_M\left(\left|\nabla_h \rho_0^{in}\right|\right)}{\left|\nabla_h\rho_0^{in}\right|^2}\,\big(\beta_{\veps,M}\cdot\nabla_{h}^{\perp}\rho_0^{in}\big)\,
\nabla_{h}^{\perp} \rho_0^{in}\,. \nonumber
\end{align}

Regarding the first term appearing on the right of the previous relation, we notice that, for any $T>0$ fixed, one has
\begin{align}
\label{unif:est:hope}
&\left\|\frac{1-b_M\left(\left|\nabla_h \rho_0^{in}\right|\right)}{\left|\nabla_h\rho_0^{in}\right|}\,\big(\beta_{\veps,M}\cdot \nabla_{h}\rho_0^{in}\big)
\right\|_{L^{1}_T(L^2(\bbR^2))} 
\,\leq\,\left\|\rho_0^{in}\right\|_{W^{1,\infty}(\bbR^2)}\,
\left\|\oline{u_{\eps}}_M\right\|_{L^2_T(H^{1}(\bbR^2))}^2\,\leq\,C\,, 
\end{align}
with $ C>0 $ independent of $ M\in\N $ and $\eps\in\,]0,1]$, due to Lemma \ref{Lemma:merc}. Then, for any $M>0$ fixed, up to subsequence in $\veps\in\,]0,1]$, we have
\begin{equation}
\label{weak:conv:Gamma:tilde}
\frac{1-b_M\left(\left|\nabla_h \rho_0^{in}\right|\right)}{\left|\nabla_h\rho_0^{in}\right|}\,\big(\beta_{\veps,M}\cdot \nabla_{h}\rho_0^{in}\big)
\,\cvwstar\, \wtilde{\Gamma}_M \qquad\qquad \mbox{ in }\qquad \mathcal{M}\big([0,T] \times \bbR^2\big)
\end{equation}
as $ \eps $ converges to zero. Let us notice that estimate \eqref{unif:est:hope} is independent of $ M\in\N $. Hence, by choosing a diagonal subsequence,
there exists a subsequence of $ \eps\in\,]0,1] $ (which we do not relabel) such that convergence \eqref{weak:conv:Gamma:tilde} holds \emph{for any} $ M\in\N$.
In addition, using again that \eqref{unif:est:hope} is uniform in $M\in\N$, up to extraction (in $M$) we have
\begin{equation}
	\label{conv:F1}
\forall\,T>0\,,\qquad \wtilde{\Gamma}_{M}\, \cvwstar\, \wtilde{\Gamma}\qquad\qquad \mbox{ in } \qquad \mathcal{M}\big([0,T] \times \bbR^2\big)\,.
\end{equation} 
Notice however that convergences \eqref{weak:conv:Gamma:tilde} and \eqref{conv:F1} are not enough to prove our result. As a matter of fact,
in the integral
\begin{equation*}
\iint_{\R_+\times\bbR^2}
\frac{1-b_M\left(\left|\nabla_h \rho_0^{in}\right|\right)}{\left|\nabla_h\rho_0^{in}\right|^2}\,\big(\beta_{\veps,M}\cdot \nabla_{h}\rho_0^{in}\big)\,
\nabla_{h}\rho_0^{in}\cdot \nabla_h^{\perp}\vphi\,\dd x_h\dt\,,
\end{equation*}
the function 
\begin{equation}
	\label{sing:function}
\frac{\nabla_{h}\rho_0^{in}}{\left|\nabla_{h}\rho_0^{in}\right|}\cdot\nabla_h^{\perp}\vphi\qquad\qquad \mbox{ is \emph{not} continuous,}
\end{equation}
so the weak-$*$ convergence does not apply to pass to the limit. To tackle this difficulty, we now show that the measures 
\[
\frac{1-b_M\left(\left|\nabla_h \rho_0^{in}\right|\right)}{\left|\nabla_h\rho_0^{in}\right|}\,\big(\beta_{\veps,M}\cdot \nabla_{h}\rho_0^{in}\big)
\qquad\qquad \mbox{ and } \qquad\qquad \wtilde{\Gamma}_M 
\]
do \emph{not charge} the set where the function \eqref{sing:function} is not continuous, namely the points of the support of $ \vphi $
such that $\left| \nabla_h \rho_0^{in} \right| = 0  $.
For proving this property, let $T>0$ and $R>0$ such that $\Supp\vphi\subset[0,T]\times B_R $, and let $j\in\N $ be a big parameter. Then,
similarly to \eqref{con:o(1)}, one has
\begin{align*}
&\left\|\frac{1-b_M\left(\left|\nabla_h \rho_0^{in}\right|\right)}{\left|\nabla_h\rho_0^{in}\right|}\,\big(\beta_{\veps,M}\cdot \nabla_{h}\rho_0^{in}\big)
\,b_j\left(\left|\nabla_h \rho_0^{in} \right|\right) \right\|_{L^{1}_T(L^1(B_R))} \\
&\qquad\qquad
\leq\,\left\|\overline{u_{\eps}}_M\right\|_{L^2_T(H^{1}(\bbR^2))}\,\left\|\overline{u_{\eps}}_M\right\|_{L^2_T(L^{4}(\bbR^2))}\,
\left(\mc L \left(\left\{ x \in B_R\;\Big|\quad \left| \nabla_h \rho_0^{in}\right| \leq 2/j \right\}\right) \right)^{1/4} \,,
\end{align*}
which converges to zero as $j\ra+\infty $, uniformly both in $ \eps $ and $ M $. This implies that no concentration occurs where
$\left|\nabla_h \rho_0^{in}\right| = 0 $. If this is expected, owing to the support of the function $1-b_M\left(\left|\nabla_h \rho_0^{in}\right|\right)$,
the important point is that the same property is inherited by the measures $ \wtilde{\Gamma}_M $, due to the lower-semicontinuity of the weak-$*$ convergence.
With this observation and the convergences \eqref{weak:conv:Gamma:tilde} and \eqref{conv:F1} we infer that, up to suitable extractions,
we have
%

\begin{align}
 \label{con:Gamma:M}
&\lim_{M\ra+\infty}\lim_{\veps\ra0^+}\iint_{\R_+\times\R^2}
\frac{1-b_M\left(\left|\nabla_h \rho_0^{in}\right|\right)}{\left|\nabla_h\rho_0^{in}\right|^2}\,\big(\beta_{\veps,M}\cdot \nabla_{h}\rho_0^{in}\big)\,
\nabla_{h}\rho_0^{in}\cdot\nabla_h^\perp\vphi\,\dd x_h\dt \\
&\qquad\qquad\qquad\qquad\qquad\qquad\qquad\qquad\qquad\qquad\qquad\qquad\qquad
=\,\left\lan \wtilde\Gamma,\frac{\nabla_h \rho_0^{in}}{\left|\nabla_h \rho_0^{in}\right|}\cdot\nabla_h^\perp\vphi \right\ran\,. \nonumber
\end{align}

To complete the proof, we have to handle the second term on the right-hand side of \eqref{last:rewrite}. We are going to show that
\begin{equation} \label{lim:remainder}
\limsup_{\veps\ra0^+}\iint_{\R_+\times\bbR^2}
\frac{1-b_M\left(\left|\nabla_h \rho_0^{in}\right|\right)}{\left|\nabla_h\rho_0^{in}\right|^2}\,\big(\beta_{\veps,M}\cdot \nabla_{h}^\perp\rho_0^{in}\big)\,
\nabla_{h}^\perp\rho_0^{in}\cdot \nabla_h^{\perp}\vphi\,\dd x_h\dt\,=\,o(M)\,.
\end{equation}
Notice that, owing to the divergence-free condition on $ \overline{u_{\eps}}_M $, we can write
\[
\curlh\big(\rho_0^{in}\, \overline{u_{\eps,h}}_M\big)\,\overline{u_{\eps,h}}_M^{\perp}\cdot \nabla_{h}^{\perp}\rho_0^{in}\, =\,
\curlh\big(\rho_0^{in}\, \overline{u_{\eps,h}}_M\big)\,\divh\left(\rho_0^{in}\,\overline{u_{\eps,h}}_M\right)\,.
\]
From decomposition \eqref{dec:V:eps:M} in Proposition \ref{prop:27}, we then deduce that
\begin{equation*}
\limsup_{\eps \to 0^+}\left\|\curlh\big(\rho_0^{in}\, \overline{u_{\eps,h}}_M\big)\,\overline{u_{\eps,h}}_M^{\perp}\cdot \nabla_{h}^{\perp}\rho_0^{in}\,-\,
\overline{\eta_{\eps}}_M\,\div\big(\overline{V_{\eps}}_M\big)\right\|_{L^2_T(L^{2}(K))}\,\leq\,C\,2^{-M}
\end{equation*}  
for any compact $ K \subset \bbR^2 $, where we recall that $\overline{\eta_{\eps}}_M\,=\,\curlh\big(\overline{V_{\eps,h}}_M\big) $.
Now, using the wave system written in Proposition \ref{p:wave}, we gather
\begin{equation*}
\oline{\eta_\veps}_M\,\divh\big(\overline{V_{\eps,h}}_M\big)\,=\,\overline{\eta_{\eps}}_M\,\eps\,\d_t \overline{\sigma_{\eps}}_M\, =\,
\big(\oline{\eta_{\eps}}_M -\oline{\sigma_{\eps}}_M\big)\,\eps\,\partial_t\oline{\sigma_{\eps}}_M\,+\,\eps\,\partial_t\left|\overline{\sigma_{\eps}}_M\right|^2/2\,.
\end{equation*}
The second term in the right-hand side of the last equality is obviously small, once we integrate by parts with respect to time. For the first term
appearing therein, instead, we have
\begin{align*}
&\iint_{\R_+\times\bbR^2}\frac{1-b_M\left(\left|\nabla_h \rho_0^{in}\right|\right)}{\left|\nabla_h\rho_0^{in}\right|^2}\,
\big(\overline{\eta_{\eps}}_M -\overline{\sigma_{\eps}}_M\big)\,\eps\,\d_t\oline{\sigma_{\eps}}_M\,\nabla_{h}^{\perp}\rho_0^{in}\cdot\nabla_h^{\perp}\vphi\,
\dd x_h\dt \\
&\qquad= \,-\,\veps\iint_{\R_+\times\bbR^2}\frac{1-b_M\left(\left|\nabla_h \rho_0^{in}\right|\right)}{\left|\nabla_h\rho_0^{in}\right|^2}\,
\d_t\big(\overline{\eta_{\eps}}_M -\overline{\sigma_{\eps}}_M\big)\,\oline{\sigma_{\eps}}_M\,\nabla_{h}\rho_0^{in}\cdot\nabla_h\vphi\,
\dd x_h\dt\\
&\qquad\qquad\qquad\,-\,\veps\iint_{\R_+\times\bbR^2}\frac{1-b_M\left(\left|\nabla_h \rho_0^{in}\right|\right)}{\left|\nabla_h\rho_0^{in}\right|^2}\,
\big(\overline{\eta_{\eps}}_M -\overline{\sigma_{\eps}}_M\big)\,\oline{\sigma_{\eps}}_M\,\nabla_{h}\rho_0^{in}\cdot\nabla_h\d_t\vphi\,
\dd x_h\dt\\
&\qquad\qquad\quad\qquad\qquad \, -\,\veps\int_{\bbR^2}\frac{1-b_M\left(\left|\nabla_h \rho_0^{in}\right|\right)}{\left|\nabla_h\rho_0^{in}\right|^2}\,
\left(\overline{\eta_{\eps}^{in}}_M -\overline{\sigma_{\eps}^{in}}_M\right)\,\oline{\sigma_{\eps}^{in}}_M\,\nabla_{h}\rho_0^{in}\cdot\nabla_h\vphi(0,\cdot)\,
\dd x_h\,.
\end{align*}
By using Proposition \ref{p:wave} again, we deduce that all the terms on the right of the previous equality converge to $0$ when $\veps\ra0^+$,for any fixed $ M $.
Thus, we have finally proved \eqref{lim:remainder}.

In the end, convergence \eqref{conv:Gamma:tilde} is a consequence of Lemma \ref{l:diff_reg} and relations \eqref{con:Gamma:M} and \eqref{lim:remainder}. The proof
of the proposition is completed.
\end{proof}

\subsubsection{The convective term: conclusions}

To resume, we have proved the following result about the convergence of the convective term \eqref{nlt:to:lim}.

\begin{cor}
	\label{cor:conv:Gamma}
There exists a distribution $ \Gamma \in \mathcal{D}'\big(\R_+\times \bbR^2\big) $ such that, up to the extraction of a subsequence,
for any test function $\vphi\in C^\infty_c\big(\R_+\times\R^2\big)$, in the limit $\veps\ra0^+$ one has
\begin{equation*}
\iint_{\R_+\times\bbR^2}\overline{\rho_{\eps} u_{\eps,h} \otimes u_{\eps,h}} : \nabla_h \nabla_h^{\perp}\vphi \, \dx_h \dt\, \longrightarrow
\langle \rho_0^{in}\, \nabla_h \Gamma, \nabla_h^{\perp}\vphi \rangle\,.
\end{equation*}
\end{cor}

\begin{proof} The proof is a consequence of the previous few sections. So let us resume the main steps.

Proposition \ref{last:prob} ensures that
the convergence of \eqref{nlt:to:lim} reduces to the one of \eqref{eq:pass-to-lim}.
At this point, we can use equality \eqref{rew:equ:term:7}, Lemma \ref{l:first-approx} and Proposition \ref{p:Gamma3}. Then, to prove the result it is enough
to set $ \Gamma\,:=\, -\,\big(\Gamma_1 +\Gamma_2 + \Gamma_3\big)$.
\end{proof}

\subsection{Passage to the limit}

In this section, we conclude the proof of Theorem \ref{Main:Theo} by showing how to pass to the limit in  \eqref{weak:form:to:lim},
namely in the equation
\begin{align*}
&\iint_{\R_+\times\R^2}\Big(-\left(\oline{\eta_\veps}-\oline{\s_\veps}\right)\,\d_t\vphi\,-\,
\oline{\rho_\veps u_{\veps,h}\otimes u_{\veps,h}}:\nabla_h\nabla_h^\perp\vphi\,+\,\oline{\o_\veps}\,\Delta_h\vphi\Big)\,\dd x_h\dt \\
\nonumber &\qquad\qquad\qquad\qquad\quad
-\,\frac{\alpha_{\eps}}{\ell_{\eps}}\,\iint_{\R_+\times\R^2}\big(\o_\veps^+\,+\,\o_\veps^-\big)\,\vphi\,\dd x_h\dt\,=\,
\int_{\R^2}\left(\overline{\eta_{\eps}^{in}} - \overline{\sigma_{\eps}^{in}}\right)\,\vphi(0,\cdot)\,\dd x_h\,,
\end{align*}
where $\vphi\in C^\infty_c\big(\R_+\times\R^2\big)$ is a test function.
Recall that, by assumption, $ \alpha_{\eps}/\ell_{\eps}\, \to\, \lambda\geq0 $.

In the case $ \lambda > 0 $, convergence \eqref{Conv:u}, Lemma \ref{lemma:cons:bound:data}, Proposition \ref{conv:eta:sigma} and Corollary \ref{cor:conv:Gamma}
ensure that the above equality converges, as $ \eps\ra0^+ $, to
\begin{align*}
&-\iint_{\R_+\times\R^2}\Big(\curlh\big(\rho_0^{in}\,u\big)-\oline{\s}\Big)\,\d_t\vphi\,\dd x_h\dt\,-\,
\langle \rho_0^{in}\, \nabla_h \Gamma, \nabla_h^{\perp}\vphi \rangle\,+\,\iint_{\R_+\times\R^2}\o\,\Delta_h\vphi\,\dd x_h\dt \\
\nonumber &\qquad\qquad\qquad\qquad
-\,\lambda\,\iint_{\R_+\times\R^2}\big(\o^+\,+\,\o^-\big)\,\vphi\,\dd x_h\dt\,=\,
\int_{\R^2}\Big(\curlh\big(m^{in}_0\big) - r_0^{in}\Big)\,\vphi(0,\cdot)\,\dd x_h\,,
\end{align*}
where $r^{in}_0$ and $m^{in}_0$ have been introduced in Subsection \ref{ss:assumptions} and where  we have set $\o\,:=\,\curlh(u)$
and $\o^\pm\,:=\,\curlh\big(u^\pm\big)$. Then, Theorem \ref{Main:Theo} is proved if we are able to identify $ u $, obtained as the weak limit of
$\big(\overline{u_{\eps,h}}\big)_\veps$, with $ u^{+} $ and $ u^{-} $ defined in Lemma \ref{lemma:cons:bound:data}.

To do that, notice that 
\begin{equation}
	\label{identification}
u - u^{-}\, =\,\big(u - \overline{u_{\eps, h}}\big)\, +\,\big(\overline{u_{\eps, h}} - u_{\eps, h}(\cdot,\cdot,-\ell_{\eps})\big)\, +\,
\big(u_{\eps,h}(\cdot,\cdot, -\ell_{\eps}) - u^{-}\big)\,.
\end{equation} 
The convergence \eqref{Conv:u} and Lemma \ref{lemma:cons:bound:data} together ensure that
\begin{equation*}
	 u - \overline{u_{\eps,h}}\,\cv\, 0 \qquad \text{ and } \qquad u_{\eps,h}(\cdot,\cdot, -\ell_{\eps}) - u^{-}\, \cv\, 0
\qquad\qquad \text{ in } \quad L^2_{\rm loc}\big(\R_+;L^2(\bbR^2)\big)\,.
\end{equation*}
As for the remaining term, we can argue as done in \eqref{eq:u_bar} to get
\begin{align*}
\left\|\overline{u}_{\eps} - u_{\eps}(\cdot,\cdot, -\ell_{\eps})\right\|_{L^2_T(L^2(\bbR^2))}^2\,&\leq\,\, C\,\ell_{\eps}^2\,
\overline{\left\| \partial_3 u_{\eps}(\cdot,\cdot,x_3)\right\|_{L^2_T(L^2(\bbR^2))}^2}\,,
\end{align*}
which of course converges to $0$ when $\veps\ra0^+$, owing to the bounds of Lemma \ref{Lemma:merc}.
We have thus shown that the right-hand side of \eqref{identification} converges weakly to zero in $ L^2\big([0,T]\times \bbR^2\big)$,
for any $T>0$. This implies that $ \overline{u} = u^{-} $. The equality $ \overline{u} = u^{+} $ follows similarly.

\medbreak
To conclude, let us consider the case $ \lambda = 0 $. The only convergence which changes with respect to the case $\lam>0$ is the one of the term
\begin{align*}
\frac{\alpha_{\eps}}{\ell_{\eps}}\,\iint_{\R_+\times\R^2}\big(\o_\veps^+\,+\,\o_\veps^-\big)\,\vphi\,\dd x_h\dt\,=\,
-\,\frac{\alpha_{\eps}}{\ell_{\eps}}\,\iint_{\R_+\times\R^2}\big(u_\veps(\cdot,\cdot,+\ell_\veps)\,+\,u_\veps(\cdot,\cdot,-\ell_\veps)\big)\cdot
\nabla_h^\perp\vphi\,\dd x_h\dt\,.
\end{align*}
We have to show that the above integral converges to zero. However, this is a straightforward consequence of the
uniform estimates of Lemma \ref{Lemma:merc}, which allow us to bound
\begin{align*}
& \left|\frac{\alpha_{\eps}}{\ell_{\eps}}\,\iint_{\R_+\times\R^2}\big(u_\veps(\cdot,\cdot,+\ell_\veps)\,+\,u_\veps(\cdot,\cdot,-\ell_\veps)\big)\cdot
\nabla_h^\perp\vphi\,\dd x_h\dt \right| \\
&\qquad\qquad\qquad \leq\,\frac{\alpha_{\eps}}{\ell_{\eps}}\,
\left\|u_\veps(\cdot,\cdot,+\ell_\veps)\,+\,u_\veps(\cdot,\cdot,-\ell_\veps)\right\|_{L^2_T(L^2(\R^2))}\,\left\|\nabla_h\vphi\right\|_{L^2_T(L^2(\R^2))}\,
\leq\,C\,\sqrt{\frac{\alpha_{\eps}}{\ell_{\eps}}}\,,
\end{align*}
which converges to zero as, by assumption, $\al_\veps/\ell_\veps\,\longrightarrow\,\lam=0$ when $\veps\ra0^+$.

Theorem \ref{Main:Theo} is now completely proved.

\end{document}